\documentclass[a4paper,10pt,leqno]{amsart}

\title[Character theory and Euler characteristics for infinite groups]
{Character theory and Euler characteristic\\for orbispaces and infinite groups}

\author{Wolfgang L\"uck}
\address{Mathematisches Institut, Universit\"at Bonn\\
  Endenicher Allee 60\\
  53115 Bonn, Germany}
\email{wolfgang.lueck@him.uni-bonn.de}

\author{Irakli Patchkoria}
\address{Department of Mathematics\\
  University of Aberdeen\\
  Fraser Noble Building\\
Aberdeen AB24 3UE, UK}
\email{irakli.patchkoria@abdn.ac.uk}

\author{Stefan Schwede}
\address{Mathematisches Institut, Universit\"at Bonn\\
  Endenicher Allee 60\\
  53115 Bonn, Germany}
\email{schwede@math.uni-bonn.de}

\date{\today}

\makeatletter
\@namedef{subjclassname@2020}{%
  \textup{2020} Mathematics Subject Classification}
\makeatother

\usepackage[latin1]{inputenc}
\usepackage[english]{babel}
\usepackage{amsmath,amssymb,amsfonts,amsthm,mathrsfs}
\usepackage[all,cmtip]{xy}
\usepackage[hidelinks]{hyperref}
\xyoption{all}
\usepackage{stmaryrd}
\usepackage{enumerate}

\DeclareMathOperator{\colim}{colim}
\DeclareMathOperator{\hocolim}{hocolim}
\DeclareMathOperator{\cont}{cont}
\DeclareMathOperator{\con}{con}
\DeclareMathOperator{\gl}{gl}
\DeclareMathOperator{\Hom}{Hom}
\DeclareMathOperator{\Mono}{Mono}
\DeclareMathOperator{\Img}{Im}
\DeclareMathOperator{\Orb}{Orb}
\DeclareMathOperator{\Glo}{Glo}
\DeclareMathOperator{\Cl}{Cl}
\DeclareMathOperator{\Out}{Out}

\DeclareMathOperator{\Id}{Id}
\DeclareMathOperator{\op}{op}

\DeclareMathOperator{\stab}{stab}
\DeclareMathOperator{\orb}{orb}
\DeclareMathOperator{\res}{res}
\DeclareMathOperator{\Gal}{Gal}

\newcommand{\mF}{{\mathbb F}}
\newcommand{\mN}{{\mathbb N}}
\newcommand{\mQ}{{\mathbb Q}}
\newcommand{\mZ}{{\mathbb Z}}
\newcommand{\Fin}{{\mathcal F}in}
\newcommand{\Lc}{{\mathcal L}}
\newcommand{\Oc}{{\mathcal O}}
\newcommand{\bT}{{\mathbf T}}
\newcommand{\iso}{\cong}
\newcommand{\tensor}{\otimes}
\newcommand{\xra}{\xrightarrow}
\newcommand{\xla}{\xleftarrow}
\newcommand{\bs}{\backslash}
\newcommand{\dbs}{\backslash\!\!\backslash}
\newcommand{\un}{\underline}
\newcommand{\td}[1]{\langle #1\rangle}
\renewcommand{\to}{\longrightarrow}

\numberwithin{equation}{section}
\newtheorem{theorem}[equation]{Theorem}
\newtheorem{lemma}[equation]{Lemma}
\newtheorem{cor}[equation]{Corollary}
\newtheorem{prop}[equation]{Proposition}
\newtheorem*{theorem*}{Theorem}
 
\theoremstyle{definition}
\newtheorem{defn}[equation]{Definition}
\newtheorem{remark}[equation]{Remark}
\newtheorem{eg}[equation]{Example}
\newtheorem{construction}[equation]{Construction}

\begin{document}

\begin{abstract}
  Given a discrete group $G$ with a finite model for $\underline{E}G$,
  we study $K(n)^*(BG)$ and $E^*(BG)$, where $K(n)$ is the $n$-th Morava $K$-theory
  for a given prime and $E$ is the height $n$ Morava $E$-theory.
  In particular we generalize the character theory of Hopkins, Kuhn and Ravenel
  who studied these objects for finite groups.
  We give a formula for a localization of $E^*(BG)$
  and the $K(n)$-theoretic Euler characteristic of $BG$ in terms of centralizers.
  In certain cases these calculations lead to a full computation of $E^*(BG)$,
  for example when $G$ is a right angled Coxeter group, and for $G=SL_3(\mZ)$.
  We apply our results  to the mapping class group $\Gamma_\frac{p-1}{2}$ for an odd prime $p$
  and to certain arithmetic groups, including the symplectic group $Sp_{p-1}(\mZ)$
  for an odd prime $p$ and $SL_2(\mathcal{O}_K)$ for a totally real field $K$. 
\end{abstract}

\maketitle

\tableofcontents

\section*{Introduction}

The purpose of this paper is to extend the generalized character theory of
Hopkins, Kuhn and Ravenel \cite{HKR} from finite groups to discrete groups $G$
that admit a finite model for the classifying space of proper actions.
We demonstrate the usefulness of our generalization by explicitly calculating
the Morava $K$-theory Euler characteristic and the $E$-cohomology of $B G$
for several interesting examples of such groups that arise
in geometric group theory and arithmetic, including the mapping class group
$\Gamma_{\frac{p-1}{2}}$ for an odd prime $p$, the symplectic group $Sp_{p-1}(\mZ)$
for $p\geq 5$, and $SL_2(\mathcal{O}_K)$ for a totally real field $K$.

To put our results into context, we briefly review some of the highlights of \cite{HKR}.
The Atiyah--Segal completion theorem \cite{AScompletion}
identifies the complex $K$-theory of the classifying space of a finite group $G$
with the completion of the complex representation ring $R(G)$ at its augmentation ideal.
And classical character theory identifies $R(G) \otimes \mathbb{C}$
with the ring $\Cl(G, \mathbb{C})$ of complex valued class functions on $G$.
Hopkins, Kuhn and Ravenel generalize these ideas to higher chromatic heights:
in \cite{HKR}, they study the generalized cohomology ring $E^*(BG)$
for finite groups $G$ and for specific complex oriented cohomology theories $E$ that generalize
$p$-adic complex $K$-theory.

Let $K(n)$ denote the $n$-th Morava $K$-theory at a prime $p$,
with the coefficient ring $K(n)^*=\mathbb{F}_p[v_n^{\pm 1}]$,
where $v_n$ is of degree $-2(p^n-1)$. 
A result of Ravenel \cite{RavBG} shows that for $G$ finite, $K(n)^*(BG)$
is a finitely generated graded vector space over the evenly graded field $K(n)^*$.
Hence it is meaningful to define the $K(n)$-theory Euler characteristic as
\[\chi_{K(n)}(BG) = \dim_{K(n)^*} K(n)^{\text{even}}(BG)-\dim_{K(n)^*} K(n)^{\text{odd}}(BG).\] 
One of the main results of Hopkins, Kuhn and Ravenel computes $\chi_{K(n)}(BG)$
in group-theoretic terms. The group $G$ acts by simultaneous conjugation
on the set $G_{n,p}$ of $n$-tuples of pairwise commuting elements of $p$-power order.
By \cite[Theorem B]{HKR}, the Euler characteristic $\chi_{K(n)}(BG)$ equals
the number of $G$-orbits of $G_{n,p}$. 
A particularly interesting special case is when $K(n)^*(BG)$ is concentrated in even degrees;
then the formula of Hopkins, Kuhn and Ravenel computes the dimension
of $K(n)^*(BG)$ as a graded $K(n)^*$-module.

The importance of this computation becomes clear when trying to understand $E^*(BG)$,
where $E$ is the Morava $E$-theory spectrum at the height $n$ and prime $p$.
Hopkins, Kuhn and Ravenel introduce a ring $L(E^*)$
which is faithfully flat over $p^{-1}E^*$, and hence a rational algebra.
Roughly speaking, $L(E^*)$ can be thought of as a character ring obtained
by adjoining enough $p$-th power roots of unity to $E^*$.
In \cite[Theorem C]{HKR}, Hopkins, Kuhn and Ravenel exhibit an isomorphism
\[L(E^*) \otimes_{E^*} E^*(B G) \cong \Cl_{n,p}(G; L(E^*))\]
to the ring of `generalized class functions', i.e.,
functions from $G_{n,p}$ to $L(E^*)$ that are constant on orbits of the conjugation action by $G$.
These results opened the door to an integral understanding of $E^*(BG)$ for finite groups $G$.
Since $E^*(BG)$ is finitely generated as a graded $E^*$-module
by \cite[Corollary 4.4]{GS}, a result of Strickland \cite{StricklandMorava}
implies that if $K(n)^*(BG)$ is even,
then $E^*(BG)$ is even and free as a graded $E^*$-module. 
Hence the above rational computation of $E^*(BG)$ leads to the integral computation. \smallskip

{\bf Our results.}
We generalize the results of Hopkins, Kuhn and Ravenel
from finite groups to infinite discrete groups with a finite model for $\underline{E}G$,
the classifying $G$-space for proper actions.
Many interesting groups in geometric group theory and arithmetic fall into this class,
and often important moduli spaces are modeled
by the quotient space $\underline{B}G=G \bs \underline{E}G$, see for example \cite{Lsurvey}.
Our main motivation is to study new invariants for such groups,
and to see how much chromatic homotopy theory is reflected in the geometric group theory
and number theory.
Another motivation is the work of the first author on Chern characters
and rational equivariant cohomology theories.
The papers \cite{LCrelle1, LCrelle2, LChern} exhibit formulas for such cohomology theories,
and in particular for the rational $K$-theory, of the classifying space $BG$,
which are very similar to the formulas of Hopkins--Kuhn--Ravenel.
Analogous formulas are also provided by Adem \cite{Adem}
in the case of virtually torsion free groups.
All these formulas generalize the splitting for the rational equivariant $K$-theory which goes back to Artin, tom Dieck and Atiyah--Segal.
The present paper is a common generalization of \cite{HKR}
and \cite{Adem, LCrelle1, LCrelle2, LChern}:
On the one hand we extend the results of \cite{HKR}
from finite groups to certain infinite groups,
and on the other hand we generalize the results of \cite{Adem, LCrelle1, LCrelle2, LChern}
from the chromatic height $1$ to arbitrary chromatic height.

Our first main result, proved as Corollary \ref{cor:euler morava quotients},
generalizes the group-theoretic formula
for the Morava $K$-theory Euler characteristic \cite[Theorem B]{HKR} to infinite groups.\smallskip

{\bf Theorem A.} {\em
Let $G$ be a discrete group that admits a finite $G$-CW-model for $\un{E}G$. Then 

\[\chi_{K(n)}(B G)= \sum_{[g_1,\dots,g_n] \in G \bs G_{n,p} }\chi_\mQ(B C\td{g_1, \dots, g_n}).
\]
}\smallskip

Here $G_{n,p}$ denotes the set of $n$-tuples
of pairwise commuting elements of $G$, each of which has $p$-power order.
The group $G$ acts on $G_{n,p}$ by simultaneous conjugation.
Further, $C\td{g_1, \dots, g_n}$ is the centralizer of the abelian subgroup
$\td{g_1, \dots, g_n}$ generated by $g_1, g_2, \dots, g_n$,
and $\chi_{\mQ}$ denotes the rational homology Euler characteristic. 
Since $G$ admits a finite model for $\underline{E}G$,
the set $G \bs G_{n,p}$ of orbits is finite. 

In Theorem A, a new feature shows up for infinite discrete groups that is not present
in the work of Hopkins--Kuhn--Ravenel. If $G$ is finite, then so are
the centralizers  $C\td{g_1, \dots, g_n}$, so their classifying spaces
are rationally acyclic, and $\chi_\mQ(B C\td{g_1, \dots, g_n})=1$.
Hence for finite groups, the right hand side of the formula in Theorem A
is a fancy way to write the number of $G$-orbits on $G_{n,p}$.
For infinite discrete groups, in contrast,
the centralizers occurring in Theorem A need not be finite nor rationally acyclic.
So the rational Euler characteristic $\chi_\mQ(B C\td{g_1, \dots, g_n})$
need not be $1$, and the right hand side in Theorem A
reflects combinatorial properties of the group $G$.  
In Example \ref{dihedral}, we illustrate this phenomenon
with the amalgamated product $G=D_8 \ast_{C_4} D_8$
of two copies of the dihedral group of order $8$ along the cyclic group of order $4$.
One of the centralizers in this case is isomorphic to $C_4 \times \mathbb{Z}$ and hence not rationally acyclic.

Our second main result, proved as Corollary \ref{cor:Characters for BG},
generalizes the character isomorphism \cite[Theorem C]{HKR} to infinite groups.\smallskip

{\bf Theorem B.} {\em
Let $G$ be a discrete group that admits a finite $G$-CW-model for $\un{E}G$. Then the character map 
  \[ \chi_{n,p}\ : \ L(E^*)\tensor_{E^*} E^*(B G) \ \xra{\iso}\
    \prod_{[g_1,\dots,g_n] \in G \bs G_{n,p}} H^*( B C\td{g_1, \dots, g_n}; L(E^*) ) \]
  is an isomorphism. 
}\smallskip

Again, the situation is more subtle than for finite groups.
If $G$ is finite, then $B C\td{g_1, \dots, g_n}$ is rationally acyclic,
and hence $L(E^*)$-acyclic. So for finite groups, the right hand side
of the character isomorphism is a free $L(E^*)$-module whose rank is the number
of $G$-orbits on $G_{n,p}$, and thus isomorphic to the
generalized class functions $\Cl_{n,p}(G; L(E^*))$.
As already mentioned above,
for infinite discrete groups, the centralizers $C\td{g_1, \dots, g_n}$
need not be finite nor rationally acyclic.
So the factors on the right hand side of the character isomorphism need not
simply be free $L(E^*)$-modules of rank $1$.
In the special case of chromatic height $1$,
the theory $E$ is the $p$-adic $K$-theory and $L(E^*)=\mathbb{Q}_p(\zeta_{p^{\infty}})[u^{\pm 1}]$,
where $u$ is the Bott class.
Hence for $n=1$,  Theorem B recovers a $p$-adic version
of \cite[Theorem 0.1]{LCrelle2} and the second formula in \cite[Theorem 4.2]{Adem}.

Our third main result relates the Morava $K$-theory Euler characteristic of an infinite group
to its orbifold Euler characteristic.
The latter is defined for virtually torsion-free groups $G$
with finite virtual cohomological dimension, as follows:
If $\Gamma$ is any finite index torsion-free subgroup of $G$, then
\[ \chi_{\orb}(BG)=\chi_{\mathbb{Q}}(B\Gamma)/[G:\Gamma] ,\]
compare \cite{Walleuler, SerreVCD};
here $\chi_{\mathbb{Q}}$ is again the rational homology Euler characteristic,
which coincides with the classical Euler characteristic for a finite complex.
This rational number often reflects geometric and arithmetic properties of the group $G$.
A good example is provided by the mapping class group $\Gamma_g^1$
of a closed orientable surface of genus $g \geq 1$ with one marked point.
By a theorem of Harer and Zagier \cite{HZ},
the orbifold Euler characteristic $\chi_{\orb}(B\Gamma_g^1)$ equals the special value $\zeta(1-2g)$
of the Riemann $\zeta$-function. 

In Theorem \ref{thm:chi_K(n)_chi_orb} we establish the following relation:\smallskip

{\bf Theorem C.} {\em
  Let $G$ be a discrete group that admits a finite $G$-CW-model for $\un{E}G$,
  and such that every torsion element in $G$ has $p$-power order. Then 
\[ \chi_{K(n)}(BG) = \sum_{[g_0,g_1,\dots,g_n] \in G \bs G_{n+1,p}} \chi_{\orb}(BC\td{g_0,g_1, \dots, g_n}).
\]
}\smallskip

As we explain in more detail in Remark \ref{rk:chromatic sequence},
these results motivate that for discrete groups with only $p$-power torsion
for a fixed prime $p$, the orbifold Euler characteristic
can be thought as the chromatic height $-1$ invariant.
Indeed, we get a chromatic sequence starting at `height $-1$':
\[\chi_{\orb}, \chi_{\mathbb{Q}}, \chi_{K(1)}, \dots, \chi_{K(n)}, \dots, \]
where the $i$-th term can be obtained from the $j$-th term for $j<i$
using the tuples of $i-j$ many commuting elements of finite order.
\smallskip In Remark \ref{rk:chromatic sequence} we also explain how this result is related to the paper of Yanovski \cite{Yanovski}. \smallskip

{\bf Computations and applications.}
The long final Section \ref{sec:examples} discusses several classes
of infinite groups $G$ with finite model for $\un{E}G$,
and it illustrates our general theory with concrete computations.
We fully compute Morava $K$-theory and $E$-theory for right angled Coxeter groups
and the general linear group $SL_3(\mathbb{Z})$.
Computations for the right angled Coxeter groups generalize the results of \cite{DegL}
to higher chromatic heights. The computations for $SL_3(\mathbb{Z})$
partially recover the results of \cite{TezYag}.
Further, we compute $\chi_{K(n)}(BG)$ and  $L(E^*)\tensor_{E^*} E^*(B G)$
for $G$ a crystallographic group of the form $\mZ^n \rtimes \mZ/p$.
At an odd prime $p$, we also give explicit formulas
for the mapping class group $\Gamma_\frac{p-1}{2}$ of a closed oriented surface
of genus $\frac{p-1}{2}$, for $G=SL_2(O_K)$, where $K$ is a totally real field,
and for $G=GL_{p-1}(\mZ), SL_{p-1}(\mZ), Sp_{p-1}(\mZ)$.
These formulas involve explicit number theoretic invariants such as special values
of zeta functions and class numbers. We list some of them here; the results
are all new for the height $n>1$, and some also for $n=1$:

\begin{itemize}

\item For every totally real field $K$ and odd prime $p$, 
\[\chi_{K(n)}(BSL_2(\Oc_K))= 2\zeta_K(-1)+\sum_{(H)} ({\vert H_{(p)} \vert}^n-\tfrac{2}{\vert H \vert}),\]
where $\zeta_K$ is the Dedekind zeta function of $K$, see Proposition \ref{prop:Euler sl2}.
The sum runs over the conjugacy classes of the maximal finite subgroups (which are all abelian).
The formula generalizes Brown's formula \cite[Section 9.1]{KBro1}
for $\chi_{\mathbb{Q}}(BSL_2(\Oc_K))$ to higher chromatic heights.
The formula was already known for $n=1$ case, see \cite[Example 4.4]{Adem}. 

\item For $p\geq 5$ and $G=GL_{p-1}(\mZ)$,
  we get the identity $\chi_{K(n)}(BGL_{p-1}(\mZ))=\chi_{\mathbb{Q}}(BGL_{p-1}(\mZ))$,
  see Proposition \ref{prop:Euler_KN_equals_rational_GL}.
  The invariant $\chi_{\mathbb{Q}}(BGL_{p-1}(\mZ))$ is known to vanish for $p\geq 13$ \cite[Theorem 0.1 (a)]{Hor}. However, for the height $n$ Morava $E$-theory at the prime $p$, the character theory still provides an interesting formula
\begin{align*}\quad \quad \quad L(E^*)\tensor_{E^*} E^*(BGL_{p-1}(\mZ)) & \cong  \\ H^*(BGL_{p-1}(\mZ); L(E^*)) & \oplus \bigoplus_{ \frac{p^n-1}{p-1} \vert \Cl(\mathbb{Q}(\zeta_p)) \vert} (L(E^*) \oplus L(E^*)[1])^{\otimes_{L(E^*)} \frac{p-3}{2}},\end{align*}
where $\vert \Cl(\mathbb{Q}(\zeta_p)) \vert$ is the class number of the cyclotomic field $\mathbb{Q}(\zeta_p)$. The rational cohomology $H^*(BGL_{p-1}(\mZ); \mathbb{Q})$ is not known in general. However it is known for $p=3, 5,7$, so that for these primes we can compute
$L(E^*)\tensor_{E^*} E^*(BGL_{p-1}(\mZ))$ as a graded $L(E^*)$-module;
see Remark \ref{rk:GL_2(Z) for p=3} for the case $p=3$.
Again the case $n=1$ was already considered in \cite[Example 4.5]{Adem}. 

\item For any odd prime $p$, we establish the formula
  \[\chi_{K(n)}(BSp_{p-1}(\mZ))=
    \chi_{\mathbb{Q}}(BSp_{p-1}(\mZ))+2^{\frac{p-1}{2}} h_p^{-}\cdot \tfrac{p^n-1}{p-1},\]
where $h_p^{-}$ is the relative class number of the cyclotomic field $\mathbb{Q}(\zeta_p)$,
see Proposition \ref{prop:Euler char Sp and Morava E Sp}. 

\item For a prime $p\geq 5$, we derive the formula
\[\chi_{K(n)}(B\Gamma_{\frac{p-1}{2}})=\chi_{\mathbb{Q}}(B\Gamma_{\frac{p-1}{2}})+\tfrac{(p^n-1)(p+1)}{6} ,\]
see Proposition \ref{prop:Euler char Gamma and Morava E Gamma}.
\end{itemize}\smallskip

{\bf Orbispaces.}
At the technical level, a key tool is the systematic use of orbispaces
in the sense of \cite{gepner-henriques}.
We generalize several aspects of the generalized character theory
of Hopkins, Kuhn and Ravenel from the class of finite $G$-CW-complexes for finite groups $G$
to the class of compact orbispaces.
The connection to geometric group theory stems from the fact that
compact orbispaces include the global classifying spaces
of all infinite discrete groups with finite $G$-CW-model for $\un{E}G$.
More generally for any discrete group $G$ and finite proper $G$-CW complex $A$,
the global quotient orbispace $G \dbs A$ is compact.
Similar generalizations were investigated by Lurie in \cite{Elliptic},
though applications to infinite groups have not been explored there.
See the remark below at the end of the introduction. 

The class of compact orbispaces provides a unified approach
to the different kinds of Euler characteristics relevant for this paper.
The relationships between these invariants are conceptually explained by explicit constructions
in the category of orbispaces, such as the formal loop space introduced
in Construction \ref{con:shift}.
However, for the readers who prefer to avoid orbispaces
we provide Remarks \ref{remark: alternative proof 1} and \ref{remark: alternative proof 2},
where alternative ways of proving the above theorems are sketched.
These approaches use equivariant cohomology theories and $G$-CW decompositions along the lines of \cite{LBook, LChern}, and might be more accessible for the readers working
in geometric group theory. \smallskip

{\bf Organization.}
Here is a more detailed outline of the contents of this paper.
In Section \ref{sec: orbi section} we recall generalities on orbispaces.
The global classifying space $B_{\gl}G$ of a discrete group $G$
is introduced in Example \ref{example: global quotient}.
We recall the concept of {\em compact} orbispace and record
some finiteness results in Theorem \ref{thm:compact_finiteness}.
An easy but fundamental fact is that when $G$ admits a finite model for $\un{E}G$,
then its global classifying space $B_{\gl}G$ is a compact orbispace.
In Theorem \ref{thm:basis}\  we calculate the orbispace Burnside ring $A_{\orb}$,
the universal recipient for additive invariants of compact orbispaces;
this calculation is used in Definition \ref{def:orb_chi}
to define the orbispace Euler characteristic $\chi_{\orb}[X]$.
This invariant generalizes the orbifold Euler characteristic of virtually torsion-free
groups with finite $\un{E}G$, by the relation $\chi_{\orb}(BG)=\chi_{\orb}[B_{\gl}G]$.

Section \ref{sec: Morava Euler char general}
studies the Morava $K$-theory Euler characteristic of compact orbispaces and discrete groups,
possibly infinite.
The main result is Theorem \ref{thm:Euler}; for a compact orbispace $X$,
it identifies the Morava $K$-theory Euler characteristic of the underlying space
as the rational Euler characteristic of a certain space $X\td{\mZ_p^n}$
made from the values of $X$ at all finite quotients of $\mZ_p^n$.
This result generalizes the Hopkins--Kuhn--Ravenel formula
$\chi_{K(n)}(BG) = |G\bs G_{n,p}|$ for finite groups $G$,
which is used as an input in the proof.
Inspired by the work of Stapleton \cite{Stap}, we also observe in Theorem \ref{thm:Euler} that
for every compact orbispaces $X$
and $m,n\geq 1$, the $K(m+n)$-Euler characteristic of the underlying space of $X$
agrees with the $K(m)$-Euler characteristic of the space  $X\td{\mZ_p^n}$.
In Corollary \ref{cor:euler morava quotients}
we specialize the theorem to global quotient orbispaces.
When applied to $B_{\gl}G$ for discrete groups $G$ with finite model for $\un{E}G$,
it yields our Theorem A, as well as the identity
\[ \chi_{K(m+n)}(BG)= \sum_{[g_1,\dots,g_n] \in G \bs G_{n,p}} \chi_{K(m)} (BC\td{g_1, \dots, g_n})\]
for all $m,n\geq 1$.

Section \ref{sec: E theory of BG} generalizes the character theory of Hopkins--Kuhn--Ravenel.
In Construction \ref{con: HKR map},
we extend their character map from global quotients of finite group actions
to compact orbispaces, 
and we show in Theorem \ref{thm:generalize HKR} that it is an isomorphism. 
When applied to $B_{\gl}G$ for discrete groups with finite model for $\un{E}G$,
it yields our Theorem B, compare Corollary \ref{cor:Characters for BG}.

In Section \ref{sec: Morava orbi char} we relate the orbifold Euler characteristic 
to the Morava $K$-theoretic Euler characteristic.
The upshot is that $\chi_{\orb}(BG)$ displays features of a chromatic height $-1$ invariant,
see Remark \ref{rk:chromatic sequence}.
For this discussion we make use of the {\em formal loop space}
$\Lc X$ of an orbispace $X$, see Construction \ref{con:shift}.
We show in Theorem \ref{thm:Euler=Euler} that for a compact orbispace,
the formal loop orbispace is again compact,
and the orbispace Euler characteristic of $\Lc X$
equals the rational Euler characteristic of the underlying space of $X$.
We already gave an outline of Section \ref{sec:examples}, which discusses many examples.\smallskip

{\bf Related work.}
Many results in this paper are related to the formulas in \cite{Stap}.
Though the paper \cite{Stap} deals with finite groups, the idea of using the full cohomology of centralizers in this context is already present there. Similar ideas in the context of $\mathbb{F}_p$-cohomology are present in \cite{Lee}, \cite{Henn1}, \cite{Henn} and \cite{JacMc}. We expect that the results in this paper can be generalized to the transchromatic context and many results in \cite{Stap} have versions for infinite discrete groups. 
  
Lurie generalizes Hopkins--Kuhn--Ravenel character theory
to the case of orbi\-spaces in \cite{Elliptic},
see also an exposition in \cite{Raksit}.
The approach in \cite{Elliptic} is slightly different from ours since Lurie uses
global spaces defined on finite abelian groups (which he also calls `orbispaces').
This allows him to reproduce the results of \cite{HKR} and \cite{Stap}
and generalize them to the context of orbispaces.
We are interested in exploring the application of the character theory in the context of infinite groups and geometric group theory. For this purpose orbispaces seem to provide a very convenient language. This is why our orbispaces are directly defined on all finite groups as opposed to finite abelian groups as considered in  \cite{Elliptic}. \smallskip

{\bf Acknowledgments.}
We would like to thank Alejandro Adem, Dieter Degrijse, Fabian Hebestreit,
Alexander Kupers, Sil Linskens, Akhil Mathew, Lennart Meier,
Oscar Randal-Williams, Matthias Wendt and Mingjia Zhang for helpful conversations.
The second author was supported
by the EPSRC grant EP/X038424/1 ``Classifying spaces, proper actions and stable homotopy theory''.
The first and third authors are members of the Hausdorff Center for Mathematics
at the University of Bonn (DFG GZ 2047/1, project ID 390685813).
The second author also would like to thank the Hausdorff Center for Mathematics
for the hospitality during several visits to Bonn.

\section{Orbispaces} \label{sec: orbi section}
In this section we recall some basic facts about orbispaces in the sense of Gepner
and Henriques \cite{gepner-henriques}. 
Important examples of orbispaces are global classifying spaces of discrete groups,
and more generally global quotient spaces associated to $G$-spaces,
see Example \ref{example: global quotient}.
We recall the concept of {\em compact} orbispace and record
some finiteness results in Theorem \ref{thm:compact_finiteness}.
An easy but fundamental fact is that when $G$ admits a finite model for $\un{E}G$,
then its global classifying space $B_{\gl}G$ is a compact orbispace.

Then we introduce the {\em orbispace Burnside ring} $A_{\orb}$, the universal recipient
of an Euler characteristic for compact orbispaces. We also show that the underlying abelian
group of the orbispace Burnside ring is free, with a basis parameterized by isomorphism classes
of finite groups. Altogether, this means that in order to specify
an Euler characteristic with values in some abelian group, we can freely assign values
to the global classifying spaces of finite groups.
In Definition \ref{def:orb_chi} we define the orbispace Euler characteristic $\chi_{\orb}[X]$
for compact orbispaces.
This invariant generalizes the orbifold Euler characteristic of virtually torsion-free
groups with finite $\un{E}G$, by the relation $\chi_{\orb}(BG)=\chi_{\orb}[B_{\gl}G]$.
\medskip

\begin{construction}[The global indexing categories]\label{con:Orb}
  We recall from \cite{gepner-henriques} the topological category $\Glo$ and
  its non-full subcategory $\Orb$. To distinguish the presheaves of spaces on $\Glo$ and $\Orb$
  we adopt the terminology suggested by Rezk \cite{rezk1} and use the terms {\em global spaces}
  and {\em orbispaces}, respectively. The reader should beware, however,
  that global spaces are also called orbispaces by some authors, for example
  in \cite{gepner-henriques, Elliptic}.
  For this paper, orbispaces are the more important
  class, but global spaces show up in the discussion of the formal loop space in
  Section \ref{sec: Morava orbi char}.

  In short, the topological category $\Glo$ is the coherent nerve
  of the $(2,1)$-category of groups, group homomorphisms, and conjugations.
  And $\Orb$ is the coherent nerve of the $(2,1)$-category of groups,
  injective group homomorphisms, and conjugations.
  We expand this definition.
  The objects of $\Glo$ and $\Orb$ are all discrete groups. 
  If $G$ and $K$ are two groups, the morphism space $\Glo(K,G)$
  is the geometric realization of the  translation groupoid $\underline{\Hom}(K,G)$ 
  of the conjugation $G$-action on $\Hom(K,G)$.
  Unpacking this even further, the groupoid
  $\underline{\Hom}(K,G)$ has group homomorphisms $\alpha:K\to G$ as objects,
  and morphisms are pairs $(g,\alpha)$ in $G\times \Hom(K,G)$,
  whose source and target are $\alpha$ and $c_g\circ\alpha$, respectively,
  with $c_g(\gamma)=g\gamma g^{-1}$ the inner automorphism.
  We thus have a homeomorphism
  \[ \Glo(K,G) \cong  EG \times_G  \Hom(K,G),\]
  where $E G$ is the bar construction model for the universal free $G$-space,
  and the right hand side is the homotopy orbit space of the $G$-action, by conjugation,
  on the set of group homomorphisms.
  The space $\Glo(K,G)$ is thus a 1-type, and there is a homotopy equivalence
  \[   \Glo(K,G) \simeq  \coprod_{[\alpha:K\to G]}  B C(\alpha). \]
  Here the disjoint union is over conjugacy classes of homomorphisms from $K$ to $G$;
  and $C(\alpha)$ is the centralizer, in $G$, of the image of $\alpha:K\to G$.
  Composition in $\Glo$ arises as geometric realization of the composition functor
  \begin{equation}\label{eq:composition_Hom}
  \circ \ : \ \underline{\Hom}(K,G)\times\underline{\Hom}(L,K)\to\underline{\Hom}(L,G)  
  \end{equation}
  that is composition of group homomorphisms on objects, and 
  \[ (g,\alpha)\circ (k,\beta) = (g\cdot\alpha(k),\alpha\circ\beta) \]
  on morphisms.
  Then $\Orb$ is the non-full topological subcategory of $\Glo$ spanned by the monomorphisms.
  In other words, $\Orb(K,G)$ is the union of those path components of $\Glo(K,G)$
  that are indexed by injective group homomorphisms. Thus
 \begin{equation} \label{eq:fix_of_Orb}
 \Orb(K,G) \cong  EG \times_G  \Mono(K,G) \simeq  \coprod_{[\alpha:K\to G]}  B C(\alpha), \end{equation} 
  where this time, the disjoint union is over $G$-conjugacy classes of monomorphisms
  from $K$ to $G$.
\end{construction}

We write $\Orb_{\Fin}$ for the full topological subcategory
of $\Orb$ spanned by the finite groups.

\begin{defn} \label{def: orbispace}
  An {\em orbispace} is a contravariant continuous functor from $\Orb_{\Fin}$
  to the category $\bT$ of
  compactly generated spaces. The category of orbispaces is denoted by $orbspc$. 
  A morphism $f:X\to Y$ of orbispaces is an {\em equivalence} if for every finite group $K$,
  the map $f(K):X(K)\to Y(K)$ is a weak homotopy equivalence.
\end{defn}

For this paper, the most important examples of orbispaces
are global quotients of actions of discrete groups.

\begin{eg}[Global quotient orbispaces] \label{example: global quotient}
  We let $G$ be a discrete group, and we let $A$ be a $G$-space.
  The {\em global quotient orbispace} $G\dbs A$  is defined as follows.
  For a finite group $K$, we let $G$ act on the space
  \[  \coprod_{\alpha\in\Mono(K,G)} A^{\Img(\alpha)} \]
  by $g\cdot (a,\alpha)=(g a,c_g\circ\alpha)$.
  Then we define $(G\dbs A)(K)$ as the homotopy orbit space
  \[ (G\dbs A)(K)\ = \ E G \times_G \left( \coprod_{\alpha\in\Mono(K,G)} A^{\Img(\alpha)}\right). \]
  In particular, the underlying space $(G\dbs A)(1)$ is the homotopy orbit space $EG \times_G A$.
  Choosing representatives of the conjugacy classes of monomorphisms $K\to G$
  exhibits the homotopy type as
  \[ (G\dbs A)(K)\ \simeq \
    \coprod_{[\alpha]\in G\bs \Mono(K,G)} E C(\alpha) \times_{C(\alpha)}  A^{\Img(\alpha)}, \]
  where the square brackets indicate conjugacy classes.

  To define the functoriality of $G\dbs A$ we let $L$ be another finite group.
  We reinterpret $(G\dbs A)(K)$ a as homotopy colimit, and derive the $\Orb$-functoriality
  from the functoriality of homotopy colimits.
  We write $\underline{\Mono}(K,G)$ for the translation category
  of the conjugation $G$-action on $\Mono(K,G)$.
  In other words, $\underline{\Mono}(K,G)$ is the full subcategory
  of the translation category $\underline{\Hom}(K,G)$
  discussed in Construction \ref{con:Orb}, with objects the injective homomorphisms.
  The space $\Orb(K,G)$ is then the nerve of the category $\underline{\Mono}(K,G)$.
  The composition functors \eqref{eq:composition_Hom} restrict to monomorphisms,
  and make $\underline{\Mono}$ a sub-2-category of the
  $(2,1)$-category $\underline{\Hom}$.

  The $G$-space $A$ gives rise to a functor
  \[ A\sharp K \ : \ \underline{\Mono}(K,G) \ \to \ \bT\]
  to the category of spaces that sends a monomorphism $\alpha:K\to G$
  to $A^{\Img(\alpha)}$, and a morphism $(g,\alpha)$ of $\underline{\Mono}(K,G)$
  to the translation map $g\cdot:A^{\Img(\alpha)}\to A^{\Img(c_g\circ\alpha)}$.
  The space $(G\dbs A)(K)$ is then precisely the homotopy colimit, in the sense of
  Bousfield--Kan \cite[Chapter XII]{bousfield-kan}, of the functor $A\sharp K$.
  For every object $(\alpha,\beta)$ of
  $\underline{\Mono}(K,G)\times \underline{\Mono}(L,K)$,
  the group $\Img(\alpha\beta)$ is a subgroup of $\Img(\alpha)$,
  so the inclusion is a continuous map
  \[ (A\sharp K)(\alpha)\ = \  A^{\Img(\alpha)}\ \to \
    A^{\Img(\alpha\beta)}\ = \ (A\sharp L)(\alpha\circ\beta)\ .  \]
  For all morphisms $((g,\alpha),(k,\beta))$ of
  $\underline{\Mono}(K,G)\times \underline{\Mono}(L,K)$, the following square commutes:
  \[ \xymatrix@C=15mm{
      A^{\Img(\alpha)}\ar[r]^-{\text{inclusion}} \ar[d]_{g\cdot} &
      A^{\Img(\alpha\beta)}\ar[d]^{g\alpha(k)\cdot} \\
      A^{\Img(c_g \alpha)}\ar[r]_-{\text{inclusion}} &
      A^{\Img(c_g \alpha c_k \beta)}
    } \]
  This means that the inclusions define a natural transformation
  of functors 
  from the left-lower composite to the  upper-right composite
  in the following diagram of functors:
  \[  \xymatrix@C=18mm{\underline{\Mono}(K,G)\times \underline{\Mono}(L,K)
      \ar[d]_{\text{projection}}\ar[r]^-{\text{composition}} &
      \underline{\Mono}(L,G)\ar[d]^{A\sharp L}\\
      \underline{\Mono}(K,G)\ar[r]_-{A\sharp K}
      \ar@{}[ur]^(.3){}="a"^(.7){}="b" \ar@{=>}^{\text{inclusion\qquad }} "a";"b" 
& \bT }\]
  This natural transformation then induces a continuous map of Bousfield--Kan homotopy colimits,
  and thus a continuous map
  \begin{align*}
    (G\dbs A)(K)&\times \Orb(L,K) =\\
                 &\hocolim_{\underline{\Mono}(K,G)\times \underline{\Mono}(L,K) }
                   (A\sharp K)\circ\text{projection}  \\
    \xra{\text{inclusion}_*} &\hocolim_{\underline{\Mono}(K,G)\times \underline{\Mono}(L,K) }
      (A\sharp L)\circ \text{composition} \\
                \xra{\phantom{\text{inclusion}_*}} &\hocolim_{\underline{\Mono}(L,G) }(A\sharp L) = (G\dbs A)(L)
  \end{align*}
%
%
%
%
This defines the $\Orb$-functoriality of $G\dbs A$.
\end{eg}  

As the following proposition shows,
the orbispace $G\dbs A$ only sees the `proper' $G$-homotopy type of $A$,
and it is agnostic about the fixed point spaces for infinite subgroups of $G$.
Said differently, the $G\dbs -$ is fully homotopical
for `$\Fin$-weak equivalences', i.e., weak equivalences on fixed points
of all finite subgroups.

\begin{prop}\label{prop:Fin-invariance}
  Let $G$ be a discrete group.
  Let $f:A\to B$ be a continuous $G$-map between $G$-spaces such that
  for every finite subgroup $H$ of $G$, the map of fixed points $f^H:A^H\to B^H$
  is a weak equivalence. Then the morphism $G\dbs f:G\dbs A\to G\dbs B$ is
  an equivalence of orbispaces.
\end{prop}
\begin{proof}
  We let $K$ be any finite group. Then for every monomorphism
  $\alpha:K\to G$, the map $f^{\Img \alpha}:A^{\Img \alpha}\to B^{\Img\alpha}$
  is a weak equivalence by hypothesis.
  So the $G$-equivariant map
  \[   \coprod_{\alpha\in\Mono(K,G)} f^{\Img(\alpha)} :
    \coprod_{\alpha\in\Mono(K,G)} A^{\Img(\alpha)}\to
    \coprod_{\alpha\in\Mono(K,G)} B^{\Img(\alpha)}\]
  is a weak equivalence of underlying spaces. The homotopy orbit construction
  then takes it to a weak equivalence of spaces.
  So the morphism $G\dbs f$ is a weak equivalence at $K$.
\end{proof}

\begin{eg}[Global classifying spaces]
  The {\em global classifying space} of a discrete group $G$
  is the represented orbispace $B_{\gl}G=\Orb(-,G)$,
  or rather its restriction to the subcategory $\Orb_{\Fin}$.
  By inspection of definitions, this is isomorphic to the global quotient orbispace
  of $G$ acting on a one-point space, i.e.,
  \[ B_{\gl}G \ = \ G\dbs \ast . \]
  By Proposition \ref{prop:Fin-invariance},
  the unique map $\un{E}G\to *$ induces an equivalence of orbispaces
  \[ G\dbs\un{E}G \xra{\simeq} G\dbs \ast = B_{\gl}G.\]
\end{eg}

\begin{eg}
  If $G$ is a torsion-free group, then a monomorphism from a finite group to $G$
  must necessarily have a trivial source.
  So for every $G$-space $A$, the orbispace $G\dbs A$ has the homotopy orbit space $E G\times_G A$
  as its value at the trivial group, and $G\dbs A$ is empty at all non-trivial finite groups.
\end{eg}

\begin{eg}\label{eg:induction_formula}
  We let $H$ be a subgroup of a discrete group $G$, and we let $A$ be an $H$-space.
  We shall now specify an equivalence of orbispaces 
  \[ H\dbs A\ \xra{\simeq}\ G\dbs (G\times_H A). \]
  When $A$ consists of a single point, this specializes to an equivalence from
  $B_{\gl}H$ to $G\dbs (G/H)$.
  For a finite group $K$, the inclusion $\iota:H\to G$ induces a homotopy equivalence
  \[ (H\dbs A)(K)= E H\times_H( \coprod_{\beta\in\Mono(K,H)} A^{\Img(\beta)}) \xra{\simeq}
    E G\times_H( \coprod_{\beta\in\Mono(K,H)} A^{\Img(\beta)}). \]
  Moreover, the $G$-equivariant homeomorphism
  \begin{align*}
    G\times_H( \coprod_{\beta\in\Mono(K,H)} A^{\Img(\beta)})
    &\to \ \coprod_{\alpha\in\Mono(K,G)} (G\times_H A)^{\Img(\alpha)} \\
  [g,(a,\beta)]\qquad  &\longmapsto \qquad ([g,a],c_g\circ\iota\circ\beta)\ ,
  \end{align*}
  induces a homeomorphism of homotopy orbits
  \begin{align*}
    E G\times_H&( \coprod_{\beta\in\Mono(K,H)} A^{\Img(\beta)}) =
       E G\times_G(G\times_H( \coprod_{\beta\in\Mono(K,H)} A^{\Img(\beta)}))\\
    &\iso
    E G\times_G(\coprod_{\alpha\in\Mono(K,G)} (G\times_H A)^{\Img(\alpha)})=
    (G\dbs (G\times_H A))(K).
  \end{align*}
  We omit the verification that for varying $K$, the composite maps
  $(H\dbs A)(K)\to (G\dbs (G\times_H A))(K)$ form a morphism of orbispaces,
  and are thus an equivalence of orbispaces.
\end{eg}

\begin{defn}
  A orbispace is {\em compact} if it belongs to the smallest subcategory
  of the homotopy category of orbispaces that contains the empty orbispace,
  the orbispace $B_{\gl}G$ for all finite groups $G$,
  and that is closed under homotopy pushouts.
\end{defn}

Depending on the reader's background and preferences, the term `homotopy pushout'
can either be interpreted in concrete terms via a double mapping cylinder,
or in abstract terms as a pushout in the $\infty$-category of orbispaces.

\begin{eg}
  Let $G$ be a discrete group and let us denote by $G\bT$ the category of $G$-spaces. The functor
  \[ G\dbs -\ : \ G\bT \ \to \ orbspc \]
  preserves homotopy pushouts, and the orbispace $G\dbs(G/H)$ is equivalent
  to $B_{\gl}H$. 
  So if $A$ is a finite proper $G$-CW-complex, then the global quotient $G\dbs A$
  is a compact orbispace.
  In particular, if $G$ admits a finite $G$-CW-model for $\underline{E}G$,
  then the global classifying space $B_{\gl}G$ is a compact orbispace.
\end{eg}

In part (iii) of the following theorem and in the rest of the paper,
we let $K(n)$ denote the $n$-th Morava $K$-theory spectrum at the prime $p$,
with the coefficient ring $K(n)^*=\mF_p[v_n^{\pm 1}]$, where $v_n$ is of degree $-2(p^n-1)$.
And in part (iv) and below, we let $E$ denote the Morava $E$-theory spectrum, for
an implicit height $n$ and prime $p$, see \cite[Section 7]{goerss-hopkins:moduli}. This is
a Landweber exact theory with the
coefficient ring $E^*= W(\mF_{p^n})\llbracket u_1, \dots, u_{n-1}\rrbracket[u^{\pm 1}]$,
where $u$ is of degree~$-2$.

\begin{theorem}\label{thm:compact_finiteness}
  Let $X$ be a compact orbispace.
  \begin{enumerate}[\em (i)]
  \item  There is a number $k\geq 1$ such that for all finite groups $H$
    of order larger than $k$, the space  $X(H)$ is empty.
  \item For every finite group $K$, the total rational homology of the space $X(H)$
    is finite-dimensional.
  \item For every prime $p$ and $n\geq 1$, the Morava $K$-theory $K(n)^*(X(1))$
    of the underlying space $X(1)$ is finite-dimensional over $K(n)^*$.
  \item For every prime $p$ and $n\geq 1$,
    the Morava $E$-cohomology $E^*(X(1))$ is finitely generated as a graded $E^*$-module.
  \end{enumerate}
\end{theorem}
\begin{proof}
  The overall argument is the same for all four parts: we show that the
  class of compact orbispaces that has the respective property contains
  the empty orbispace, it contains  the orbispace $B_{\gl}G$ for all finite groups $G$,
  and it is closed under homotopy pushouts.
  The pushout property for (i) follows from the fact that colimits in orbispaces
  are computed objectwise. For (ii), (iii) and (iv), it follows
  from the Mayer--Vietoris sequences;
  in the case of (iv) this exploits that the coefficient ring $E^*$ is Noetherian. 
  And the empty orbispace clearly has all four properties.

  It remains to show that all these classes contain $B_{\gl}G$ for any finite $G$.
  If $G$ is a finite group of order $d$, then there are no monomorphisms
  from groups $H$ to $G$ if the order of $H$ is larger than $d$.
  So for such groups $H$, the space $(B_{\gl}G)(H)$ is empty by \eqref{eq:fix_of_Orb}.
  Independent of the order of $H$, the space
  $(B_{\gl}G)(H)$ is always a disjoint union of classifying spaces of finite groups,
  also by \eqref{eq:fix_of_Orb}. So its rational homology is concentrated in dimension 0,
  where it is finite dimensional. Finally, the underlying space of
  $B_{\gl}G$ is a classifying space for the finite group $G$.
  So its Morava $K$-theory is finite dimensional over $K(n)^*$ by Ravenel's theorem
  \cite{RavBG}, and its Morava $E$-cohomology is finitely generated over $E^*$ by
  \cite[Corollary 4.4]{GS}.
\end{proof}

In the following construction we exploit that
the equivalence classes of compact orbispaces form a set.

\begin{defn}
  The {\em orbispace Burnside ring} $A_{\orb}$ is the quotient of the free abelian group
  on the equivalence classes of compact orbispaces by the subgroup generated by the relation
  $[\emptyset]=0$, and the relations
  \begin{equation}\label{eq:pushout_relation}
    [W] \ = \ [X] \ + \ [Y]\ - \ [Z]     
  \end{equation}
  for all compact orbispaces $X$, $Y$, $Z$ and $W$ such that
  $W$ can be written as a homotopy pushout of $X$ and $Y$ along $Z$.
\end{defn}

A function $f$ that assigns to every compact
orbispace $X$ an element $f(X)$ of an abelian group $B$ is an {\em additive invariant}
if it satisfies the following properties:
\begin{enumerate}
\item[(a)] If $X$ is equivalent to $Y$, then $f(X)=f(Y)$.
\item[(b)] The empty orbispace is assigned the zero element of $B$.
\item[(c)] If $W$ is the homotopy pushout of two compact orbispaces $X$ and $Y$
  along a compact orbispace $Z$,  then $f(W)=f(X)+f(Y)-f(Z)$.
\end{enumerate}
By design, for every such additive invariant,
there is a unique group homomorphism $\phi:A_{\orb}\to B$ such that
$f(X)=\phi[X]$ for every compact orbispace $X$.
In this sense, the orbispace Burnside ring is the universal recipient for
additive invariants of compact orbispaces.

\begin{remark}
  For the special case  $Z=\emptyset$ the pushout relation \eqref{eq:pushout_relation}
  becomes the relation $[X]+[Y]=[X\amalg Y]$, i.e., the coproduct of compact orbispaces
  represents the sum in the group structure of $A_{\orb}$. 
  We write $X^\diamond$ for the unreduced suspension of an orbispace $X$,
  i.e., the (homotopy) pushout of the diagram:
  \[       \{0,1\} \ \xla{\text{proj}} \ X\times \{0,1\} \xra{\text{incl}} X\times [0,1] \]
  Because $X\times[0,1]$ is equivalent to $X$,
  relation \eqref{eq:pushout_relation} yields $[X^\diamond]=2{\cdot}[\ast]-[X]$.
  So the inverse of $[X]$ in $A_{\orb}$ is represented
  by $(X\amalg\{0,1\})^\diamond$, the unreduced suspension of $X$ with two disjoint points added.
  Because $A_{\orb}$ is generated by the classes of compact orbispaces,
  and because sums and inverses in $A_{\orb}$ can be represented by geometric constructions
  in the realm of compact orbispaces, we conclude that every element of $A_{\orb}$
  is of the form $[X]$ for some compact orbispace $X$.  
\end{remark}

The following theorem determines the structure of the orbispace
Burnside ring, and it gives a `pointwise' criterion for when two compact
orbispaces represent the same class in $A_{\orb}$.
It is straightforward from the definition that $A_{\orb}$ is generated
by the classes of the global classifying spaces of all finite groups.
The less obvious fact is that the pushout relation \eqref{eq:pushout_relation}
does not introduce any hidden relations between these generators.

Part (ii) of the next theorem exploits that
for a compact orbispace $X$ and a finite group $K$, the rational homology 
of the space $X(K)$ has finite total dimension, compare Theorem \ref{thm:compact_finiteness} (ii);
hence the rational Euler characteristic of $X(K)$ is well-defined.

\begin{theorem}\label{thm:basis}\ 
  \begin{enumerate}[\em (i)]
  \item  The abelian group $A_{\orb}$ is free, and a basis is given by the classes
    $[B_{\gl}G]$ as $G$ runs over a set of representatives
    of the isomorphism classes of finite groups.
  \item If $X$ and $Y$ are two compact orbispaces, then $[X]=[Y]$ in $A_{\orb}$
    if and only for every finite group $K$, the spaces $X(K)$ and $Y(K)$ have the same
    rational Euler characteristic.
  \end{enumerate}
\end{theorem}
\begin{proof}
  We let $F$ denote the subgroup of $A_{\orb}$ generated by the classes $[B_{\gl}G]$
  for all finite groups $G$.
  We consider the class of all orbispaces $X$ such that $[X]\in F$.
  This class contains all global classifying spaces of finite groups by design,
  and it is closed under pushout by the pushout relation \eqref{eq:pushout_relation}.
  So $[X]\in F$ for all compact orbispaces $X$. This shows that $A_{\orb}$
  is generated by the classes of a global classifying spaces of finite groups,
  which establishes half of (i).

  We prove the other half of (i) together with (ii).
  For a finite group $K$, we consider the function that assigns to a compact orbispace $X$
  the rational Euler characteristic $\chi_\mQ(X(K))$ of the space $X(K)$.
  Then $\chi_\mQ(\emptyset)=0$; and if $W$ is a homotopy pushout of $X$ and $Y$ along $Z$,
  then $W(K)$ is a homotopy pushout of $X(K)$ and $Y(K)$ along $Z(K)$,
  so $\chi_\mQ(W(K))=\chi_\mQ(X(K))+\chi_\mQ(Y(K))-\chi_\mQ(Z(K))$.
  So sending $X$ to $\chi_\mQ(X(K))$ is an additive invariant.
  Hence there is a unique group homomorphism
  $\phi_K:A_{\orb}\to \mZ$ such that $\phi_K[X]=\chi_\mQ(X(K))$ for every compact orbispace $X$.
  In particular, if $[X]=[Y]$ in $A_{\orb}$, then
  $\chi_\mQ(X(K))=\chi_\mQ(Y(K))$ for every finite group $K$.

  The number of isomorphism classes of finite groups
  is countable. We choose representatives $\{G_n\}_{n\geq 0}$ in some order
  of non-decreasing cardinality, i.e., $|G_n|\leq |G_{n+1}|$ for all $n\geq 0$.
  The homomorphisms $\phi_{G_n}$ are the components of a homomorphism
  \[ \phi =(\phi_{G_n})_{n\geq 0}\ : \ A_{\orb}\ \to \ {\prod}_\mN\,\mZ \]
  to a countable product of copies of $\mZ$.
  For $m<n$ we have  $|G_m|\leq |G_n|$, and $G_m$ and $G_n$ are not isomorphic.
  So there is no monomorphism from $G_n$ to $G_m$.
  Hence the space $(B_{\gl}G_m)(G_n)$ is empty by \eqref{eq:fix_of_Orb},
  and thus $\phi_{G_n}[B_{\gl}G_m]=\chi_\mQ((B_{\gl}G_m)(G_n))=0$ for $m<n$.
  Also by \eqref{eq:fix_of_Orb}, the space $(B_{\gl}G_n)(G_n)$
  is a disjoint union, indexed by $\Out(G_n)$,
  of classifying spaces of the center of $G_n$. So the rational homology of
  $(B_{\gl}G_n)(G_n)$ is concentrated in dimension 0,
  where its dimension is the order of $\Out(G_n)$.
  Thus $\phi_{G_n}[B_{\gl}G_n]=\chi_\mQ((B_{\gl}G_n)(G_n))=|\Out(G_n)|$.
  We conclude that the $\mN\times\mN$ integer matrix with coefficients
  $\phi_{G_n}[B_{\gl}G_m]$ is upper triangular with nonzero entries
  on the diagonal. Hence the classes $[B_{\gl}G_n]$ for $n\in\mN$
  are linearly independent. Since these classes also generate $A_{\orb}$,
  they form a $\mZ$-basis, and the homomorphism $\phi:A_{\orb}\to\prod_\mN\mZ$
  is injective. This concludes the proof of (i) and (ii).
\end{proof}

\begin{eg}
    Let $G$ be a discrete group and $A$ a finite proper $G$-CW-complex.
  Then the global quotient orbispace $G\dbs A$ is compact,
  and its class in the orbispace Burnside ring satisfies
  \begin{equation} \label{eq:GmodA_in_Aorb}
  [G\dbs A] \ = \ 
    \sum_{n\geq 0} (-1)^n\cdot \sum_{G \sigma} [B_{\gl} (\stab(\sigma)) ] \ \in\
    A_{\orb}.   
  \end{equation}
  Here the inner sum runs over all $G$-orbits of $n$-cells of $A$,
  and $\stab(\sigma)$ is the stabilizer group of the cell $\sigma$.
  Indeed, both sides are additive for disjoint unions and pushouts in $A$,
  so it suffices to check the claim for $A=G/H$ for all finite subgroups $H$ of $G$.
  In this special case, the claim holds because $G\dbs (G/H)$ is equivalent
  to $B_{\gl}H$ as an orbispace, see Example \ref{eg:induction_formula},
\end{eg}

\begin{defn}[Orbispace Euler characteristic]\label{def:orb_chi}
  By Theorem \ref{thm:basis} there is a unique group homomorphism
  \[ \chi_{\orb}\ : \ A_{\orb}\ \to \ \mQ, \]
  such that $\chi_{\orb}[B_{\gl}G]=1/|G|$ for every finite group $G$.
  For a compact orbispace $X$, we refer to $\chi_{\orb}[X]$
  as the {\em orbispace Euler characteristic} of $X$.
\end{defn}

Many examples of compact orbispaces are global quotients $G\dbs A$
for discrete groups $G$ and finite proper $G$-CW-complexes $A$.
We now give formulas for their orbispace Euler characteristics.

\begin{eg}\label{eg:chi of G dbs A}\ 
  \begin{enumerate}[(i)]
\item  Let $G$ be a discrete group and $A$ a finite proper $G$-CW-complex.
  Applying the homomorphism $\chi_{\orb}:A_{\orb}\to\mQ$ to the formula \eqref{eq:GmodA_in_Aorb}
  that expresses the global quotient orbispace $G\dbs A$ in terms of the distinguished
  basis of $A_{\orb}$ yields the equality of rational numbers
  \[ \chi_{\orb}[G\dbs A] \ = \ 
    \sum_{n\geq 0} (-1)^n\cdot \sum_{G \sigma} \frac{1}{|\stab(\sigma)|}.\]
  If $G$ happens to act {\em freely} on $A$, then this becomes
  \[ \chi_{\orb}[G\dbs A] \ = \ 
    \sum_{n\geq 0} (-1)^n\cdot | \text{equivariant $n$-cells}| \ = \ \chi_{\mathbb{Q}}(G\bs A),\]
  the Euler characteristic of the orbit space.
\item
  We let $\Gamma$ be a finite index subgroup of a discrete group $G$,
  and we let $A$ be a finite proper $G$-CW-complex. Then 
  \[ \chi_{\orb}[G\dbs A]\ = \ \frac{1}{[G:\Gamma]}\cdot\chi_{\orb}[\Gamma\dbs A]. \]
  Indeed, both sides are additive for disjoint unions and homotopy pushouts in $A$,
  so it suffices to check the claim for $A=G/H$ for all finite subgroups $H$ of $G$.
  In that case,
  \begin{align*}
    | \Gamma\bs G| \
    &= \ \sum_{\Gamma g H\in \Gamma\bs G/H} |\Gamma\bs (\Gamma g H)| \\
    &= \  \sum_{\Gamma g H\in \Gamma\bs G/H} |(\Gamma^g \cap {H})\bs H| \
    = \  \sum_{\Gamma g H\in \Gamma\bs G/H} |H| / |\Gamma\cap {^g H}|,
  \end{align*}
  and hence
 \begin{align*}
   [G:\Gamma]\cdot \chi_{\orb}&[G\dbs (G/H)]\
    = \ |\Gamma\bs G| / |H|\
    = \ \sum_{\Gamma g H\in \Gamma\bs G/H} \frac{1}{|\Gamma\cap{^g H}|}\\
    &= \    \sum_{\Gamma g H\in \Gamma\bs G/H} \chi_{\orb}[\Gamma\dbs (\Gamma/\Gamma\cap{^g H})]\
    = \ \chi_{\orb}[\Gamma\dbs (G/H)]. 
 \end{align*}
\item
  We let $G$ be a discrete group that is virtually torsion-free,
  and we let $A$ be a finite proper $G$-CW-complex.
  We let $\Gamma$ be any torsion-free subgroup of $G$ of finite index;
  then $A$ is a finite free $\Gamma$-CW-complex, so items
  (i) and (ii) yield the relation
  \[ \chi_{\orb}[G\dbs A]\ = \ \frac{1}{[G:\Gamma]}\cdot\chi_{\orb}[\Gamma\dbs A]
    \ = \ \frac{1}{[G:\Gamma]}\cdot\chi_{\mathbb{Q}}(\Gamma\bs A).\]
  The right hand side is thus independent of the choice of finite index torsion-free subgroup $\Gamma$,
  and usually taken as the definition
  of the {\em orbifold Euler characteristic} of the $G$-space $A$.
  We conclude that in this situation
  our orbispace Euler characteristic specializes to the orbifold Euler characteristic.
\item
  For a {\em finite} group $G$ and every finite $G$-CW-complex $A$,
  we can take $\Gamma=\{1\}$ in the previous item (iii).
  We deduce that the orbispace Euler characteristic of the global quotient orbispace is
  \[\chi_{\orb}[G\dbs A]\ =\  \chi_{\mathbb{Q}}(A)/|G|,  \]
  where $\chi_{\mathbb{Q}}(A)$ is the ordinary Euler characteristic of the underlying non-equivariant space.
\item
  Let $G$ be a countable discrete group that admits a finite $G$-CW-model
  for $\un{E}G$, the universal $G$-space for proper actions. 
  Suppose that $G$ is also virtually torsion-free, and that $\Gamma$
  is any finite index torsion-free subgroup.
  Then item (iii) applies to $A=\un{E}G$, and yields
  \[ \chi_{\orb}[B_{\gl}G]\ = \ \chi_{\orb}[G\dbs \un{E}G]\ = \ 
    \frac{1}{[G:\Gamma]}\cdot\chi_{\mathbb{Q}}(B\Gamma). \]
  Thus the orbispace Euler characteristic of the global classifying space
  $B_{\gl}G= G\dbs\un{E}G$ is the `virtual Euler characteristic' of $G$ as defined by Wall \cite{Walleuler},
  also called the orbifold Euler characteristic of $G$.
  \end{enumerate}
\end{eg}

\begin{remark}[Ring structure on $A_{\orb}$]
  The terminology orbispace Burnside {\em ring} promises a ring structure
  on the abelian group $A_{\orb}$.
  We briefly indicate how this ring structures arises; since we do not need the
  multiplication for this paper, we refrain from giving full details.
  The topological category $\Orb_{\Fin}$
  has a continuous symmetric monoidal structure that is given by product of groups
  on objects; on morphism spaces, the monoidal structure is the
  effect on homotopy orbits by conjugation of the product of group monomorphisms.
  The $\infty$-category obtained as the coherent nerve of $\Orb_{\Fin}$ thus inherits
  a symmetric monoidal structure. 
  We emphasize that this symmetric monoidal structure on $N\Orb_{\Fin}$ is {\em not} cartesian;
  this stems from the fact that
  the two components of a group monomorphism into a product need not be injective.
  
  Still, a symmetric monoidal structure on an $\infty$-category induces a symmetric monoidal
  structure, the {\em Day convolution product}, on the associated $\infty$-category
  of anima-valued presheaves, see e.g. \cite[Section 3 and Corollary 3.25]{LNP}.
  In our situation, this becomes a symmetric monoidal
  structure, the {\em box product} $\boxtimes$, on the $\infty$-category of orbispaces.
  The key property of the box product for our purposes are:
  \begin{itemize}
  \item the box product preserves colimits in each variable;
  \item for finite groups $G$ and $K$, the box product $(B_{\gl} G)\boxtimes(B_{\gl}K)$
    is equivalent to $B_{\gl}( G\times K)$.
  \end{itemize}
  From these properties one can reason that the assignment 
  \[\cdot\ : \  A_{\orb}\times A_{\orb}\ \to \ A_{\orb}, \quad [X]\cdot [Y]\ = \ [X\boxtimes Y]\]
  is well-defined, biadditive, and that it defines a commutative ring structure
  on $A_{\orb}$ with unit object the terminal orbispace $\ast\simeq B_{\gl}\{1\}$.
  In particular, $[B_{\gl}G]\cdot [B_{\gl}K] =  [B_{\gl}(G\times K)]$,
  i.e., the distinguished basis of $A_{\orb}$ is a monomial basis for the ring structure.

  The orbispace Euler characteristic $\chi_{\orb} : A_{\orb}\to  \mQ$
  introduced in Definition \ref{def:orb_chi} is a ring homomorphism.
  It suffices to check the multiplicativity on the generators,
  where it holds by
  \begin{align*}
    \chi_{\orb}[B_{\gl}G]\cdot\chi_{\orb}[B_{\gl}K]\
    &= \ 1/|G|\cdot 1/|K|\ = \  1/|G\times K| \\
    &= \ \chi_{\orb}[B_{\gl}(G\times K)]\ = \ \chi_{\orb}([B_{\gl}G]\cdot[B_{\gl}K]) .  
  \end{align*}
\end{remark}

\section{Orbispaces and Morava \texorpdfstring{$K$}{K}-theory} \label{sec: Morava Euler char general}

We let $K(n)$ denote the $n$-th Morava $K$-theory spectrum at the prime $p$.
In this section we identify the Morava $K$-theory Euler characteristic of the underlying
space of a compact orbispace $Y$ with the rational Euler characteristic of
another space $Y\td{\mZ_p^n}$.
Here $\mZ_p$ is the group of $p$-adic integers.
As we explain in the proof, for $Y=B_{\gl}G$ with $G$ finite,
this specializes to \cite[Theorem B]{HKR}.

\begin{construction}
  For an orbispace $Y$, we define a space $Y\td{\mZ_p^n}$ by
  \[ Y\td{\mZ_p^n}\ = \  \coprod_{N\leq \mZ_p^n\text{ f.i.}}  Y( \mZ_p^n/N),\]
  where the disjoint union is indexed by all finite index subgroups of $\mZ_p^n$.
\end{construction}

Particularly important orbispaces are the global quotients $G\dbs A$,
for a discrete group $A$.
The next theorem identifies the spaces $(G\dbs A)\td{\mZ_p^n}$ more explicitly.
Extending the notation from \cite{HKR},
we let $G_{n,p}$ denote the set of
$n$-tuples $(g_1, \dots, g_n)$ of commuting elements of $G$,
with each $g_i$ having $p$-power order.
The group $G$ acts on $G_{n,p}$ by simultaneous conjugation, and we
write $G\bs G_{n,p}$ for the set of conjugacy classes.

\begin{theorem} \label{thm:eval for orbits}
  Let $G$ be a discrete group and $A$ a $G$-space.
  Then there is a natural weak equivalence 
  \[  \coprod_{[g_1,\dots,g_n] \in G \bs G_{n,p}} E C\td{g_1,\dots,g_n}\times_{C\td{g_1,\dots,g_n}}
    A^{\td{g_1,\dots,g_n}}\ \xra{\ \simeq \ } \ (G\dbs A)\td{\mZ_p^n}. \]
  Here $\td{g_1,\dots,g_n}$ is the subgroup generated by $g_1,\dots,g_n$,
  and $C\td{g_1,\dots,g_n}$ is its centralizer in $G$.
  In particular, there is a weak equivalence
  \[  \coprod_{[g_1,\dots,g_n] \in G \bs G_{n,p}} B C\td{g_1,\dots,g_n}
    \ \xra{\ \simeq \ } \ (B_{\gl }G)\td{\mZ_p^n}. \]
\end{theorem}
\begin{proof}
  Every tuple $(g_1,\dots,g_n)$ of commuting $p$-power order elements
  determines a continuous homomorphism $\mathbb{Z}_p^n\to G$
  by sending the standard topological basis $(e_1, \dots, e_n)$ to the
  tuple $(g_1,\dots,g_n)$.
  This establishes a bijection
  \[  \Hom^{\cont}(\mathbb{Z}_p^n, G) \cong G_{n,p}, \;\;\; \alpha \mapsto (\alpha(e_1), \dots, \alpha(e_n)) \]
  that is equivariant for the conjugation action of $G$ on both sides.
  Also, the image of $\alpha$ is clearly the subgroup generated by
  $\alpha(e_1), \dots, \alpha(e_n)$.
  For the course of the proof it will be convenient to work with
  the $G$-set $\Hom^{\cont}(\mathbb{Z}_p^n, G)$ instead of the $G$-set $G_{n,p}$.

  We let the group $G$ act on the space
  $\coprod_{\alpha\in \Hom^{\cont}(\mZ_p^n, G)}  A^{\Img \alpha}$
  by $g\cdot(a,\alpha)=(g a,c_g\circ\alpha)$. We claim that the homotopy orbit space
  \begin{equation}\label{eq:homotopy_orbits}
 E G\times_G \bigg(  \coprod_{\alpha\in \Hom^{\cont}(\mZ_p^n, G)}  A^{\Img \alpha}\bigg)    
  \end{equation}
  is homeomorphic to $(G\dbs A)\td{\mZ_p^n}$.  
  If $\alpha:\mZ_p^n\to G$ is a continuous homomorphism,
  then $N=\ker(\alpha)$ is a finite index subgroup,
  i.e., it is one of the subgroups that index the
  disjoint union in the definition of $(G\dbs A)\td{\mZ_p^n}$.
  We decompose the disjoint union according to the kernels of the homomorphisms as 
  \[    \coprod_{\tiny{N\leq \mZ_p^n \;\text{f.i.}}}
    \coprod_{\bar\alpha\in\Mono( (\mZ_p^n)/N,G)} A^{\Img\bar\alpha}  ; \]
  in the inner disjoint union, $\bar\alpha$ is the unique homomorphism
  whose composite with the projection $\mZ_p^n\to\mZ_p^n/N$ is $\alpha$.
  Clearly, $\bar\alpha$ and $\alpha$ then have the same image.
  Homotopy orbits commute with disjoint unions.
  So the space \eqref{eq:homotopy_orbits} is homeomorphic to
  \begin{align*}
    \coprod_{\tiny{N\leq \mZ_p^n \;\text{f.i.} }}
    &E G\times_G\bigg(  \coprod_{\bar\alpha\in\Mono( \mZ_p^n/N,G)} A^{\Img\bar\alpha} \bigg) \\
    &= \coprod_{\tiny{N\leq\mZ_p^n \;\text{f.i.}}}
      (G\dbs A)(\mZ_p^n/N)  =  (G\dbs A)\td{\mZ_p^n}.
  \end{align*}
  Now we choose representatives of the $G$-conjugacy classes of continuous homomorphisms
$\alpha:\mZ_p^n\to G$. These provide a $G$-equivariant decomposition
\[ \coprod_{\alpha\in \Hom^{\cont}(\mZ_p^n, G)}  A^{\Img \alpha}\ = \
\coprod_{[\alpha]\in G\bs\Hom^{\cont}(\mZ_p^n, G)}  G\times_{C(\alpha)}A^{\Img \alpha}\ .   \]
The homotopy orbit functor $E G\times_G-$ preserves disjoint unions, and
the spaces $E G\times_G(G\times_{C(\alpha)}A^{\Img \alpha})$ and 
$E C(\alpha)\times_{C(\alpha)} A^{\Img \alpha}$ are weakly equivalent. This proves the theorem.
\end{proof}

For every compact orbispace $Y$, the Morava $K$-theory $K(n)^*(Y(1))$
of the underlying space $Y(1)$ is finite-dimensional over $K(n)^*$,
compare Theorem \ref{thm:compact_finiteness} (iii).
So we can define the {\em Morava $K$-theory Euler characteristic} by
\[ \chi_{K(n)}(Y(1)) \ =  \
  \dim_{K(n)^*}(K(n)^{\text{even}}(Y(1)))\ -\ \dim_{K(n)^*}(K(n)^{\text{odd}}(Y(1))).\]

\begin{theorem}\label{thm:Euler}
  Let $p$ be a prime number, and $m,n\geq 1$, and let $Y$ be a compact orbispace.
  \begin{enumerate}[\em (i)]
  \item The total rational homology of the space $Y\td{\mZ_p^n}$
    is finite dimensional, and $K(m)^*(Y\td{\mZ_p^n})$ is
    finite dimensional over $K(m)^*$.
  \item The following relations between Euler characteristics hold:
    \begin{align*}
      \chi_{K(n)}(Y(1)) \ &= \  \chi_\mQ (Y\td{\mZ_p^n});\\
      \chi_{K(m+n)}(Y(1))\ &=\ \chi_{K(m)}( Y\td{\mZ_p^n}). 
    \end{align*}
  \end{enumerate}\end{theorem}
\begin{proof}
  We start with the special case $Y=B_{\gl}G$ for a finite group $G$.
  In this case, Theorem \ref{thm:eval for orbits}
  identifies the space $(B_{\gl}G)\td{\mZ_p^n}$
  with a finite disjoint union, indexed by the set $G\bs G_{n,p}$,
  of classifying spaces of finite groups;
  so its rational homology is concentrated in dimension 0,
  where its dimension equals the cardinality of the set $G\bs G_{n,p}$.
  This cardinality agrees with the Morava $K$-theory Euler characteristic of $B G$
  by  \cite[Theorem B]{HKR}.
  Similarly, $K(m)^*((B_{\gl}G)\td{\mZ_p^n})$
  is finite dimensional over $K(m)^*$ by Ravenel's theorem \cite{RavBG}.
  Moreover,
  \begin{align*}
    \chi_{K(m)}( (B_{\gl}G)\td{\mZ_p^n})
    &= \sum_{[g_1, \dots, g_n] \in G \bs G_{n,p}} \chi_{K(m)}( B C\td{g_1, \dots, g_n}) \\
    &= \sum_{[g_1, \dots, g_n] \in G \bs G_{n,p}} |C\td{g_1, \dots, g_n}\bs C\td{g_1, \dots, g_n}_{m,p}| \\
    &= |G\bs G_{m+n,p}|
    =\chi_{K(m+n)}( B G) = \chi_{K(m+n)}( (B_{\gl}G)(1)).
  \end{align*}
  The first equality is Theorem \ref{thm:eval for orbits};
  the second and fourth equalities are \cite[Theorem B]{HKR};
  the third equality is the straightforward algebraic fact that
  upon choosing representatives of the $G$-conjugacy classes in $G_{n,p}$,
  the map
    \[\coprod_{[g_1,\dots, g_n] \in G\bs G_{n,p}} C\td{g_1,\dots,g_n} \bs C\td{g_1,\dots,g_n}_{m,p} \to G \bs G_{m+n,p}\]
    sending $[x_1, \dots, x_m] \in C\td{g_1, \dots, g_n} \bs C\td{g_1, \dots, g_n}_{m,p} $
    to $[x_1,\dots,x_m, g_1,\dots,g_n]$ is bijective.
  So the theorem holds for the global classifying spaces of all finite groups.
  
  The theorem clearly also holds for the empty orbispace.
  Since the functors sending $Y$ to $Y(1)$ and $Y\td{\mZ_p^n}$ preserve colimits,
  and by the additivity of $\chi_{K(n)}$ and $\chi_\mQ$ for homotopy pushouts of spaces,
  the class of compact orbispaces $Y$ for which the theorem holds is closed under homotopy pushouts.
  So theorem holds for all compact orbispaces.
\end{proof}

Next we examine the special case $Y=G\dbs A$,
where $G$ is a discrete group and $A$  a finite proper $G$-CW complex.
For a $G$-space $A$, we continue to denote by $A^{\td{g_1, \dots, g_n}}$
the fixed space  with respect to the subgroup generated by the tuple $(g_1, \dots, g_n)$.

\begin{cor} \label{cor:euler morava quotients}
  Let $G$ be a discrete group and let $A$ be a finite proper $G$-CW complex.
  Then for all $m,n \geq 1$,
  \begin{align*}
   \chi_{K(n)}(EG \times_G A)&=\sum_{[g_1,\dots,g_n] \in G \bs G_{n,p} }
    \chi_{\mathbb{Q}}(E C\td{g_1, \dots, g_n}\times_{C\td{g_1, \dots, g_n}} A^{\td{g_1,\dots,g_n}}),\\
 \chi_{K(m+n)}(EG \times_G A)&= \sum_{[g_1,\dots,g_n] \in G \bs G_{n,p}} \chi_{K(m)} (EC\td{g_1, \dots, g_n} \times_{C\td{g_1, \dots, g_n}} A^{\td{g_1, \dots, g_n}}).
  \end{align*}
  In particular, if $G$ has a finite model for $\underline{E}G$, then
  \begin{align*}
    \chi_{K(n)}(B G)\ &= \ \sum_{[g_1,\dots,g_n] \in G \bs G_{n,p} }\chi_\mQ(B C\td{g_1, \dots, g_n}), \\
\chi_{K(m+n)}(BG)&= \sum_{[g_1,\dots,g_n] \in G \bs G_{n,p}} \chi_{K(m)} (BC\td{g_1, \dots, g_n}).
  \end{align*}
\end{cor}
\begin{proof}
  We apply Theorem \ref{thm:Euler} to the compact orbispace $G\dbs A$.
  The underlying space is equivalent to the homotopy orbit space
  \[ (G\dbs A)(1)\ \simeq \ E G\times_G A . \]
  Theorem \ref{thm:eval for orbits} provides a splitting 
  \begin{align*}
  (G \dbs A)\td{\mZ_p^n}\ 
  &\simeq\  \coprod_{[g_1, \dots, g_n]\in G \bs G_{n,p}} E C\td{g_1, \dots, g_n}\times_{C\td{g_1, \dots, g_n}}  A^{\td{g_1, \dots, g_n}}.
\end{align*}
Specializing Theorem \ref{thm:Euler} to the orbispace $G\dbs A$ thus yields the result.
\end{proof}

We recall a standard fact:

\begin{lemma} \label{lemma: finite model for centralizers}
  Let $H$ be a subgroup of a discrete group $G$, with centralizer $C(H)$.
  Let $A$ be a finite proper $G$-CW complex.
  Then the $C(H)$-space $A^H$ is a finite proper $C(H)$-CW complex. 
\end{lemma}
\begin{proof}
  For every $G$-CW-complex $A$, the space $A^H$ is a $C(H)$-CW complex. 
  All points in $A$ have finite stabilizers in $G$,
  so all points in $A^H$ in particular have finite stabilizers in $C(H)$.
  For the finiteness claim about the $C(H)$-CW-structure it suffices to check the case $A=G/K$
  for all finite subgroups $K$ of $G$.
  In this case, we must show that the set $(G/K)^H$ has finitely many $C(H)$-orbits.
  If $g K\in G/K$ is $H$-fixed, then $H^g\leq K$,
  so conjugation by $g$ is a monomorphism $c_g^{-1} : H \to K$, with $c_g^{-1}(h)=g^{-1}h g $.
  The map
  \[ C(H) \bs (G/K)^H \to K\bs \Mono(H,K) , \quad C(H)\cdot (g K)\mapsto K\cdot c_g^{-1} \]
  is injective, where $K$ acts on the set of monomorphisms by conjugation. 
  Since the target is finite, so is the source.
\end{proof}

\begin{remark}
  We let $G$ be a discrete group and $A$ a finite proper $G$-CW-complex. Then the projection
  \[EG \times_G A \to G \bs A\]
   induces an isomorphism in rational homology. In particular, if $G$ has a finite model for $\underline{E}G$, then the rational homology of $BG$ is isomorphic to the rational homology of the quotient $\underline{B}G=G \bs \underline{E}G$ and thus $\chi_{\mathbb{Q}}(BG)=\chi_{\mathbb{Q}}(\underline{B}G)$. 
 More generally, for all $(g_1, \dots, g_n)\in G_{n,p}$, applying
  Lemma \ref{lemma: finite model for centralizers} to the finite subgroup $H=\td{g_1, \dots, g_n}$ generated by $g_1, \dots, g_n$,
  shows that the $C\td{g_1, \dots, g_n}$-space $A^{\td{g_1, \dots, g_n}}$
  is a finite proper $C\td{g_1, \dots, g_n}$-CW complex.
  So the projection 
  \[ E C\td{g_1, \dots, g_n}\times_{C\td{g_1, \dots, g_n}}  A^{\td{g_1, \dots, g_n}}\to
    C\td{g_1, \dots, g_n}\bs A^{\td{g_1, \dots, g_n}}\]
  induces an isomorphism in rational homology.
  In the first equation of Corollary \ref{cor:euler morava quotients},
  the rational Euler characteristic of the homotopy orbit space
  can thus be replaced by the rational Euler characteristic of the strict orbit space
  $C\td{g_1, \dots, g_n}\bs A^{\td{g_1,\dots,g_n}}$. In particular if $G$ has a finite model for $\underline{E}G$, then $(\underline{E}G)^{\td{g_1, \dots, g_n}}$ is a finite model for  $\underline{E}C\td{g_1, \dots, g_n}$ and $\chi_{\mathbb{Q}}(BC\td{g_1, \dots, g_n})=\chi_{\mathbb{Q}}(\underline{B}C\td{g_1, \dots, g_n})$. 
\end{remark}

If the group $G$ is finite, then so are the groups $C\td{g_1, \dots, g_n}$,
and thus the rational Euler characteristic of $B C\td{g_1, \dots, g_n}$ is $1$.
So in this case, Corollary \ref{cor:euler morava quotients} specializes to
\cite[Theorem B (Part 1)]{HKR}, saying that $\chi_{K(n)}(B G)$ is the number of $G$-orbits
in $G_{n,p}$.
Since our proof relies on the results of \cite{HKR}, we are not reproving
the theorem of Hopkins, Kuhn and Ravenel, though.

Corollary \ref{cor:euler morava quotients} is also closely related
to \cite[Theorem 4.2]{Adem} and  \cite[Theorem 0.1]{LCrelle2}.
If we take the case $n=1$, then $K(1)$ is a summand of the mod $p$ complex $K$-theory. Our result then looks as follows:
\[ \chi_{K(1)}(EG \times_G A)=\sum_{[g] \in G \bs G_{1,p} } \chi_{\mathbb{Q}}(C\td{g} \bs A^{(g)})\] 
This agrees with the formula in \cite[Theorem 4.2]{Adem} and exactly picks out
the `$p$-primary' part of  \cite[Theorem 0.1]{LCrelle2}.

\begin{remark} \label{remark: alternative proof 1}
  One can prove Corollary \ref{cor:euler morava quotients} without reference to orbispaces,
  by directly working with finite proper $G$-CW complexes, as follows.
  Given $G$ a discrete group with finite $\underline{E}G$, the functors $K(n)^*(EG \times_G A)$ and 
  \[\bigoplus_{[g_1,\dots,g_n] \in G \bs G_{n,p} } H^*(C\td{g_1,\dots,g_n} \bs A^{\td{g_1,\dots,g_n}}; \mathbb{Q}) \]
  are proper equivariant cohomology theories in $A$;
  so by additivity and the Mayer--Vietoris property,
  it suffices to check the desired identity in the case $A=G/H$
  for finite subgroups $H$ of $G$.
  In that case, the map
  \[\coprod_{[g_1,\dots,g_n] \in G\bs G_{n,p}} C\td{g_1,\dots,g_n} \bs (G/H)^{\td{g_1,\dots,g_n}} \to H\bs H_{n,p} \]
  sending $[x H] \in C\td{g_1,\dots,g_n} \bs (G/H)^{\td{g_1,\dots,g_n}}$ to $[x^{-1}g_1x, \dots, x^{-1}g_n x]$ is bijective, where we have implicitly chosen representatives of the $G$-orbits
  of $G_{n,p}$. Hence
  \begin{align*}
    \chi_{K(n)}(EG \times_G G/H)
    &= \chi_{K(n)}(B H)=  \vert H \bs H_{n,p} \vert \\
    &= \sum_{[g_1,\dots,g_n] \in G \bs G_{n,p}} \vert C\td{g_1,\dots,g_n} \bs (G/H)^{\td{g_1,\dots,g_n}} \vert .
  \end{align*}
The second equality is \cite[Theorem B (Part 1)]{HKR}.
By the proof of Lemma \ref{lemma: finite model for centralizers},
$(G/H)^{\td{g_1,\dots,g_n}}$ has finitely many  $C\td{g_1,\dots,g_n}$-orbits, and hence 
\[
  \vert C\td{g_1,\dots,g_n} \bs (G/H)^{\td{g_1,\dots,g_n}} \vert =
  \chi_{\mathbb{Q}}(C\td{g_1,\dots,g_n} \bs (G/H)^{\td{g_1,\dots,g_n}})  . \]
\end{remark}

The following proposition records a standard fact for Euler characteristics
and follows similarly as the results of \cite[Section 6.6]{LBook} and \cite{KBro1}.
Indeed, both sides of the formula in (i) are additive invariants in the finite proper $G$-CW complex $A$,
and for $A=G/H$, where $H$ is finite the formula holds by \cite[Theorem B (Part 1)]{HKR}.
So (i) holds, and (ii) is a special case of (i). 

\begin{prop} \label{prop:euleralternating}
  Let $G$ be a discrete group.
  \begin{enumerate}[\em (i)]
  \item
    For every finite proper $G$-CW complex $A$ and  $n\geq 0$, 
    \[\chi_{K(n)}(EG \times_G A)=\sum_{G\sigma} (-1)^{n_\sigma} \vert H^{\sigma} \bs H^{\sigma}_{n,p} \vert,\]
    where the sum runs over all $G$-orbits of cells $\sigma$ of $A$, the number $n_\sigma$ is the dimension of $\sigma$ and $H^{\sigma}$ is the stabilizer of $\sigma$. 
  \item
    If $G$ admits a finite model for $\underline{E}G$, then
    \[\chi_{K(n)}(BG)=\sum_{G\sigma} (-1)^{n_\sigma} \vert H^{\sigma} \bs H^{\sigma}_{n,p} \vert,\] 
    where the sum runs over all $G$-orbits of cells $\sigma$ of the finite $G$-CW model for $\underline{E}G$, the number $n_\sigma$ is the dimension of $\sigma$ and $H^{\sigma}$ is the stabilizer of $\sigma$.
  \end{enumerate}
\end{prop}

\begin{eg} \label{dihedral}
  If $G=H \ast_K L$ is an amalgamated product of finite groups, then Bass--Serre theory \cite{Serre}
  provides a one-dimensional finite model for $\underline{E}G$,
  namely a tree with two equivariant $0$-cells with stabilizers $H$ and $L$ and one equivariant $1$-cell with stabilizer $K$, see e.g., \cite[Example 4.10]{Lsurvey}.
  Proposition \ref{prop:euleralternating} (ii) thus yields
 \[ \chi_{K(n)} (BG)= \vert H \bs H_{n,p} \vert + \vert L \bs L_{n,p} \vert - \vert K \bs K_{n,p} \vert .\]
Hence by Corollary \ref{cor:euler morava quotients}, we get a formula
\begin{equation} \label{n=1amalgams} \vert H \bs H_{n,p} \vert + \vert L \bs L_{n,p} \vert - \vert K \bs K_{n,p} \vert =  \sum_{[g_1, \dots, g_n] \in G \bs G_{n,p} } \chi_{\mathbb{Q}}(BC\td{g_1, \dots, g_n}). \end{equation}
Already for $n=1$ this equation has non-trivial group theoretic meaning,
see also \cite[Example 4.1]{Adem}.
The relation between $G \bs G_{1,p}$ on the one hand and $H \bs H_{1,p}$, $L \bs L_{1,p}$
and $K \bs K_{1,p}$ on the other hand is subtle. For example, non-conjugate elements of $K$
might become conjugate in $G$.
This makes centralizers of elements in $G$ very subtle in general, see e.g., \cite[Theorem 4.5]{MKS} and \cite[Theorem 1]{KS2}. 

We write $D_8=C_4 \rtimes C_2$ for the dihedral group of order $8$.
We examine the formula \eqref{n=1amalgams} more closely for the amalgamated product
\[G=D_8 \ast_{C_4} D_8\]
at the prime $2$ and height $n=1$.
The groups $D_8$ and $C_4$ have five and four conjugacy classes of elements, respectively.
So the left hand side of \eqref{n=1amalgams} expands to
\[2 \vert {D_8 \bs (D_8)}_{1,2} \vert - \vert C_4 \bs ({C_4})_{1,2} \vert = 2 \cdot 5 -4=6.\]
In any amalgamated product of finite groups $G=H \ast_K L$,
finite order elements are conjugate to elements in $H$ or $L$,
compare the structure theorem for amalgamated products \cite[Theorem 2]{Serre},
see also \cite[Theorem 4.6]{MKS} and \cite[Lemma 4.3 and Lemma 4.4]{GLO}.
So in our case $G=D_8 \ast_{C_4} D_8$, every finite order element has 2-power order.
The structure theorem also implies that the set $G \bs G_{1,2}$ has cardinality $2+2+3=7$,
where $2$ is the number of conjugacy classes in $D_8$ which do not belong to $C_4$,
and $3$ is the number of conjugacy classes of $C_4$ inside $D_8$.
The difference between $\vert G \bs G_{1,2} \vert =7$ and the left hand side of \eqref{n=1amalgams}
must be accounted for by centralizers.
And indeed, the centralizer of a generator of $C_4$ in $G$ is isomorphic to $C_4 \times \mathbb{Z}$,
the free summand being generated by the product of two reflections on different sides
of the amalgamated product; and all the other centralizers are rationally acyclic.
Because $\chi_{\mathbb{Q}}(B(C_4 \times \mathbb{Z}))=\chi_{\mathbb{Q}}(S^1)=0$, we deduce that
\[ \sum_{[g] \in G \bs G_{1,2} } \chi_{\mathbb{Q}}(BC\td{g})=6+ \chi_{\mathbb{Q}}(B(C_4 \times \mathbb{Z}))=6,\]
which agrees with the left hand side of \eqref{n=1amalgams}. 
\end{eg}

\section{Hopkins--Kuhn--Ravenel character theory for orbispaces} \label{sec: E theory of BG}

The purpose of this section is to establish a generalization of the
character isomorphism of Hopkins, Kuhn and Ravenel \cite[Theorem C]{HKR}
from finite groups to infinite discrete groups.
Even more generally, we will prove in Theorem \ref{thm:generalize HKR} that
a specific character homomorphism \eqref{eq:character}
for compact orbispaces is an isomorphism.
We then specialize it in Corollary \ref{cor:Characters for BG}
to global quotients $G\dbs A$ for finite proper $G$-CW complexes $A$,
and to $B_{\gl}G$ for groups $G$ with finite $\underline{E}G$.

\begin{construction}[Hopkins--Kuhn--Ravenel character map] \label{con: HKR map}
  We let $p$ be a prime number, and $n\geq 1$.
  We let $E$ denote the $n$-th Morava $E$-theory spectrum, sometimes also called the
  Lubin--Tate spectrum, at the prime $p$ for height $n$.
  Hopkins, Kuhn and Ravenel \cite[Section 1.3]{HKR} define a graded ring $L(E^*)$
  as follows. The inverse system of group epimorphisms
  \[ \dots\to (\mZ/p^{k+1})^n  \to  (\mZ/p^k)^n  \to  \dots \to (\mZ/p)^n  \to  \{1\} \]
  induces maps of classifying spaces. Applying $E^*$-cohomology yields a
  direct system of morphisms of graded-commutative $E^*$-algebras
  \[ E^* \to E^*(B(\mZ/p)^n) \to \dots \to  E^*( B(\mZ/p^k)^n) \to E^*(B(\mZ/p^{k+1})^n) \to \dots\    \]
  The colimit of this system is denoted $E^*_{\text{cont}}(B \mZ_p^n)$.
  Then $S$ denotes the subset of $E^2_{\text{cont}}(B \mZ_p^n)$ consisting of the
  first Chern classes of all 1-dimensional complex representations of
  the groups $(\mZ/p^k)^n$, or, equivalently, of all continuous homomorphisms
  $\alpha:\mZ_p^n\to U(1)$. The ring $L(E^*)$ is then defined as the localization
  \[ L(E^*)\ = \ S^{-1} E^*_{\text{cont}}(B\mZ_p^n).    \]
  
  When the orbispace $X$ is compact, we shall now define
  a natural homomorphism of graded $L(E^*)$-modules
  \begin{equation}\label{eq:character}
    \chi_{n,p}^X\ : \ L(E^*)\tensor_{E^*} E^*(X(1)) \ \to\ H^*( X\td{\mZ_p^n};L(E^*)).
  \end{equation}
  Our definition specializes to the Hopkins--Kuhn--Ravenel character map
  $\chi_{n,p}^G$ from \cite[Theorem C]{HKR}
  if we take $X=G \dbs A$ for a finite group $G$ and a finite $G$-CW-complex $A$,
  see Example \ref{eg:HKR special case}.

  Because $X$ is compact, Theorem \ref{thm:compact_finiteness} (i)
  provides an number $k$ such that the values $X(K)$ are empty for all finite groups $K$
  of order larger than $p^k$. In particular, $X(\mZ_p^n/N)$ is empty whenever $N$ is a subgroup
  of $\mZ_p^n$ of index larger than $p^k$.
  If the index of $N$ is less than or equal to $p^k$, then in particular, $(p^k\mZ_p)^n\leq N$.
  Hence $X\td{\mZ_p^n}$ can be rewritten as
  \[ X\td{\mZ_p^n}\  \iso    \coprod_{M\leq (\mZ/p^k)^n}  X( (\mZ/p^k)^n/M).  \]
  For every normal subgroup $M$ of a finite group $G$, the functoriality
  of the orbispace $X$ provides a continuous map
  \[  B G\times  X(G/M) \ \xra{B\pi\times\Id} \ B(G/M) \times  X(G/M) = \Orb(e,G/M)\times X(G/M) \ \to \  X(1).\]
  For varying subgroups of $G=(\mZ/p^k)^n$, we obtain a map
  \[ \coprod_{M\leq (\mZ/p^k)^n} B(\mZ/p^k)^n\times  X((\mZ/p^k)^n/M) \ \to \  X(1).\]
  Using functoriality for this action map and
  the K{\"u}nneth formula for $E^*$-cohomology \cite[Corollary 5.11]{HKR} we obtain a morphism of $E^*$-algebras
  \begin{align*}
    E^*(X(1))\ \to \ & 
                       E^*(\coprod_{M\leq (\mZ/p^k)^n} B(\mZ/p^k)^n\times X((\mZ/p^k)^n/M))\\
              \iso \quad & E^*( B(\mZ/p^k)^n)\tensor_{E^*} E^*( \coprod_{M\leq (\mZ/p^k)^n} X((\mZ/p^k)^n/M)) \\
    \xra{\varphi_k\tensor\Id} \  &L(E^*)\tensor_{E^*} E^*( X\td{\mZ_p^n})\
    \xra{\text{Chern character}} \ H^*( X\td{\mZ_p^n}; L(E^*)).
  \end{align*}
  Here $\varphi_k$ is the composite homomorphism of graded-commutative $E^*$-algebras
  \[ E^*( B(\mZ/p^k)^n)\ \to \ E^*_{\text{cont}}(B \mZ_p^n)\ \to\
    S^{-1} E^*_{\text{cont}}(B \mZ_p^n)\ = \ L(E^*).\]
  The Chern character is defined because $L(E^*)$ is a $(p^{-1}E^*)$-algebra
  which is rational \cite[Theorem C]{HKR}.
  The target of the previous composite is an $L(E^*)$-algebra, so scalar extension from
  $E^*$ to $L(E^*)$ yields the character map \eqref{eq:character}.
  We omit the verification that the character map just defined remains unchanged if we increase
  the number $k$, so that it is independent of the choice of $k$.
\end{construction}

\begin{eg}\label{eg:HKR special case}
  We explain how the  Hopkins--Kuhn--Ravenel character map
  \cite[Theorem C]{HKR} arises as a special case of the character map \eqref{eq:character}.
  We let $G$ be a finite group, and we let $A$ be a finite $G$-CW-complex.
  The underlying space $X(1)=(G\dbs A)(1)$ is the homotopy orbit space $EG \times_G A$;
  so the left hand side of our character map coincides with the source of
  the Hopkins--Kuhn--Ravenel character map.
  By  Theorem \ref{thm:eval for orbits}, we have
  \begin{align*}\label{eq:G A at Z_p^n}
    (G\dbs A)\td{\mZ_p^n}
    &\simeq \   \coprod_{[\alpha]\in G \bs \Hom(\mZ_p^n,G)}  E C(\alpha) \times_{C(\alpha)} A^{\Img(\alpha)} \\ 
    &\simeq EG \times_G (\coprod_{\alpha\in\Hom(\mZ_p^n,G)}  A^{\Img(\alpha)}).
 \end{align*}
Because the ring $L(E^*)$ is a $\mQ$-algebra  and the group $G$ is finite, the functor $H^*(-;L(E^*))$
  turns $G$-homotopy orbits into fixed points and we get
  \begin{align*}
    H^*((G\dbs A)\td{\mZ_p^n}&;L(E^*))\
    \iso \    H^*(  \coprod_{\alpha\in\Hom(\mZ_p^n,G)}  A^{\Img(\alpha)};L(E^*))^G \\
    &\iso \    L(E^*)\tensor_{E^*} E^*( \coprod_{\alpha\in\Hom(\mZ_p^n,G)}  A^{\Img(\alpha)})^G \\
    &= \    L(E^*)\tensor_{E^*} E^*\left(  \text{Fix}_{n,p}(G,A)\right)^G \ = \
      \Cl_{n,p}(G,A;L(E^*)).
  \end{align*}
  The second isomorphism uses that each summand $A^{\Img(\alpha)}$ is a finite
  CW-complex, and that there are only finitely many homomorphisms from $\mZ_p^n$ to $G$.
  The two equalities are definitions from \cite{HKR}.
  We conclude that for $X=G\dbs A$, also the target of our character map 
  \eqref{eq:character} coincides with the target of the Hopkins--Kuhn--Ravenel character map.
  Moreover, comparison of the definitions shows that for $X=G\dbs A$, also our map
  coincides with the one defined by Hopkins, Kuhn and Ravenel.
\end{eg}

\begin{theorem}\label{thm:generalize HKR}
  Let $p$ be a prime number, and $n\geq 0$.
  Then for every compact orbispace $X$, the character map  \eqref{eq:character}
  is an isomorphism.
\end{theorem}
\begin{proof}
  It suffices to show the following statements:
  \begin{enumerate}
  \item[(a)] The class of orbispaces for which 
    the character map  \eqref{eq:character}
  is an isomorphism contains the empty orbispace and is closed under homotopy pushouts.
  \item[(b)]  The class  of orbispaces for which
    the character map  \eqref{eq:character} is an isomorphism
    contains $B_{\gl}G$ for every finite group $G$.
  \end{enumerate}  
  Property (a) holds because source and target of the character map are
  cohomology theories in the orbispace $X$.
  Claim (b) is the special case of \cite[Theorem C]{HKR}
  for $X=\ast$.
\end{proof}

  We let $G$ be a discrete group that admits a finite $G$-CW-model for $\un{E}G$.
  Then the orbispace $X=G\dbs \un{E}G\simeq B_{\gl}G$ is compact,
  and so the character map  \eqref{eq:character} is an isomorphism.
  In this example, the underlying space is a classifying space for the group $G$,
  so the source of the character map specializes to $L(E^*)\tensor_{E^*} E^*(B G)$.
  Theorem \ref{thm:eval for orbits} provides a weak equivalence
  \begin{align*}
    (B_{\gl}G)\td{\mZ_p^n} \
    \simeq \    \coprod_{[g_1,\dots,g_n] \in G \bs G_{n,p}}  B C\td{g_1, \dots, g_n}.  
  \end{align*}
  Because $G$ admits a finite $G$-CW-model for $\un{E}G$,
  there are only finitely many conjugacy classes of $n$-tuples of pairwise commuting
  elements of $p$-power order.
  By Theorem \ref{thm:generalize HKR}, we obtain:

  \begin{cor} \label{cor:Characters for BG}
  Let $G$ be a discrete group that admits a finite $G$-CW-model for $\un{E}G$. Then the character map 
  \[ \chi_{n,p}^{B_{\gl}G}\colon L(E^*)\tensor_{E^*} E^*(B G) \ \xrightarrow{\cong}\
    \prod_{[g_1,\dots,g_n] \in G \bs G_{n,p}} H^*( B C\td{g_1, \dots, g_n}; L(E^*) ) \]
  is an isomorphism. 
\end{cor}

In Corollary \ref{cor:Characters for BG},
the centralizers $C\td{g_1, \dots, g_n}$ need not be finite,
and the factors on right hand side need not be free of rank 1 over $L(E^*)$.
More generally, for a discrete group $G$ and a finite proper $G$-CW complex $A$,
applying Theorem \ref{thm:generalize HKR} to the compact orbispace $G \dbs A$ 
yields a character isomorphism:
 \begin{align*}
   \chi_{n,p}^{G\dbs A}\colon L(E^*) &\otimes_{E^*} E^*(EG \times_G A) \cong \\
   &\prod_{[g_1,\dots,g_n] \in G \bs G_{n,p}} H^*(C\td{g_1, \dots, g_n} \bs A^{\td{g_1, \dots, g_n}}; L(E^*))
 \end{align*}

\begin{remark} \label{remark: alternative proof 2}
  Also Corollary \ref{cor:Characters for BG} can be proved without reference to orbispaces.
  We give a sketch here: For a discrete group $G$ with a finite model for $\un{E}G$, consider the cohomology theory $L(E^*) \otimes_{E^*} E^*(EG \times_G A)$, where $A$ is a finite proper $G$-CW complex. We decompose it using \cite[Theorem 5.5 and Example 5.6]{LChern}: For a finite $G$-CW complex $A$, one has a natural splitting 
\begin{align} \label{alternative proof 3} L(E^*) \otimes_{E^*} E^*(EG \times_G A)\cong \prod_{(H)} \Hom_{\mathbb{Q}[W(H)]}(H_*(C(H) \bs A^H; \mathbb{Q}), T^{E^*}_H),\end{align}
where $(H)$ runs over the conjugacy classes of finite subgroups, $C(H)$ is the centralizer of $H$ in $G$, and $W(H)=N(H)/H \cdot C(H)$, where $N(H)$ is the normalizer. The term $T_H^{E^*}$ is defined as the kernel of the product of restriction maps
\[\res : L(E^*) \otimes_{E^*} E^*(BH) \to \prod_{K \leq H,\;K \neq H} L(E^*) \otimes_{E^*} E^*(BK) . \]
The right hand side of \eqref{alternative proof 3} has an appropriate grading which we do not explain here.
Using \cite[Theorem A and Theorem C]{HKR}, one can see that $T^{E^*}_H$ is trivial unless $H$ is an abelian $p$-group. In the latter case it is isomorphic as a $\mathbb{Q}[W(H)]$-module  to
\[  \bigoplus_{\tiny \begin{aligned}(h_1, \dots, h_n) \in H, \\ \td{h_1, \dots h_n}=H\end{aligned}} L(E^*). \]
Using this, with some more algebraic manipulations one can identify the right hand side of
\eqref{alternative proof 3} with
\[\prod_{[g_1,\dots,g_n] \in G \bs G_{n,p}} H^*(C\td{g_1, \dots, g_n} \bs A^{\td{g_1, \dots, g_n}}; L(E^*)).\]
We do not go into more technicalities here but invite those readers who prefer this approach to work out the details. \end{remark}

\section{Connections to orbifold Euler characteristic} \label{sec: Morava orbi char}

The goal of this section is to relate the orbifold Euler characteristic to the Morava $K$-theory Euler characteristic.
To this end, we introduce the {\em formal loop space} $\Lc X$ of an orbispace $X$
in Construction \ref{con:shift}.
We show in Theorem \ref{thm:Euler=Euler} that the formal loop space of a compact orbispace
is again compact, and the orbispace Euler characteristic of $\Lc X$
equals the rational Euler characteristic of the underlying space of $X$.
In Theorem \ref{thm:chi_K(n)_chi_orb} we apply this general relation
to global classifying spaces of discrete groups with finite model for
$\underline{E}G$, and deduce a formula for the $K(n)$-Euler characteristic
of $BG$ in terms of orbifold Euler characteristics of centralizers of
specific $(n+1)$-tuples of commuting finite order elements.
In Remark \ref{rk:chromatic sequence} we explain how these relationships,
for groups which only have $p$-primary torsion,
make $\chi_{\orb}$ behave like a chromatic height $-1$ invariant,
as opposed to $\chi_{\mQ}$ and $\chi_{K(n)}$, which are chromatic height $0$ and $n$ invariants,
respectively. 

\begin{construction}[Epi-mono factorization]
  Every group homomorphism factors uniquely as the projection to a quotient group,
  followed by a monomorphism. This epi-mono factorization extends to the global indexing
  category, as we recall now.
  We let $N$ be a normal subgroup of a discrete group $K$.
  Then for every group $G$, precomposition with the projection $\pi_N : K\to K/N$
  induces an injective map $\pi^*_N: \Mono(K/N,G)\to\Hom(K,G)$. As $N$ varies over all
  normal subgroups of $K$, these maps form a bijection
  \[\coprod_{N\triangleleft K} \Mono(K/N,G)\xra{\iso}\Hom(K,G) \]
  that is moreover equivariant for the conjugation action of $G$ on both sides.
  So taking $G$-homotopy orbits yields a homeomorphism
  \begin{align}\label{eq:epimono homeo}
    \coprod_{N\triangleleft K} \Orb(K/N,G)&= \coprod_{N\triangleleft K} E G\times_G \Mono(K/N,G)
                                               \nonumber\\
    &\xra{\ \iso\ } E G\times_G \Hom(K,G) = \Glo(K,G). 
  \end{align}
  By abuse of terminology, we shall also refer to this homeomorphism as the
  {\em epi-mono factorization}.
\end{construction}

\begin{construction}[Formal loop space] \label{con:shift}
  We define the {\em formal loop space} $\Lc Y$ of an orbispace $Y$.
  The name is justified by the connection to the formal loop spaces in the
  sense of Lurie  \cite[Construction 3.4.3]{Elliptic}, see Remark \ref{rk:relation2Lurie},
  and because for every finite group $G$, the underlying space of
  $\Lc(B_{\gl} G)$ has the homotopy type of the free loop space of $B G$,
  see Theorem \ref{thm:shift of orbits}.
  We alert the reader that $\Lc Y$ is {\em different} from the pointwise free loop space.
  For example, the formal loop space construction preserves colimits
  (both 1-categorically and $\infty$-categorically), while
  the pointwise free loop space does not.
  In general, the underlying space of $\Lc Y$ is typically not equivalent
  to the free loop space of $Y(1)$.
  
  We define the value of the formal loop space at a finite group $K$ by
  \[ (\Lc Y)(K)=\ \coprod_{\tiny{\begin{aligned}N\triangleleft K\times \mZ \;\text{f.i.} \\ N \cap (K \times 0)=1\end{aligned}}}\ Y( (K\times \mZ) /N),\]  
  The coproduct is indexed by finite index normal subgroups $N$ of $K\times \mZ$
  that intersect $K\times0$ in the trivial group.

  To define the functoriality of $\Lc Y$ in the finite group,
  we fix a finite index normal subgroup $M$ of $G\times \mZ$ with $M\cap(G\times0)=1$.
  Product with the group $\mZ$ and postcomposition with the projection
  $\pi_M:G\times\mZ\to (G\times\mZ)/M$ pass to continuous maps on morphism spaces.
  Together with the epi-mono factorization \eqref{eq:epimono homeo}, this yields a composite:
  \begin{align}\label{eq:timesZfactor}
    \Orb(K, G) &\xra{-\times\mZ} \Orb(K\times\mZ, G\times\mZ)\nonumber\\
    \xra{(\pi_M)_*} &\Glo(K\times\mZ, (G\times\mZ)/M)
                      \iso \coprod_{N\triangleleft K\times \mZ}\Orb((K\times\mZ)/N,(G\times\mZ)/M)
  \end{align}
  Because  $M\cap(G\times0)=1$,
  the image of this composite is contained in the union of those
  summands indexed by normal subgroups $N$ of $K\times\mZ$
  that satisfy $N\cap (K\times0)=1$.
  Moreover, because $M$ has finite index in $G\times \mZ$,
  the maps lands in those summands such that $N$ has finite index in $K\times\mZ$.
  We can thus define the partial functoriality on the summand indexed by $M$ as the composite
  \begin{align*}
    Y( (G\times \mZ) /M)
    &\times \Orb(K, G)\xra{\eqref{eq:timesZfactor}}\\
    &\coprod_{\tiny{\begin{aligned}N\triangleleft  K\times \mZ \;\text{f.i.} \\ N \cap (K \times 0)=1\end{aligned}}}\ Y( (G\times \mZ) /M)\times \Orb((K\times\mZ)/N,(G\times\mZ)/M)\\
    &\xra{\ \circ\ } \coprod_{\tiny{\begin{aligned}N\triangleleft  K\times \mZ \;\text{f.i.} \\ N \cap (K \times 0)=1\end{aligned}}}\ Y( (K\times \mZ) /N)
      =(\Lc Y)(K).
  \end{align*}
  As $M$ runs over all those subgroups $M$ of $G\times \mZ$ that index the sum
  defining $(\Lc Y)(G)$, these maps assemble into the desired functoriality 
  \[  (\Lc Y)(G) \times \Orb(K, G) \to  (\Lc Y)(K). \]
  The fact that these maps are also associative boils down to associativity
  of the epi-mono factorization \eqref{eq:epimono homeo}.
\end{construction}
  
\begin{remark}\label{rk:relation2Lurie}
  The formal loop orbispace $\Lc Y$ is an adaptation of 
  the formal loop space construction defined by Lurie in \cite[Construction 3.4.3]{Elliptic}
  to our context, the differences being that
  Lurie works in the $\infty$-category of global spaces indexed on
  finite abelian groups, whereas we work in a 1-categorical model
  for the $\infty$-category of orbispaces indexed on all finite groups.
  By definition, the category $\Orb_{\Fin}$ is a wide but non-full subcategory
  of $\Glo_{\Fin}$, the full subcategory of $\Glo$ spanned by the finite groups,
  compare Construction \ref{con:Orb}.
  So restriction of functors from $\Glo_{\Fin}$ to $\Orb_{\Fin}$
  and its left adjoint, continuous left Kan extension,
  form an adjoint functor pair
  \[\xymatrix@C=15mm{  orbspc = {\bT}_{\Orb} \ar[r]<0.5ex>^-{(-)_{\Glo}} & {\bT}_{\Glo} \ar[l]<0.5ex>^-{U}}. \]
  Here ${\bT}_{\Glo}$ denotes the $1$-category of global spaces,
  i.e., continuous functors from $\Glo_{\Fin}^{\op}$ to spaces.
  Since the right adjoint $U$ forgets the functoriality in non-injective group homomorphisms,
  the left adjoint $(-)_{\Glo}$ can be thought of as `freely adding inflations',
  i.e., restriction along surjective group homomorphisms.
  For more details we recommend Rezk's preprint \cite{rezk1}.
  
  For any global space $Z$ and torsion abelian group $\Lambda$,
  Lurie defines the formal loop space $L^{\Lambda} Z$
  \cite[Construction 3.4.3]{Elliptic}.
  In the special case $\Lambda=\mQ/\mZ$, 
  the construction can be rewritten as
  \[ (L^{\mQ/\mZ} Z)(G)=\colim_{n\in\mN_{>0}} Z(G \times \mZ/n),\]
  where the colimit is over the poset, under divisibility, of positive natural numbers,
  by inflation along epimorphisms $\mZ/k n\to \mZ/n$.
  The connection to the formal loop space of Construction \ref{con:shift} 
  is that the formal loop space $L^{\mQ/\mZ}(Y_{\Glo})$ of the globalization
  of an orbispace $Y$
  is naturally equivalent to the globalization  of the formal loop orbispace $\Lc Y$.
  \end{remark}

Next, we give an explicit formula for the formal loop space
of a global quotient orbispace.
This is a central technical result which will allow
us to compute various Euler characteristics in terms of centralizers.
We write $G_f$ for the set of finite order elements in $G$;
the group $G$ acts on $G_f$ by conjugation.
We let $C\td{g}$ denote the centralizer of an element $g\in G$.
We recall the folklore fact that the free loop space of the classifying space
of a discrete group is equivalent to the disjoint union, over conjugacy classes of elements,
of the classifying spaces of the centralizers.
So the following theorem in particular shows that for finite groups $G$,
the underlying space of $\Lc(B_{\gl} G)$
has the homotopy type of the free loop space of $B G$, whence the name.

\begin{theorem} \label{thm:shift of orbits}
  Let $G$ be a discrete group and $A$ a $G$-space.
  Then there is a natural equivalence of orbispaces
  \[  \coprod_{[g] \in G \bs G_f} C\td{g}\dbs A^{\td{g}}\ \xra{\ \simeq \ } \ \Lc(G\dbs A). \]
  In particular, there is an equivalence of orbispaces
  \[  \coprod_{[g] \in G \bs G_f} B_{\gl }C\td{g}\ \xra{\ \simeq \ } \ \Lc(B_{\gl}G). \]
\end{theorem}
\begin{proof}
  The group $G$ acts on the space $\coprod_{g\in G_f}  A^{\td{g}}$
  by $\gamma\cdot(g,a)=(\gamma g\gamma^{-1},\gamma a)$.
  We claim that  $\Lc(G\dbs A)$ is isomorphic to the global quotient orbispace
  \[G\dbs\bigg(  \coprod_{g\in G_f}  A^{\td{g}}\bigg). \]
  To prove that, we evaluate at a finite group $K$.
  By definition, the value of the above global quotient
  at $K$ is the homotopy orbit space of the $G$-space
  \[ \coprod_{\beta\in \Mono(K,G)} \bigg( \coprod_{g\in G_f}  A^{\td{g}}\bigg)^{\Img(\beta)} .\]
  For a given monomorphism $\beta:K\to G$, an element $(g,a)$ of the inner disjoint
  union is fixed by $\Img(\beta)$ if and only if
  the relation $(\beta(k) g\beta(k)^{-1},\beta(k)a)=(g,a)$
  holds for all $k\in K$. This is equivalent to the conditions that $\Img(\beta)$ centralizes $g$, and that $a$ is also fixed by $\Img(\beta)$.
  Such pairs combine into homomorphisms 
  \[  K\times \mZ\to G\ ,\quad (k,m)\mapsto \beta(k)\cdot g^m\]
  from the product to $G$.
  So the previous $G$-space equals the $G$-space
  \[ \coprod_{\tiny{ \begin{aligned}&\gamma\in \Hom^{\text{f.img.}}(K\times \mZ,G) \\ &\quad \quad \quad \gamma|_K\;\;\text{monic} \end{aligned}}}  A^{\Img \gamma}; \]
  the disjoint union is taken over homomorphisms
  $\gamma:K\times \mZ\to G$ with finite image, and whose restriction
  to $K \times 0$ is injective.
  Further, if $\gamma$ is such a homomorphism,
  then $N=\ker(\gamma)$ is a finite index normal subgroup whose intersection
  with $K \times 0$ is trivial, i.e., it is one of the subgroups that index the
  disjoint union in the definition of $\Lc(G\dbs A)(K)$.
  We decompose the disjoint union according to the kernels of the homomorphism as 
  \[    \coprod_{\tiny{\begin{aligned}N\triangleleft K\times \mZ \;\text{f.i.} \\ N \cap (K \times 0)=1\end{aligned}}}
  \coprod_{\alpha\in\Mono( (K\times \mZ)/N,G)} A^{\Img\alpha}  ; \]
in the inner disjoint union, $\alpha$ is the unique homomorphism
whose composite with the projection $K\times \mZ\to(K\times \mZ)/N$ is $\gamma$.
Homotopy orbits commute with disjoint unions.
So the value of the orbispace $G\dbs\big(  \coprod_{g\in G_f}  A^{\td{g}}\big)$ at $K$ is
\begin{align*}
      \coprod_{\tiny{\begin{aligned}N\triangleleft K\times \mZ \;\text{f.i.} \\ N \cap (K \times 0)=1\end{aligned}}}
  &E G\times_G\bigg(  \coprod_{\alpha\in\Mono( (K\times \mZ)/N,G)} A^{\Img\alpha} \bigg) \\
  &= \coprod_{\tiny{\begin{aligned}N\triangleleft K\times \mZ \;\text{f.i.} \\ N \cap (K \times 0)=1\end{aligned}}}
    (G\dbs A)((K\times \mZ)/N)
  = \Lc(G\dbs A)(K).
\end{align*}
We omit the verification that the pointwise homeomorphisms are natural in the group $K$.
Now we choose representatives of the $G$-conjugacy classes of finite order elements of $G$.
These provide a $G$-equivariant decomposition
\[ \coprod_{g\in G_f}  A^{\td{g}}\ = \
\coprod_{[g]\in G\bs G_f}  G\times_{C\td{g}}A^{\td{g}}.   \]
The global quotient functor $G\dbs-$ preserves disjoint unions,
and the orbispaces $G\dbs(G\times_{C\td{g}}A^{\td{g}})$ and 
$C\td{g}\dbs A^{\td{g}}$ are equivalent, see Example \ref{eg:induction_formula}.
This proves the theorem.
\end{proof}

\begin{theorem}\label{thm:Euler=Euler}
  For every compact orbispace $X$,
  the formal loop orbispace $\Lc X$ is compact, and the relation
  \[\  \chi_{\orb}[ \Lc X ] \ =\ \chi_\mQ(X(1)) \]
 between Euler characteristics holds.
\end{theorem}
\begin{proof}
  Theorem \ref{thm:shift of orbits} shows that  for every finite group $G$,
  the orbispace $\Lc(B_{\gl}G)$ is compact and
  \[ \chi_{\orb}[\Lc(B_{\gl}G)]  = 
    \sum_{[g]\in G\bs G_f} \chi_{\orb}[B_{\gl} C\td{g}] =  
    \sum_{[g]\in G\bs G_f} \frac{1}{|C_G(g)|}  =  1 =  \chi_{\mQ}(B G).\]
  So the theorem holds for global classifying spaces of finite groups.
  The functor $\Lc$ preserves colimits, so $\Lc X$ is compact whenever $X$ is.
  Both sides of the desired equation are additive invariants in $X$,
  so the formula holds in general.
\end{proof}

\begin{eg} \label{eg:Brown theorem}
  Theorem \ref{thm:Euler=Euler} generalizes a result of Brown \cite[Theorem 6.2]{KBro2}.
  We let $G$ be a discrete group that admits a finite $G$-CW-model for $\un{E}G$.
  Then by Lemma \ref{lemma: finite model for centralizers},
  for every finite order element $g\in G$, the space
  $(\un{E} G)^{\td{g}}$ is a finite $C\td{g}$-CW-model for
  the universal space for proper actions of the centralizer $C\td{g}$;
  in particular, $B_{\gl} G=G\dbs \un{E}G$ and $B_{\gl}C\td{g}$ are compact orbispaces.
  Using Theorem \ref{thm:Euler=Euler} and Theorem \ref{thm:shift of orbits}, we conclude that
  \[ \chi_\mQ(B G)\ = \ 
    \chi_{\orb}[\Lc(B_{\gl}G)]
    \ = \  \sum_{[g]\in G\bs G_f} \chi_{\orb}[B_{\gl} C\td{g}]. \]
  This recovers the above mentioned result of Brown.
\end{eg}

In the next lemma and afterwards, we shall use a new piece of notation.
Let $G$ be a discrete group, let $p$ be a prime, and let $n \geq 1$.
We write $G_{n,p,+1}$ for the set of $(n+1)$-tuples $(x, g_1,\dots, g_n)$ of commuting elements,
with $g_1,\dots,g_n$ of $p$-power order, and $x$ a general finite order element. 
Such tuples are in bijective correspondence with group homomorphism
$\mZ\times \mZ_p^n\to G$ with finite image.

\begin{theorem}\label{thm:chi_K(n)_chi_orb}
  Let $G$ be a discrete group, and let $A$ be a finite proper $G$-CW complex.
  Then for every $n \geq 1$,
  \[\chi_{K(n)}(EG \times_G A)=\sum_{[x, g_1,\dots,g_n] \in G \bs G_{n,p,+1}} \chi_{\orb}[C\td{x,g_1,\dots,g_n}\dbs A^{\td{x,g_1,\dots,g_n}}].\]
  In particular, if $G$ has a finite model for $\underline{E}G$, then
  \[ \chi_{K(n)}(BG) = \sum_{[x, g_1,\dots,g_n] \in G \bs G_{n,p,+1}} \chi_{\orb}[B_{\gl}C\td{x,g_1,\dots,g_n}]. \]
\end{theorem}
\begin{proof}
  For each tuple $(g_1,\dots,g_n)$ in $G_{n,p}$,
  Theorem \ref{thm:shift of orbits} provides a decomposition
  \begin{align*}
   \Lc (C\td{g_1, \dots, g_n} \dbs A^{\td{g_1, \dots, g_n}})
    &\simeq \coprod_{[x] \in C\td{g_1, \dots, g_n}^{\con}_f } C_{C\td{g_1, \dots, g_n}}\td{x}\dbs(A^{\td{g_1, \dots, g_n}})^{\td{x}}\\
    &= \coprod_{[x] \in C\td{g_1, \dots, g_n}^{\con}_f } C\td{x,g_1,\dots,g_n}\dbs A^{\td{x,g_1,\dots,g_n}},
  \end{align*}
  where  $C\td{g_1, \dots, g_n}^{\con}_f$ denotes the set $C\td{g_1, \dots, g_n} \bs C\td{g_1, \dots, g_n}_f$ of conjugacy classes of finite order elements in $C\td{g_1, \dots, g_n}$ for brevity. 
  Thus
  \begin{align*}
    \chi_{K(n)}(EG \times_G A)
    &=\sum_{[g_1,\dots,g_n] \in G \bs G_{n,p} }
      \chi_{\mathbb{Q}}(E C\td{g_1, \dots, g_n}\times_{C\td{g_1,\dots,g_n}} A^{\td{g_1,\dots,g_n}})\\
    &=\sum_{[g_1,\dots,g_n] \in G \bs G_{n,p} }
      \chi_{\orb}[ \Lc(C\td{g_1, \dots, g_n}\dbs A^{\td{g_1,\dots,g_n}})]\\
    &=\sum_{[g_1,\dots,g_n] \in G \bs G_{n,p} }
      \sum_{[x] \in C\td{g_1,\dots,g_n}^{\con}_f} \chi_{\orb}[C\td{x,g_1,\dots,g_n}\dbs A^{\td{x,g_1,\dots,g_n}}]\\
    &=\sum_{[x, g_1,\dots,g_n] \in G \bs G_{n,p,+1}} \chi_{\orb}[C\td{x,g_1,\dots,g_n}\dbs A^{\td{x,g_1,\dots,g_n}}].
  \end{align*}
  The first equation is Corollary \ref{cor:euler morava quotients};
  the second equation is Theorem \ref{thm:Euler=Euler} for the orbispaces
  $C\td{g_1, \dots, g_n}\dbs A^{\td{g_1,\dots,g_n}}$;
  the final equation follows from the tautological decomposition, namely that the map
  \[\coprod_{(g_1,\dots, g_n) \in G_{n,p}} C\td{g_1,\dots,g_n}_f\to  G_{n,p,+1}\]
  sending
  $x \in C\td{g_1, \dots, g_n}_{f}$ to $(x, g_1,\dots,g_n)$ is bijective.
  Here we recall that $C\td{g_1, \dots, g_n}_f$ denotes the set of finite order elements in the centralizer
  of $\td{g_1,\dots,g_n}$.
  The conjugation action of $G$ on $G_{n,p,+1}$ permutes the summands on the
  left according to the conjugation action on  $G_{n,p}$,
  and it acts by conjugation on the centralizers.
  So upon choosing representatives of the conjugacy classes,
  the tautological decomposition descends to a bijection 
  \[\coprod_{[g_1,\dots, g_n] \in G\bs G_{n,p}} C\td{g_1,\dots,g_n}_f^{\con}\to  G\bs G_{n,p,+1} .
    \]    
\end{proof}

The previous theorem is a generalization of \cite[Lemma 4.13]{HKR}.
Indeed if we take $G$ finite, then using \cite[Theorem B]{HKR}, we get
\[ \vert G \bs G_{n,p} \vert =   \sum_{[x,g_1,\dots,g_n] \in G \bs G_{n,p,+1}} \frac{1}{\vert C\td{x,g_1,\dots,g_n}  \vert}=\frac{\vert G_{n,p,+1}  \vert} {\vert G \vert}.\]
The last identity is the class equation,
and $G_{n,p,+1}$ can be identified with $\Hom(\mathbb{Z} \times \mathbb{Z}_p^n, G)$
which is used in \cite[Lemma 4.13]{HKR}.

\begin{remark} \label{rk:chromatic sequence}
  In the case when $G$ has only $p$-primary torsion and a finite model for $\underline{E}G$,
  we can now explain why $\chi_{\orb}$ should be thought
  as the chromatic height $-1$ Euler characteristic.
  In Corollary \ref{cor:euler morava quotients} we showed that
  \[ \chi_{K(n+m)}(BG) = \sum_{[g_1,\dots,g_n] \in G \bs G_{n,p}} \chi_{K(m)} (BC\td{g_1, \dots, g_n}) \]
  for all $m,n\geq 0$, where $\chi_{K(0)}$ has to be read as
  the rational Euler characteristic $\chi_\mQ$.
  If every torsion element of the group $G$ has $p$-power order, then
  $G_{n-1,p,+1}=G_{n,p}$, by definition. So Theorem \ref{thm:chi_K(n)_chi_orb} says that
  \[ \chi_{K(n-1)}(BG) = \sum_{[g_1,\dots,g_n] \in G \bs G_{n,p}} \chi_{\orb}[B_{\gl}C\td{g_1, \dots, g_n}] .\]
  This extends the formula of
  Corollary \ref{cor:euler morava quotients} to the case $m=-1$, as along as we interpret
  $\chi_{K(-1)}$ as the orbifold Euler characteristic.
  This explains why at a prime $p$, we get a chromatic sequence of Euler characteristics
  starting at $-1$:
  \[\chi_{\orb}, \chi_{\mathbb{Q}}, \chi_{K(1)}, \dots, \chi_{K(n)}, \dots,\]
  where the $i$-th term can be obtained from the $j$-th term for $j<i$
  using the tuples of $j-i$ many commuting elements of finite order.
  We will investigate these invariants more closely in Section \ref{sec:examples} for $p=2$ in the case of right angled Coxeter groups. If $p$ is odd, the latter sequence is a special case of the sequence considered in \cite{Yanovski}. In this special case, the classifying space $BG$ is a finite colimit of $\pi$-finite $p$-spaces and $\chi_{\orb}[B_{\gl}G]=\vert BG \vert$, where $\vert -\vert$ is the generalized homotopy cardinality defined by Yanovski. 
\end{remark}

\section{Computations and examples} \label{sec:examples}

In this section we apply the theory developed thus far
to concrete examples and exhibit explicit formulas for the Morava $K$-theory
Euler characteristics of several classes of infinite discrete groups.
In some cases these computations can also be done using
the equivariant Atiyah--Hirzebruch spectral sequence
by employing nice cellular models for $\underline{E}G$,
compare Proposition \ref{prop:euleralternating}.
If such models are available, then we make computations in two different ways and compare them.
However, in general the combinatorics of cellular models of $\underline{E}G$ is too involved
for this direct approach, and then we employ
Corollary \ref{cor:euler morava quotients} and Corollary \ref{cor:Characters for BG}.
Examples are arithmetic groups and mapping class groups, see
Subsections \ref{subsec:GLZ}--\ref{subsec:gamma} below.
All computations at the height $n>1$ are new, with the exception of Subsection \ref{subsec:sl3},
which can be thought of as a consequence of \cite{TezYag}.
We believe that the calculations for the symplectic and mapping class groups
in Subsections \ref{subsec:SpZ} and \ref{subsec:gamma} are new also
for the chromatic height $n=1$.

If $G$ admits a finite model for $\underline{E}G$,
we shall write $\chi_{\orb} (BG)$ for the orbispace Euler characteristic $\chi_{\orb} [B_{\gl}G]$
of the global classifying space $B_{\gl}G$.
If $G$ is additionally countable and virtually torsion-free,
then as observed in Example \ref{eg:chi of G dbs A} (v),
the invariant $\chi_{\orb} (BG)$ coincides with the `virtual Euler characteristic'
of $G$ as defined by Wall \cite{Walleuler},
also called the orbifold Euler characteristic of $G$.
So our notation is consistent with that of previous literature.

\subsection{Right angled Coxeter groups}

In this subsection we apply the general theory to right angled Coxeter groups.
The main reference is \cite[Chapter 7]{Dav}. 
Let $L$ be a finite graph with vertex set $S$ and set of edges $\mathcal{E}$.
The right angled Coxeter group associated to $L$ is the group
\[W(L)=\langle s \in S \; \vert \; s^2=1 \;\text{for all}\; s \in S,\; \text{and}\; (st)^2=1 \;\text{for all} \; \{s,t\} \in \mathcal{E} \rangle. \]
For any subset $T \subset S$, the subgroup $W_T$ generated by $T$
is again a right angled Coxeter group associated to the full subgraph spanned by $T$.
The subset $T$ is called spherical if the subgroup $W_T$ is finite. This is the case if and only if the elements of $T$ commute, i.e., the full subgraph spanned by $T$ is the complete graph. In this case $W_T$ is a product of $T$ many copies of $C_2$.

We will use a finite model for $\underline{E}W(L)$, known as the Davis complex of $W(L)$ \cite[Chapter 7]{Dav}. We denote it by $\Sigma=\Sigma(W(L))$. We do not recall its construction here but we will now give a description of an equivariant cubical cell structure on it. 

In the case when $L$ has only one vertex, we have $W(L)=C_2$.
The group $C_2$ acts on the interval $I=[-1,1]$ by sign.
More generally, given a finite spherical $T \subset S$, the subgroup $W_T=\prod_{\vert T \vert} C_2$ acts on $I_T=\prod_{i=1}^{\vert T \vert} [-1,1]$ by the componentwise sign action.
The minimal $C_2$-CW structure on  $I=[-1,1]$ with one fixed 0-cell, one free 0-cell
and one free 1-cell induces a $W_T$-CW structure on $I_T$.

In general, one has a finite increasing filtration $\bigcup_{m \geq 0}\Sigma^m =\Sigma$,
where $\Sigma^{\vert S \vert}=\Sigma$ and for all $m \geq 0$,
there are cofiber sequences of $W(L)$-spaces
\begin{equation}\label{eq:WL_filtration}
\Sigma^{m-1} \hookrightarrow \Sigma^{m} \rightarrow \bigvee_{ \text{spherical} \;T \subset S, \; \vert T \vert =m} W(L) \ltimes_{W_T} I_T/\partial I_T,   
\end{equation}
where $\partial I_T$ is the boundary of the cube $I_T$. 

The papers \cite{San} and \cite{DegL} use this filtration to compute
the equivariant $K$-homology and $K$-theory of  $\underline{E}W(L)$.
We will first use the same method to compute $K(n)^*(BW(L))$ and $E^*(BW(L))$,
and then double-check the results
by using Corollary \ref{cor:euler morava quotients} and Corollary \ref{cor:Characters for BG}
instead.

We will consider cochain complexes with values in the abelian category
of graded $K(n)^*$-modules.
Since $W(L)$ has only $2$-torsion, we work at the prime $2$.
The calculations start from the  well-known fact that
\[K(n)^*(BC_2) \cong K(n)^*[x]/(x^{2^n}),\]
where $x$ is of degree 2, namely the Euler class of the complex line bundle over $B C_2$
associated to the complex sign representation.

The next lemma computes Bredon cohomology for cubes
and elementary abelian subgroups, corresponding to the spherical subsets.

\begin{lemma} \label{lem:cubebredon}
  Let $T \subset S$ be a spherical subset of cardinality $l$ of a finite graph.
  Then 
  \[\widetilde{H}^i_{W_T}(I_T/\partial I_T, K(n)^*(-))=
    \begin{cases} (K(n)^*\{x, \dots, x^{2^n-1}\})^{\otimes l}& \text{ for $i=0$,}\\
    0& \text{ for $i\ne 0$,}
  \end{cases} \]
where the tensor power is formed over $K(n)^*$.
\end{lemma}

\begin{proof}
  We start with the case where $T$ has one element, in which case $I_T=[-1,1]$ with
  sign action by $W_T=C_2$.
  We use the minimal $C_2$-CW-structure on $I/\partial I$ with two fixed $0$-cells
  (one of which is the basepoint) and one free $1$-cell.
  The reduced Bredon cochain complex of $I/\partial I$ in the category of $K(n)^*$-modules
  is concentrated in degrees $0$ and $1$, where the groups are the kernel and cokernel,
  respectively of the homomorphism
  \[ \res \colon
    K(n)^*\{1,x,\dots,x^{2^n-1}\}= K(n)^*(B C_2)  \to K(n)^* . \]
  The restriction homomorphism is surjective and sends all positive powers of $x$ to zero,
  which proves the claim in the special case.

  In the general case we use the product $W_T$-cell structure on $I_T$
  and the K\"unneth formula for the Morava $K$-theories.
  The stabilizers of $I_T$ are elementary abelian of the form $W_{T'}$ where $T' \subset T$.
  The $W_T$-cellular structure of $I_T$ is the product of the $C_2$-cellular structures
  of $I=[-1,1]$ and each product cell has an elementary abelian stabilizer.
  Using the K\"unneth formula $K(n)^*(BG\times BH) \cong K(n)^*(BG) \otimes_{K(n)^*} K(n)^*(BH)$,
  we see that the reduced Bredon cochain complex
  \[\widetilde{C}^{\bullet}_{W_T}(I_T/\partial I_T, K(n)^*(-))\] 
  is isomorphic to the tensor power
  \[\widetilde{C}^{\bullet}_{C_2}(I/\partial I, K(n)^*(-))^{\otimes l}\]
  in the category of cochain complexes of graded $K(n)^*$-modules.
  Using the special case and that $K(n)^*$ is a graded field,
  the K\"unneth theorem for cochain complexes proves the desired result. 
\end{proof}

\begin{prop}\label{prop:chi_of_W(L)}
  Let $L$ be a finite graph with vertex set $S$, and $n \geq 0$.
  Then for the prime $p=2$,
  the $K(n)^*$-module $K(n)^*(BW(L))$ is concentrated in even degrees and
  is of dimension 
  \[ \dim_{K(n)^*}(K(n)^*(B W(L))) =
    \sum_{l=0}^{\vert S \vert} s(l)\cdot (2^n-1)^l, \]
  where $s(l)$ is the number of spherical subsets $T \subset S$ of size $l$.
  Hence the $K(n)$-Euler characteristic is
  \[\chi_{K(n)}(BW(L))=\sum_{l=0}^{\vert S \vert } s(l) (2^n-1)^l.\]
\end{prop}
\begin{proof}
  We will show that the equivariant Atiyah--Hirzebruch spectral sequence based on the Davis complex
  $\Sigma=\Sigma(W(L))$ that models $\un{E}W(L)$ collapses.
  To determine the $E^2$-term of this spectral sequence
  we use the filtration $\Sigma^m$ of $\Sigma$.
  In a first step we show by induction on $m$ that
  the Bredon cohomology groups $H^i_{W(L)}(\Sigma^m, K(n)^*(-))$ are trivial for $i\ne 0$,
  and $H^0_{W(L)}(\Sigma^m, K(n)^*(-))$ is even and of dimension
  \[ \dim_{K(n)^*}(H^0_{W(L)}(\Sigma^m, K(n)^*(-))) =
    \sum_{l=0}^m s(l)\cdot (2^n-1)^l.  \]
  The induction starts with $m=-1$, where $\Sigma^{-1}$ is empty, and there is nothing to show.
  For the inductive step we use the cofiber sequence \eqref{eq:WL_filtration}.
  By induction and Lemma \ref{lem:cubebredon}, the Bredon cohomology groups of $\Sigma^{m-1}$
  and of the cofiber vanish except in cohomological degree $i=0$,
  and they are concentrated in even internal degrees.
  So the Bredon cohomology long exact sequence decomposes into short exact sequence
  and shows that the same is true for $\Sigma^m$.
  Moreover, the $K(n)^*$-dimension of the 0-th Bredon cohomology of $\Sigma^m$
  is the sum of the dimensions of the 0-th Bredon cohomology of $\Sigma^{m-1}$
  and the cofiber, which we know by induction and Lemma \ref{lem:cubebredon}.
  This completes the inductive step.

  Because $H^*_{W(L)}(\Sigma,K(n)^*(-))$ is concentrated in Bredon cohomology degree 0,
  the Atiyah--Hirzebruch spectral sequence for $K(n)^*(B W(L))$ collapses at the $E^2$-term.
  Since the spectral sequence consists of modules over the graded field $K(n)^*$,
  the evenness property and dimension of the $E^2$-term are inherited by the abutment.
  Hence $K(n)^*(B W(L))$ is concentrated in even degrees, and
  \begin{align*}
    \dim_{K(n)^*}(K(n)^*(B W(L))) &=  \dim_{K(n)^*}(H^0_{W(L)}(\Sigma^{|S|},K(n)^*(-)))\\
    &= \sum_{l=0}^{\vert S \vert} s(l)\cdot (2^n-1)^l.
  \end{align*}
\end{proof}

As a reality check we will now rediscover the formula
for the $K(n)$-Euler characteristic of $B W(L)$
from Proposition \ref{prop:chi_of_W(L)} in a different way
by using our theory, specifically Corollary \ref{cor:euler morava quotients}.
For this we need to count the number of $W(L)$-conjugacy classes
of $n$-tuples of $2$-power order elements in $W(L)$:

\begin{prop} \label{prop:coxeterconjugacy}
  Let $L$ be a finite graph with vertex set $S$, and $n \geq 0$.
  Then
  \[\vert W(L) \bs W(L)_{n,2} \vert=\sum_{l=0}^{\vert S \vert } s(l) (2^n-1)^l.\]
\end{prop}

\begin{proof}
  By \cite[pages 13-14]{DegL}, every finite subgroup of $W$
  is subconjugate to $W_T$ for some spherical subset $T$ of $S$.
  Any element $s \in S$ gives $(2^n-1)$ many elements of $W_{n,2}$ by taking tuples $(w_1, \dots, w_n)$, where at least one $w_i=s$ and $w_j=1$ or $s$ for any $j$. Now consider any general element $(w_1, \dots, w_n) \in W_{n,2}$, where all the coordinates belong to some $W_T$, where $T$ is spherical. Let $\{s_1, \dots, s_l\}$ be the set of all letters from the reduced expressions of $w_i$-s. Then $(w_1, \dots, w_n)$ can be uniquely written up to permutation in the Cartesian power $W^{\times n}$ as a product $v_1v_2 \cdots v_l$, where $v_i \in \{1,s_i\}^{\times n} -\{(1,\dots,1)\}$. This implies that the number of tuples $(w_1, \dots, w_n) \in W_{n,2}$ where all the coordinates belong to some $W_T$ with $T$ spherical and which involve exactly $l$ many letters in total from $S$ is equal to $(2^n-1)^l$, meaning that every spherical subset of size $l$ gives exactly $(2^n-1)^l$ many elements in $W_{n,2}$. None of these tuples are conjugate since the abelianization of $W$ is an $\mathbb{F}_2$-vector space generated by $S$.
  In particular, the normalizer of $W_T$ for $T$ spherical,
  agrees with the centralizer which also follows from \cite[Proposition 4.10.2]{Dav}.
  This completes the proof. \end{proof}

\begin{remark} \label{rk:coxeterconjugacy} Proposition \ref{prop:coxeterconjugacy} can be now used to recover the formula  
\[\chi_{K(n)}(BW(L))=\sum_{l=0}^{\vert S \vert } s(l) (2^n-1)^l\]
without using the equivariant cubical cell structure of $\Sigma$. Indeed by Corollary \ref{cor:euler morava quotients}, it suffices to show that $\chi_{\mathbb{Q}}(BC(H))=1$ for any $H \leq W_{T}$, where $T$ is a spherical subset.  It follows from \cite[Proposition 25 and Theorem 32]{Bark} and \cite[Theorem 4.1.6]{Dav} that $C(H)=W_{S'}$ for some subset $S' \subset S$. But $\chi_{\mathbb{Q}}(BW_{S'})=1$ since right-angled Coxeter groups are rationally acyclic. 

\end{remark}

\begin{eg} \label{eg:Coxeterorb}
  For $n=0$, the Euler characteristic formula of Proposition \ref{prop:chi_of_W(L)}
  specializes to $\chi_{\mathbb{Q}}(BW(L))=1$, a consequence of the well-known fact
  that $B W(L)$ is rationally acyclic.
 If we pretend that we are allowed to let $n=-1$, then we recover the well-known orbifold Euler characteristic 
\[\chi_{\orb}(BW(L))=\sum_{l=0}^{\vert S \vert } s(l)\cdot \tfrac{(-1)^l}{2^l},\]
see for example \cite[(16.8)]{Dav}.
This is in tune with our slogan that the orbifold Euler characteristic
is a `height $-1$ invariant', compare Remark \ref{rk:chromatic sequence}. 
For more concreteness, we examine a special case when the graph $L$ does not contain triangles. In this case we can easily understand the number $s(l)$ for all $l \geq 0$. Indeed, $s(0)=1$, $s(1)=\vert S \vert$ and $s(2)=\vert \mathcal E \vert$. For $n \geq 3$, we have $s(n)=0$. Hence Proposition \ref{prop:coxeterconjugacy} and Remark \ref{rk:coxeterconjugacy} give us 
\[\chi_{K(n)}(BW(L))=\vert W \bs W_{n,2} \vert=1+\vert S \vert(2^n-1)+\vert  \mathcal  E \vert (2^n-1)^2.\]
For $n=-1$, one recovers the orbifold Euler characteristic \cite[(16.8)]{Dav}
\[\chi_{\orb}(B W(L))=1-\tfrac{\vert S \vert}{2}+\tfrac{\vert  \mathcal  E \vert}{4}.\]
\end{eg}

Now we apply the character theory of Section \ref{sec: E theory of BG}. Let $E$ denote the Morava $E$-theory at prime $2$ and height $n$. We write $W=W(L)$ for brevity.
Corollary \ref{cor:Characters for BG} provides an isomorphism
\[L(E^*) \otimes_{E^*} E^*(BW) \cong \bigoplus_{[w_1,\dots,w_n] \in W \bs W_{n,2}} H^*(BC\td{w_1, \dots, w_n}; L(E^*)).\]
As we observed in Remark \ref{rk:coxeterconjugacy},
the space $BC\td{w_1, \dots, w_n}$ has trivial rational cohomology. Hence we get an isomorphism
\[L(E^*) \otimes_{E^*} E^*(BW) \cong L(E^*)\{W \bs W_{n,2}\},\]
where the right hand side is the free $L(E^*)$-module on the set $W \bs W_{n,2}$,
with generators of degree $0$. The cardinality of this set,
and hence the rank of $L(E^*) \otimes_{E^*} E^*(BW)$ is given by
Proposition \ref{prop:coxeterconjugacy}. 

Because $E^*(BW(L))$ is finitely generated over $E^*$
by Theorem \ref{thm:compact_finiteness} (iv),
and because $K(n)^*(BW(L))$ is even by Proposition \ref{prop:chi_of_W(L)},
we can conclude by \cite[Proposition 3.5]{StricklandMorava}
that $E^*(BW(L))$ is in fact even and free over $E^*$. Since we also know its rank after scalar
extension to the ring $L(E^*)$, we deduce:

\begin{theorem}
  Let $L$ be a finite graph with vertex set $S$, and $n \geq 0$.
  The Morava $E$-cohomology $E^*(BW(L))$ at the prime $2$ is even and free of rank 
  \[\sum_{l=0}^{\vert S \vert } s(l) (2^n-1)^l.\]
\end{theorem}

The previous theorem can be directly established
without using \cite[Proposition 3.5]{StricklandMorava}.
In fact one can use the cellular decomposition of the Davis complex $\Sigma$
and the same strategy as in Proposition \ref{prop:chi_of_W(L)}.
One only needs to use the generalized version of K\"unneth theorem proved
in \cite[Corollary 5.11]{HKR}.
See also \cite{KLL} for a general computation of equivariant (co)homology of graph products.

\subsection{The special linear group \texorpdfstring{$SL_3(\mathbb{Z})$}{SL(3,Z}} \label{subsec:sl3}

In this subsection we recover the well-known calculations of $K(n)^{*}(BSL_3(\mZ))$ and $E^{*}(BSL_3(\mZ))$. These groups have been computed by Tezuka and Yagita in \cite{TezYag} using the finite cellular model of $\underline{E}SL_3(\mZ)$ constructed by Soul\'{e} \cite{So}. The height $1$ computations can be also recovered by \cite{San1}. Some of these calculations are also included in \cite[Section VII]{Schuster}. 

The main strategy in \cite{TezYag} and \cite{San1} is to use the equivariant Atiyah--Hirzebruch spectral sequence, compute the Bredon cohomology of $\underline{E}SL_3(\mZ)$ with various coefficients,
and then the generalized cohomology. The spectral sequences are relatively easy to handle since the virtual cohomological dimension of $SL_3(\mZ)$ is equal to $3$ \cite{BS73}.  
The explicit finite model of Soul\'{e} \cite{So} is based on finite subgroup classification of  $SL_3(\mathbb{Z})$ by \cite{Tah}. We do not recall here the details of this model and refer to \cite[Table 1]{San1} and \cite[Theorem 2]{So} for the equivariant cellular structure and the list of stabilizers. 

We first recall the approach of \cite{TezYag} for computing $\chi_{K(n)}(BSL_3(\mZ))$. After this we offer an alternative way to arrive at these calculations:
instead of completely writing out the equivariant cellular chain complex
of $\underline{E}SL_3(\mZ)$, we compute the centralizers of finite abelian subgroups
and use Corollary \ref{cor:euler morava quotients} to calculate $\chi_{K(n)}(BSL_3(\mZ))$.
For $n=1$ this has been already done in \cite[Example 4.3]{Adem}. \smallskip

The group $SL_3(\mZ)$ has only $2$ and $3$-torsion \cite{Tah}.
Hence the $p$-primary case for $p \geq 5$ is equivalent to the rational case.
By \cite[p.\,8 Corollary]{So}, the space $\underline{B}SL_3(\mZ)=SL_3(\mZ) \bs \underline{E}SL_3(\mZ)$ is contractible
and hence rationally acyclic. We start with the easier case of $p=3$.
The following is proved in \cite[Section 5]{TezYag}:

\begin{prop}\label{prop:sl3 at prime 3} Let $n \geq 1$ and $p=3$. Then $K(n)^*(BSL_3(\mZ))$ is even and $\chi_{K(n)}(BSL_3(\mZ))=3^n$. 
\end{prop}

\begin{proof} We recall the proof from \cite{TezYag}.
Proposition \ref{prop:euleralternating} yields the relation
\[\chi_{K(n)}(BSL_3(\mZ))=\sum_{G\sigma} (-1)^{n_\sigma} \vert H^{\sigma} \bs H^{\sigma}_{n,3} \vert,\] 
where the sum runs over all $G$-orbits of cells $\sigma$ of Soul\'{e}'s model
of $\underline{E}SL_3(\mZ)$, the number $n_\sigma$ is the dimension of $\sigma$ and $H^{\sigma}$ is the stabilizer of $\sigma$. By looking at \cite[Table 1]{San1}, we see that there are only six equivariant cells whose stabilizers have $3$-torsion: three $0$-cells  with stabilizer $S_4$, one $0$-cell with stabilizer $D_{12}$ and two $1$-cells with stabilizer $S_3$. A simple calculation shows that
\[\vert S_4 \bs (S_4)_{n,3} \vert =\vert D_{12} \bs (D_{12})_{n,3}\vert =\vert S_3 \bs (S_3)_{n,3} \vert =\tfrac{3^n-1}{2}+1.\]
This follows from the fact that Sylow $3$-subgroups of all these three groups are isomorphic to $C_3$ and the two generators of $C_3$ are conjugate in $S_3$ (and hence in $S_4$ and $D_{12}$). All the remaining equivariant cells have stabilizers without $3$-torsion. Again \cite[Table 1]{San1} tells us how many such equivariant cells we have: one $0$-cell, six  $1$-cells, five $2$-cells, and one $3$-cell. All in all, we get
\begin{align*}\chi_{K(n)}(BSL_3(\mZ))
  &= \sum_{G\sigma} (-1)^{n_\sigma} \vert H^{\sigma} \bs H^{\sigma}_{n,3} \vert\\
  &= 4 (\tfrac{3^n-1}{2}+1)-2(\tfrac{3^n-1}{2}+1)+1-6+5-1\\
  &=2(\tfrac{3^n-1}{2}+1)-1=3^n-1+2-1=3^n.\end{align*}
In fact as observed in \cite[Section 5]{TezYag},
the $K(n)^*$-module $K(n)^*(BSL_3(\mZ))$ can be described using \cite[Lemma 6]{So} and one can see that there is an isomorphism $K(n)^*(BSL_3(\mZ)) \cong K(n)^*(BS_3) \oplus \widetilde{K(n)}^*(BS_3)$, where $\widetilde{K(n)}^*(-)$ is the reduced cohomology.
By the remark immediately after \cite[Theorem E]{HKR}, all symmetric groups
are `good' in the sense of \cite[Definition 7.1]{HKR}, and hence their Morava $K$-theory
is concentrated in even degrees. In particular, we know that $K(n)^*(BS_3) $ is even.
\end{proof}

We would like to compare the latter result with the formula
of Corollary \ref{cor:euler morava quotients}.
By the classification of finite subgroups of \cite{Tah}, there are only
two conjugacy classes of non-trivial finite 3-subgroups of $SL_3(\mZ)$, namely the
cyclic groups of order 3 generated by the matrices 
\[A=
\begin{pmatrix} 
1 & 0 & 0\\
0 & 0 & -1\\
0 & 1 & -1
\end{pmatrix}
\;\;\;\; \text{and}  \;\;\;\;
B=
\begin{pmatrix}
0 & 1 & 0\\
0 & 0 & 1\\
1 & 0 & 0
\end{pmatrix}
,
\]
respectively.
The following is observed in \cite{Adem} and \cite{Upa} without a proof:

\begin{lemma} \label{lem:centralisers at prime 3}
  Centralizers of finite $3$-subgroups of $SL_3(\mZ)$ are finite.
\end{lemma}
\begin{proof}
  We need to show that the centralizers of the matrices $A$ and $B$ are finite.
  The centralizer of $A$ is isomorphic to the centralizer of 
\[ A'=
\begin{pmatrix}
0 & -1\\
1 & -1
\end{pmatrix}
\]
in $SL_2(\mZ)$.
This matrix has order $3$, and explicitly solving the equation
$Z A'=A' Z$ reveals that the centralizer of $A'$ is generated by
the matrix
\[
\begin{pmatrix}
1 & -1\\
1 & 0
\end{pmatrix}
\]
of order $6$. In particular, the centralizer of $A'$ in $SL_2(\mZ)$ is finite.

The second case requires a bit more work.
The matrix $B$ has three different eigenvalues $1, \zeta, \zeta^2$,
where $\zeta=e^{2\pi i/3}$ is a primitive third root of unity.
Then over $\mathbb{C}$, the matrix $B$ is diagonalizable 
\[Z^{-1}BZ=\begin{pmatrix} 
1 & 0 & 0\\
0 & \zeta & 0\\
0 & 0 & \zeta^2
\end{pmatrix},\]
where $Z$ is the matrix of eigenvectors,
\[Z=\begin{pmatrix} 
1 & 1 & 1\\
1 & \zeta & \zeta^2\\
1 & \zeta^2 & \zeta
\end{pmatrix}.\]
The centralizer inside $SL_3(\mathbb{Q}(\zeta))$ of the diagonal matrix above consists of general diagonal matrices of the form
\[C=\begin{pmatrix} 
a & 0 & 0\\
0 & b & 0\\
0 & 0 & c
\end{pmatrix}\]
with $a,b,c \in \mathbb{Q}(\zeta)$ and $abc=1$.
The centralizer of $B$ in $SL_3(\mZ)$ thus consists of all matrices
of the form $Z C Z^{-1}$ which additionally have integral coefficients.
This reveals the general form of centralizing matrices of $B$ as
\[\frac{1}{3}\begin{pmatrix} 
a+b+c & a+b\zeta^2+c\zeta  & a+b\zeta+c\zeta^2\\
a+b\zeta+c\zeta^2 & a+b+c  & a+b\zeta^2+c\zeta \\
a+b\zeta^2+c\zeta & a+b\zeta+c\zeta^2  & a+b+c 
\end{pmatrix},\]
for  $a,b,c \in \mathbb{Q}(\zeta)$,
subject to the condition that all the  matrix entries
$\frac{1}{3}(a+b+c)$, $\frac{1}{3}(a+b\zeta+c\zeta^2)$ and $\frac{1}{3}(a+b\zeta^2+c\zeta)$
are integers, and $abc=1$.
Then also 
\[ a = \frac{1}{3}(a+b+c) +\frac{1}{3}(a+b\zeta+c\zeta^2) +\frac{1}{3}(a+b\zeta^2+c\zeta)\]
lies in $\mZ$, and
\[ b = \frac{1}{3}(a+b+c) +\frac{\zeta^2}{3}(a+b\zeta+c\zeta^2) +\frac{\zeta}{3}(a+b\zeta^2+c\zeta)\]
lies in $\mZ[\zeta]$, and similarly,  $c\in\mZ[\zeta]$.
A special case of Dirichlet's unit theorem says that $\mZ[\zeta]^{\times}=\{\pm 1, \pm \zeta, \pm \zeta^2\}$ and hence there are only finitely many $a,b,c \in \mZ[\zeta]$ with $abc$=1. This shows that $C\td{B}$ is finite.  
\end{proof}

\begin{remark} \label{rk:prime 3 centraliser formula}
  Using Lemma \ref{lem:centralisers at prime 3} and Corollary \ref{cor:euler morava quotients},
  we recover the Euler characteristic $\chi_{K(n)}(BSL_3(\mZ))$.
  Indeed, Lemma \ref{lem:centralisers at prime 3}  tells us that for any $(g_1, \dots, g_n) \in SL_3(\mZ)_{n,3}$ with at least one $g_i \neq 1$, the centralizer $C\td{g_1, \dots, g_n}$ is finite.
  Moreover, $C(1)=SL_3(\mZ)$ and as mentioned above by \cite[p.\,8 Corollary]{So},
  the space $\underline{B}SL_3(\mZ)$ is contractible and hence $BSL_3(\mZ)$ is rationally acyclic. All in all we get 
\begin{align*}\chi_{K(n)}(BSL_3(\mZ) )\ = \sum_{[g_1,\dots,g_n] \in SL_3(\mZ) \bs SL_3(\mZ)_{n,3} }\chi_\mQ(B C\td{g_1, \dots, g_n})=\\ \sum_{[g_1,\dots,g_n] \in SL_3(\mZ) \bs SL_3(\mZ)_{n,3} } 1=\vert SL_3(\mZ) \bs SL_3(\mZ)_{n,3}  \vert.
\end{align*}
Again using the classification of finite subgroups of $SL_3(\mZ)$ \cite{Tah}, we see that 
\[\vert SL_3(\mZ) \bs SL_3(\mZ)_{n,3}  \vert=\frac{3^n+1}{2}+\frac{3^n+1}{2}-1=3^n.\]
We note that this approach still uses the cellular structure of \cite{So} since the proof of the contractibility of $\underline{B}SL_3(\mZ)$ does.
However, we do not need to go into the combinatorics of the cell structure
to compute $\chi_{K(n)}(BSL_3(\mZ))$. 
\end{remark}

Finally, using Corollary \ref{cor:Characters for BG}, Remark \ref{rk:prime 3 centraliser formula} and Proposition \ref{prop:sl3 at prime 3} as well as \cite[Proposition 3.5]{StricklandMorava}, we obtain:

\begin{prop}
  Let $n \geq 1$ and $p=3$. Then $E^*(BSL_3(\mZ))$ at the prime $3$
  is even and free of rank $3^n$ as a graded $E^*$-module. 
\end{prop}

The latter can be also recovered from \cite[Section 5]{TezYag}
since $E$ is Landweber exact. \smallskip

Now we switch to the prime $p=2$ case. This case is more involved, though centralizers can be computed with the same methods as in Lemma \ref{lem:centralisers at prime 3}. Again we recall the calculation from \cite{TezYag} using the cellular structure of \cite{So} and then compare the result to the formula we get from Corollary \ref{cor:euler morava quotients}. 

\begin{prop}\label{prop:sl3 at prime 2}
  Let $n \geq 1$ and $p=2$. Then $K(n)^*(BSL_3(\mZ))$ is even
  and $\chi_{K(n)}(BSL_3(\mZ))=2^{2n+1}-2^{n+1}+1$. 
\end{prop}

\begin{proof} We need formulas from \cite[Section 4]{HKR1}: $\chi_{K(n)}(S_3)=2^n$,
\[\chi_{K(n)}(S_4)=\tfrac{7\cdot4^n-3\cdot 2^n +2}{6} \;\;\; \text{and} \;\;\; \chi_{K(n)}(D_8)=\tfrac{3\cdot 4^n -2^n}{2}.\]
Along similar lines one shows that $\chi_{K(n)}(D_{12})=4^n$.
We recall again the equivariant cellular structure on $\underline{E}SL_3(\mZ)$,
compare \cite{So} and \cite[Table 1]{San1}:
There are three equivariant $0$-cells with stabilizer $S_4$, one equivariant $0$-cell with stabilizer $D_8$, and another equivariant $0$-cell with stabilizer $D_{12}$. One has two equivariant $1$-cells with stabilizer $C_2$, two with stabilizer $D_8$, two with stabilizer $S_3$, and another two with stabilizer $C_2 \times C_2$. Further, we have three equivariant $2$-cells with stabilizer $C_2$, one with stabilizer $C_2 \times C_2$, and one free equivariant $2$-cell. Finally, there is only one equivariant $3$-cell which is free. All in all by Proposition \ref{prop:euleralternating}, we get 
\begin{align*}
  \chi_{K(n)}(BSL_3(\mZ))
  = &\sum_{G\sigma} (-1)^{n_\sigma} \vert H^{\sigma} \bs H^{\sigma}_{n,2} \vert\\
  = &3 \cdot \tfrac{7\cdot4^n-3\cdot 2^n +2}{6} +4^n+\tfrac{3\cdot 4^n -2^n}{2}\\
    &-2 \cdot 4^n-2 \cdot 2^n-2 \cdot 2^n-2 \cdot \tfrac{3\cdot 4^n -2^n}{2} \\&+3 \cdot 2^n+1+4^n-1\\
  &=2^{2n+1}-2^{n+1}+1.\end{align*}
Here again the sum runs over all $G$-orbits of cells $\sigma$ of Soul\'{e}'s model of $\underline{E}SL_3(\mZ)$. 

To show that $K(n)^*(BSL_3(\mZ))$ is even, Tezuka and Yagita use a $1$-dimensional subcomplex of $\underline{E}SL_3(\mZ)$ whose Bredon cohomology agrees with the Bredon cohomology of $\underline{E}SL_3(\mZ)$. By \cite[Theorem E]{HKR} we know that  $S_4$ and $D_8$ are `good' groups,
and hence $K(n)^*(BS_4)$, and $K(n)^*(BD_{8})$ are even.
The rest follows from \cite[Theorem 2.6]{TezYag}. \end{proof}

Now we compare the latter result with Corollary \ref{cor:euler morava quotients}. 

\begin{lemma} \label{lem:centralisers at prime 2} Centralizers of finite abelian $2$-subgroups of $SL_3(\mathbb{Z})$ are rationally acyclic.
\end{lemma}

\begin{proof}
  According to \cite{Upa} the only non-finite centralizers are those of subgroups
  of order at most $2$. The trivial subgroup case follows
  by the fact that $\underline{B}SL_3(\mZ)$ is contractible.
  By \cite{Tah}, there are only two conjugacy classes of subgroups of order $2$.
  They are represented by the subgroups generated by
  \[A=
    \begin{pmatrix}
      1 & 0 & 0\\
      0 & -1 & 0\\
      0 & 0 & -1
    \end{pmatrix}
    \;\;\;\; \text{and}  \;\;\;\;
    B=
    \begin{pmatrix} 
      -1 & 0 & 0\\
      0 & 0 & 1\\
      0 & 1 & 0
    \end{pmatrix}.
  \]
  The centralizer of $A$ is isomorphic to $GL_2(\mZ)$
  which is rationally acyclic since $SL_2(\mZ)$ is.
  The computation of $C\td{B}$ is similar to the computation
  in Lemma \ref{lem:centralisers at prime 3}.
  The matrix $B$ has eigenvalues $\pm 1$ and it is diagonalized  over $\mZ[\frac{1}{2}]$ by
  \[Z^{-1}BZ=\begin{pmatrix} 
    1 & 0 & 0\\
    0 & -1 & 0\\
    0 & 0 & -1
  \end{pmatrix},\]
with  \[Z=\begin{pmatrix} 
  0 & 1 & 0\\
  1 & 0 & 1\\
    1 & 0 & -1
\end{pmatrix}\text{\quad and\quad }
Z^{-1}=\frac{1}{2}\begin{pmatrix} 
  0 & 1 & 1\\
  2 & 0 & 0\\
  0 & 1 & -1
\end{pmatrix}.
\]
The centralizer of the above diagonal matrix in $SL_2(\mZ[\frac{1}{2}])$ is the subgroup of matrices of the form
\[C=\begin{pmatrix} 
  b & 0 & 0\\
  0 & a_{11} & a_{12}\\
  0 & a_{21} & a_{22}
\end{pmatrix}\]
with determinant equal to $1$. Calculating $ZCZ^{-1}$, we get that a general element in the centralizer $C\td{B}$ is given by a matrix of the form
\[\begin{pmatrix} 
  a_{11} & a_{12}/{2} & -a_{12}/{2}\\
  a_{21} & b/{2}+a_{22}/{2} & b/{2}-a_{22}/{2}\\
  -a_{21} & b/{2}-a_{22}/{2} & b/{2}+a_{22}/{2}
\end{pmatrix},\]
where $a_{11}$, $a_{21}$, $\pm b/{2} +a_{22}/{2}$ are integers, $a_{12}$ is an even integer, and $b(a_{11}a_{22}-a_{21}a_{12})=1$. This implies that $b$ and $a_{22}$ are also integers. 

We recall that $\Gamma(2)$ denotes the congruence subgroup of $SL_2(\mZ)$ of matrices
which are the identity mod $2$. Using the extension
\[1 \to \Gamma(2) \rtimes \mZ/2 \to GL_2(\mZ) \to SL_2(\mathbb{F}_2) \to 1\]
and the above calculation, we conclude that $C\td{B}$ sits in an extension
\[1 \to \Gamma(2) \rtimes \mZ/2 \to C\td{B} \to \mZ/2 \to 1,\]
where $\mZ/2 \leq SL_2(\mathbb{F}_2)$ can be thought as the subgroup of lower triangular matrices with determinant $1$.
By a Lyndon--Hochschild--Serre spectral sequence argument,
it suffices to observe that $\Gamma(2) \rtimes \mZ/2$ is rationally acyclic. This follows since $\Gamma(2)=F_2 \times \mZ/2$, where $F_2$ is freely generated by the matrices
$\left(\begin{smallmatrix} 1 & 2 \\ 0 & 1  \end{smallmatrix}\right)$
and
$\left(\begin{smallmatrix}  1 & 0 \\ 2 & 1 \end{smallmatrix}\right)$,
and the action of $\mZ/2$ on $\Gamma(2)$ inverts both generators.
The Lyndon--Hochschild--Serre spectral sequence now implies that $\Gamma(2) \rtimes \mZ/2$ is rationally acyclic and this completes the proof. \end{proof}

\begin{remark} \label{rk:prime 2 centraliser formula}
  Using Lemma \ref{lem:centralisers at prime 2} and Corollary \ref{cor:euler morava quotients},
  as in the $3$-primary case we get
  \begin{align*}\chi_{K(n)}(BSL_3(\mZ) )\ = \sum_{[g_1,\dots,g_n] \in SL_3(\mZ) \bs SL_3(\mZ)_{n,2}}\chi_\mQ(B C\td{g_1, \dots, g_n})=\\ \sum_{[g_1,\dots,g_n] \in SL_3(\mZ) \bs SL_3(\mZ)_{n,2}} 1=\vert SL_3(\mZ) \bs SL_3(\mZ)_{n,2} \vert.
  \end{align*}
  Now using the classification of finite subgroups of $SL_3(\mZ)$ and the explicit identification of conjugacy classes of $2$-subgroups by \cite{Tah}, one can directly compute $\vert SL_3(\mZ) \bs SL_3(\mZ)_{n,2} \vert=2^{2n+1}-2^{n+1}+1$ and recover the calculation in Proposition \ref{prop:sl3 at prime 2} proving that at the prime $p=2$, we have
  \[\chi_{K(n)}(BSL_3(\mZ) ) = 2^{2n+1}-2^{n+1}+1.\]
  We leave the details to the reader. 
\end{remark}

Finally, combining Corollary \ref{cor:Characters for BG}, Remark \ref{rk:prime 2 centraliser formula} and Proposition \ref{prop:sl3 at prime 2}
as well as \cite[Proposition 3.5]{StricklandMorava}, we obtain
the following result:

\begin{prop}
  Let $n \geq 1$ and $p=2$. Then $E^*(BSL_3(\mZ))$ at the prime $2$
  is even and free of rank $2^{2n+1}-2^{n+1}+1$ as a graded $E^*$-module. 
\end{prop}

This proposition can also be deduced from \cite[Sections 2--3]{TezYag}.

\subsection{The special linear group \texorpdfstring{$SL_2(\Oc_K)$}{SL(2,O} for a totally real field \texorpdfstring{$K$}{K}} \label{subsec:slOK}

We start by reviewing the case $SL_2(\mathbb{Z})$. It is well-known that $SL_2(\mathbb{Z})$ is isomorphic to the amalgamated product $\mZ/6 \ast_{\mZ/2} \mZ/4$, see \cite[Example 4.2]{Serre}. This decomposition can be used to easily compute invariants of $SL_2(\mathbb{Z})$. The classifying space $\underline{E}SL_2(\mathbb{Z})$ is a tree, compare Example \ref{dihedral};
so we can apply the equivariant cellular structure of the Bass--Serre tree
of $\mZ/6 \ast_{\mZ/2} \mZ/4$ from \cite[Example 4.2]{Serre} to compute
\begin{align*}
  \chi_{\orb}(BSL_2(\mathbb{Z}))\
  &=\  \chi_{\orb}(B\mZ/6)+ \chi_{\orb}(B\mZ/4)- \chi_{\orb}(B\mZ/2)\\
  &=\ \tfrac{1}{6}+\tfrac{1}{4}-\tfrac{1}{2}\ =\ -\tfrac{1}{12} .
\end{align*}
To compute the classical Euler characteristic $\chi_{\mathbb{Q}}(BSL_2(\mathbb{Z}))$
one can use the Mayer--Vietoris sequence for the rational cohomology,
and that finite groups are rationally acyclic. From this one obtains
\[\chi_{\mathbb{Q}}(BSL_2(\mathbb{Z}))=1.\]
Similar Mayer--Vietoris arguments can be used to compute the $E^*$-cohomology
and Morava $K$-theory of $BSL_2(\mathbb{Z})$,
compare Proposition \ref{prop:euleralternating} and Example \ref{dihedral}.
At the prime $p=2$, we get
\[\chi_{K(n)}(BSL_2(\mathbb{Z}))=2^n+4^n-2^n=4^n\]
and at the prime $p=3$, we get
\[\chi_{K(n)}(BSL_2(\mathbb{Z}))=3^n+1-1=3^n.\]
The Mayer--Vietoris sequence shows that Morava $K$-theory and $E$-theory
of $BSL_2(\mathbb{Z})$ are even and $E^*(BSL_2(\mathbb{Z}))$ is free as a graded $E^*$-module,
of rank $4^n$ for $p=2$, and of rank $3^n$ for $p=3$.

The goal of this section is to generalize the above calculations
to rings of integers in totally real fields.
From now on we assume that $K$ is a totally real field, e.g., $K=\mathbb{Q}(\sqrt{d})$,
where $d$ is any square-free positive integer. We let $\Oc_K$ denote the ring of integers of $K$.
It follows from \cite[p.\,453]{Hard} that
\[\chi_{\orb}(BSL_2(\Oc_K))=\zeta_K(-1),\]
where $\zeta_K$ is the Dedekind zeta function of $K$. When $K=\mathbb{Q}$ this specializes to the Riemann zeta function $\zeta$ and one recovers the well-known identity from above
\[\chi_{\orb}(BSL_2(\mathbb{Z}))=-\tfrac{1}{12}=\zeta(-1).\] 
Brown \cite[Lemma p.\,251]{KBro1} shows that any finite subgroup $H \leq SL_2(\Oc_K)$ is cyclic;
and if $H$ is non-trivial and different from the center $\{\pm \Id\}$  of $SL_2(\Oc_K)$,
then it is contained in the unique maximal finite subgroup. The maximal subgroup containing $H$ is given by the centralizer $C(H)$ and additionally $C(H)=N(H)$, where $N(H)$ is the normalizer.
Using this, Brown \cite[Section 9.1]{KBro1} proves the following formula
for the rational Euler characteristic
\[\chi_{\mathbb{Q}}(BSL_2(\Oc_K))= 2\zeta_K(-1)+\sum_{(H)} (1-\tfrac{2}{\vert H \vert}),\]
where the sum runs over the conjugacy classes of the maximal finite subgroups. The number of such conjugacy classes for real quadratic fields is well understood in terms of class numbers of imaginary quadratic fields,
see \cite{Prestel} and \cite[p. 198]{Hirz}. 

The group $SL_2(\Oc_K)$ is an arithmetic group and hence admits a finite model for $\underline{E}G$ \cite[Theorem 3.2]{Ji}. Using the behavior of finite subgroups, we can see that the conditions of \cite[Corollary 2.8]{LWe} are satisfied and we get the following homotopy pushout square in $G$-spaces, with $G=SL_2(\Oc_K)$:
\begin{align}\begin{aligned} \label{LW pushout}
  \xymatrix{\coprod_{(H)} G \times_H E(H/\{\pm \Id\}) \ar[r] \ar[d] & E (G/\{\pm \Id\}) \ar[d]  \\ \coprod_{(H)} G/H \ar[r] & \underline{E}G, }
  \end{aligned}\end{align}
    where the disjoint union runs over the conjugacy classes of the maximal finite subgroups. This pushout is obtained by applying \cite[Corollary 2.8]{LWe} to the inclusion of the families of subgroups $\{1, \{\pm \Id\}\} \subset \Fin$. The following generalizes the above formula of Brown;
for $n=1$ it specializes to a result of Adem \cite[Example 4.4]{Adem}.

\begin{prop} \label{prop:Euler sl2}
  Let $p$ be an odd prime, and $K$ a totally real field. Then for any $n \geq 0$, we have
\[\chi_{K(n)}(BSL_2(\Oc_K))= 2\zeta_K(-1)+\sum_{(H)} \left({\vert H_{(p)} \vert}^n-\tfrac{2}{\vert H \vert}\right),\]
where the sum runs over the conjugacy classes of the maximal finite subgroups. 
\end{prop}
\begin{proof} Let $G=SL_2(\Oc_K)$. By taking quotients in the pushout \eqref{LW pushout},
  we get a homotopy pushout of spaces
 \[\xymatrix{\coprod_{(H)} B(H/\{\pm \Id\}) \ar[r] \ar[d] & B(G/\{\pm \Id\}) \ar[d]  \\ \coprod_{(H)} *\ar[r] & \underline{B}G, }\]
 where $\underline{B}G=G\bs \underline{E}G$.
 Using the Mayer--Vietoris sequence of this pushout in Morava $K$-theory, we obtain the equation 
\[ \chi_{K(n)}(B(G/\{\pm \Id\})) + \sum_{(H)} 1= \chi_{K(n)}(\underline{B}G ) +\sum_{(H)}  \chi_{K(n)}(B(H/\{\pm \Id\})). \]
Since $p$ is odd, the Atiyah--Hirzebruch spectral sequence shows that
$K(n)^*(B(G/\{\pm \Id\}))$ and $K(n)^*(B(H/\{\pm \Id\}))$ are isomorphic to
$K(n)^*(BG)$ and $K(n)^*(BH)$, respectively.
Additionally, since $\underline{B}G$ admits a finite CW-model,
we get that $\chi_{K(n)}(\underline{B}G )=\chi_{\mathbb{Q}}(\underline{B}G)=\chi_{\mathbb{Q}}(BG)$. Hence the latter equation can be rewritten as
\[ \chi_{K(n)}(BG) + \sum_{(H)} 1= \chi_{\mathbb{Q}}(BG) +\sum_{(H)}  \chi_{K(n)}(BH).\]
Using Brown's formula and $ \chi_{K(n)}(BH)={\vert H_{(p)} \vert}^n$ from \cite{HKR}, we obtain the formula
\[\chi_{K(n)}(BSL_2(\Oc_K))= 2\zeta_K(-1)+\sum_{(H)} \left({\vert H_{(p)} \vert}^n-\tfrac{2}{\vert H \vert}\right). \]
\end{proof}

\begin{remark} \label{rk:sl2 centraliser odd prime} There is an alternative way to prove Proposition \ref{prop:Euler sl2} using Corollary \ref{cor:euler morava quotients}.
  Indeed, by the classification of finite subgroups of $G=SL_2(\Oc_K)$
  as given for example in \cite[Lemma p.\,251]{KBro1}, 
\[ \vert G \bs G_{n,p} \vert-1=\sum_{(H)} ({\vert H_{(p)} \vert}^n-1),\]
where the sum runs over the conjugacy classes of the maximal finite subgroups. This uses that for any non-trivial finite $p$-subgroup $L$, the normalizer $N(L)$ and the centralizer $C(L)$ agree and $L$ is contained in the unique maximal finite subgroup. Using additionally that $C(L)$ is finite, by Corollary \ref{cor:euler morava quotients} and  \cite[Section 9.1]{KBro1}, we get
\begin{align*}
  \chi_{K(n)}(BSL_2(\Oc_K))\
  &=\ \sum_{[g_1,\dots,g_n] \in G \bs G_{n,p} }\chi_\mQ(B C\td{g_1, \dots, g_n})\\
  &=\ \chi_{\mathbb{Q}}(BSL_2(\Oc_K))+\vert G \bs G_{n,p} \vert -1\\
  &=\ 2\zeta_K(-1)+\sum_{(H)} (1-\tfrac{2}{\vert H \vert})+\sum_{(H)} ({\vert H_{(p)} \vert}^n-1)\\
  &=\ 2\zeta_K(-1)+\sum_{(H)} ({\vert H_{(p)} \vert}^n-\tfrac{2}{\vert H \vert}).
\end{align*}
\end{remark}

\begin{remark}
  Similar methods can be used to prove that at the prime $p=2$,
  for any $n \geq 0$, we get the formula
  \[\chi_{K(n)}(BSL_2(\Oc_K))= 2^{n+1}\zeta_K(-1)+\sum_{(H)} \left({\vert H_{(2)} \vert}^n-\tfrac{2^{n+1}}{\vert H \vert}\right).\]
  The reason why the powers of $2$ show up in this formula is that we have $n$-tuples $(\pm 1, \pm 1, \dots, \pm 1) \in SL_2(\Oc_K)_{n,2}$. We will not give complete details for this calculation. 
\end{remark}

Now we consider the Morava $E$-cohomology of $BSL_2(\Oc_K)$;
at height $n=1$ this has already been done in \cite[Example 4.4]{Adem}.
Again we will deal with the odd primary case.
For an odd prime $p$, Remark \ref{rk:sl2 centraliser odd prime}
and the character formula of Corollary \ref{cor:Characters for BG} provide an isomorphism
\[ L(E^*)\tensor_{E^*} E^*(BSL_2(\Oc_K)) \cong H^*(BSL_2(\Oc_K); L(E^*)) \oplus F, \]
where $F$ is a free $L(E^*)$-module in even degrees of rank
$\sum_{(H)} ({\vert H_{(p)} \vert}^n-1)$,
where $(H)$ runs over the conjugacy classes of the maximal finite subgroups.
Alternatively one could use the Mayer--Vietoris sequence associated
to the pushout \eqref{LW pushout}.

To compute $E^*(BSL_2(\Oc_K))$ integrally, it would help if $K(n)^*(BSL_2(\Oc_K))$ were even.
This is the case when $K=\mathbb{Q}$ as discussed at the beginning of this section.
We argue in the next example that this also holds for $K=\mathbb{Q}(\sqrt{5})$.
Consequently, we get an integral computation of $E^*(BSL_2(\mZ[\frac{1+\sqrt{5}}{2}]))$
for all odd primes and all heights. 

\begin{eg}
  Let $K=\mathbb{Q}(\sqrt{d})$ be a real quadratic field, where $d$ is a square-free positive integer. In this case an infinite model for $\underline{E}G$ is given by $\mathbb{H} \times \mathbb{H}$, where $\mathbb{H}$ is the complex upper half plane. Here $SL_2(\Oc_K)$ acts through the embedding into $SL_2(\mathbb{R}) \times SL_2(\mathbb{R})$. The latter acts on $\mathbb{H} \times \mathbb{H}$ via the M\"obius transformations. 
  It follows now from the Mayer--Vietoris sequence of the pushout \eqref{LW pushout},
  that if $K(n)^*(\underline{B}G)$ is even, then so is $K(n)^*(BG/\{\pm \Id\})$ and hence $K(n)^*(BG)$ will be even. It thus suffices to give a criterion which guarantees that $K(n)^*(\underline{B}G)$ is even.

  The space $\underline{B}G =\mathbb{H} \times \mathbb{H} / SL_2(\Oc_K)$ is known as the Hilbert modular surface of the quadratic field $K$ and its rational cohomology is well studied,
  see for example \cite{Hirz}.
  Not many integral cohomology computations of the Hilbert modular surfaces
  are available, however.
  The thesis of Brownstein \cite{Brownst} computes $H^*(\mathbb{H} \times \mathbb{H} / SL_2(\Oc_K), \mathbb{F}_p)$ for $K=\mathbb{Q}(\sqrt{5})$. It follows from the 0-line of
  the spectral sequence tabulated in \cite[p.\,43]{Brownst} that for $p$ an odd prime,
  the reduced cohomology $\widetilde{H}^n(\mathbb{H} \times \mathbb{H} / SL_2(\Oc_K), \mathbb{F}_p)$ vanishes unless $n=2$. 
  This implies that $H^*(\mathbb{H} \times \mathbb{H} / SL_2(\Oc_K), \mathbb{F}_p)$ is even when $K=\mathbb{Q}(\sqrt{5})$ and hence by the Atiyah--Hirzebruch spectral sequence we conclude that $K(n)^*(\underline{B}G)=K(n)^*(\mathbb{H} \times \mathbb{H} / SL_2(\Oc_K))$ is even.
  Thus $K(n)^*(BSL_2(\mZ[\frac{1+\sqrt{5}}{2}]))$ is even.

By \cite[Theorem, p.198]{Hirz} and \cite{Prestel}, when $K=\mathbb{Q}(\sqrt{d})$, there are explicit formulas for computing the numbers and types of conjugacy classes of maximal finite subgroups of $SL_2(\Oc_K)$. If $d \geq 7$ is square-free and coprime to $6$, then any maximal finite subgroup of $SL_2(\Oc_K)$ is isomorphic to $\mZ/4$ or $\mZ/6$ and the number of conjugacy classes of such subgroups is explicitly determined by class numbers of imaginary quadratic fields with discriminants being explicit multiples of $-d$. For $d=5$ case by \cite[Table, p.\,200]{Hirz}, it follows that any maximal finite subgroup of $SL_2(\mZ[\frac{1+\sqrt{5}}{2}])$ is isomorphic to  $\mZ/4$, $\mZ/6$ or $\mZ/10$ and there are exactly two conjugacy classes for each of these groups. We also know the special value of the Dedekind zeta function $\zeta_{\mathbb{Q}(\sqrt{5})}(-1)=\frac{1}{30}$. This allows to make the formula of Proposition \ref{prop:Euler sl2} more explicit. For example, for $p=3$ we have
\begin{align*}
  \chi_{K(n)}(BSL_2(\mZ[\tfrac{1+\sqrt{5}}{2}]))
  &=\tfrac{1}{15}+2(3^n-\tfrac{2}{6})+2(1-\tfrac{2}{4})+2(1-\tfrac{2}{10})\\
  &=2\cdot 3^n+2.
\end{align*}

Since $K(n)^*(BSL_2(\mZ[\frac{1+\sqrt{5}}{2}]))$ is even, by \cite[Proposition 3.5]{StricklandMorava} we conclude that the $E^*$-module $E^*(BSL_2(\mZ[\frac{1+\sqrt{5}}{2}]))$ is even and free of rank $2\cdot 3^n+2$. 
Similar results can be obtained at the prime $p=5$.  
\end{eg}

\subsection{Crystallographic groups}
In this subsection we consider $\chi_{K(n)}(BG)$ for any group $G$ which fits into a short exact sequence ($n \geq 1$)
\[1\to \mZ^m \to G \to \mZ/p \to 1.\]
Such groups are certain crystallographic groups and their invariants have interesting behavior depending on the action of $\mZ/p$ on $\mZ^m$. The rational and orbifold
Euler characteristics were studied in \cite[Example 6.94]{LBook} and the homological and $K$-theoretic computations have been done in \cite{Ademcrys, AGPP, Davis-Lueck}. Using these calculations and analysis of centralizers by Corollary \ref{cor:Characters for BG} and Corollary \ref{cor:euler morava quotients}, one can give explicit formulas for $L(E^*)\tensor_{E^*} E^*(BG)$ and $\chi_{K(n)}(BG)$. In this subsection we will mostly focus on $\chi_{K(n)}(BG)$ and only briefly mention the calculation of $L(E^*)\tensor_{E^*} E^*(BG)$ in the special case of free actions. 

Crystallographic groups $G$ in the above sense admit a finite $G$-CW model for $\underline{E}G$
by \cite[Example 6.94]{LBook}.
If $G$ is torsion-free, then $\chi_{K(n)}(BG)=\chi_{\mathbb{Q}}(BG)=0$. If $G$ has torsion, then it has to be $p$-torsion and any finite subgroup of $G$ is isomorphic to $\mZ/p$. In this case we have an isomorphism $G \cong \mZ^m \rtimes_{\rho} \mZ/p$, where $\rho$ is an integral representation of $\mZ/p$. If the action by $\mZ/p$ has non-trivial fixed points,
then for any $g\in G$ of order $p$,
the  centralizer $C\td{g}$ is isomorphic to the product of $\mZ/p$ with $\mZ^l$, for a fixed $l \geq 1$ (in fact $\mZ^l=(\mZ^m)^{\mZ/p}$) and hence $\chi_{\mathbb{Q}}(BC\td{g})=0$.
It follows from \cite[Example 6.94]{LBook} that $\chi_{\mathbb{Q}}(BG)=0$;
alternatively, one can also use Brown's formula \cite{KBro2} from Example \ref{eg:Brown theorem}.
Hence using Corollary \ref{cor:euler morava quotients}, we get $\chi_{K(n)}(BG)=0$ at the prime $p$. 

Finally, we suppose $\mZ/p$ acts freely on $\mZ^m\setminus \{0\}$.
Let $r$ denote the number of conjugacy classes of non-trivial finite subgroups.
In this case every non-trivial finite subgroup of $G$ is self-normalizing,
and hence
\[ \chi_{K(n)}(BG)=\chi_{\mathbb{Q}}(BG)+r(p^n-1) \]
 by Corollary \ref{cor:euler morava quotients}.
Using \cite[Example 6.94]{LBook} or \cite{KBro2}, we know that $\chi_{\mathbb{Q}}(BG)=-\frac{r}{p}+r$. Hence we obtain the formula
\[\chi_{K(n)}(BG)=rp^n-\tfrac{r}{p}.\]
It follows from the computations in  \cite{Ademcrys, AGPP, Davis-Lueck} that in fact if $\mZ/p$ acts freely on $\mZ^m\setminus \{0\}$, then we have $m=k(p-1)$ and $r=p^{k}$. Finally, we observe that analogously to the case of right angled Coxeter groups in Example \ref{eg:Coxeterorb},
it also makes sense to consider the case $n=-1$ in the above formula,
which recovers $\chi_{\orb}(BG)=0$.

We conclude by observing that when $\mZ/p$ acts freely on $\mZ^m\setminus \{0\}$,
then Corollary \ref{cor:Characters for BG} provides an isomorphism
\[L(E^*)\tensor_{E^*} E^*(BG) \cong H^*(BG; L(E^*)) \oplus F, \]
where $F$ is a free $L(E^*)$-module in even degrees of rank $r(p^n-1)$.
The rational cohomology $H^*(BG; \mathbb{Q})$ is known by \cite{Ademcrys, AGPP, Davis-Lueck}. Thus one obtains the full computation of $L(E^*)\tensor_{E^*} E^*(BG)$. We do not go here into more details. 

\subsection{The general linear group \texorpdfstring{$GL_{p-1}(\mZ)$}{GL(p-1,Z)} for a prime \texorpdfstring{$p\geq 5$}{p>=5}} \label{subsec:GLZ}

The examples discussed so far are computable in two different ways,
either by Corollary \ref{cor:euler morava quotients}
and Corollary \ref{cor:Characters for BG},
or by use of explicit cellular models for $\underline{E}G$
and arguments as in Proposition \ref{prop:euleralternating} and Proposition \ref{prop:Euler sl2}. 

In this section we present a first example $G=GL_{p-1}(\mZ)$
where it is essential to use our formulas from Corollary \ref{cor:euler morava quotients} and Corollary \ref{cor:Characters for BG}, i.e.,
we do not know other ways to arrive at these computations.
At the height $n=1$ these calculations were already done by Adem \cite[Example 4.5]{Adem}. 
As an arithmetic group, $GL_{p-1}(\mZ)$ has a finite model for $\underline{E}G$
see \cite[Theorem 3.2]{Ji}; we do not know any sufficiently explicit
cellular structure to work with.

\begin{prop}\label{prop:Euler_KN_equals_rational_GL}
  Let $p\geq 5$ be a prime. Then for every $n \geq 0$,
  \[\chi_{K(n)}(BGL_{p-1}(\mZ))=\chi_{\mathbb{Q}}(BGL_{p-1}(\mZ)).\]
\end{prop}
\begin{proof}
  It follows from \cite[Proposition 1.1]{LevNic} that the only $p$-power order elements in $GL_{p-1}(\mZ)$ are elements of order $p$.
  Every elementary abelian $p$-subgroup in $GL_{p-1}(\mZ)$ has the rank at most $1$,
  see e.g., \cite[Section 2]{Ash};
  so every finite $p$-subgroup of $GL_{p-1}(\mZ)$ is isomorphic to $\mZ/p$.
  It follows from \cite[Lemma 4]{Ash} that $C\td{A}$ for any $A$ of order $p$ is isomorphic to the group of units in $\mZ[\zeta_p]$ which by the Dirichlet unit theorem is isomorphic to
\[\mZ/p \times \mZ/2 \times \mZ^{\frac{p-3}{2}}.\]
Corollary \ref{cor:euler morava quotients} yields
\[\chi_{K(n)}(BGL_{p-1}(\mZ))=\chi_{\mathbb{Q}}(BGL_{p-1}(\mZ))+\sum_{[A_1, \dots,A_n]} \chi_{\mathbb{Q}}(BC\td{A_1, \dots,A_n}),\]
where $[A_1, \dots,A_n] \in GL_{p-1}(\mZ) \bs (GL_{p-1}(\mZ)_{n,p}-\{(1,\dots,1)\}) $. 
By the above observation 
\[C\td{A_1, \dots,A_n} \cong \mZ/p \times \mZ/2 \times \mZ^{\frac{p-3}{2}}\]
and hence $\chi_{\mathbb{Q}}(BC\td{A_1, \dots,A_n})=0$ since the rational Euler characteristic of tori is zero. This gives the desired result. \end{proof}

For Morava $E$-theory we obtain:
\begin{prop} \label{prop:Morava E GL} Let $p\geq 5$ be a prime. Then the character map gives an isomorphism
\begin{align*} L(E^*)\tensor_{E^*} E^*(BGL_{p-1}(\mZ)) & \cong  \\ H^*(BGL_{p-1}(\mZ); L(E^*)) & \oplus \bigoplus_{ \frac{p^n-1}{p-1} \vert \Cl(\mathbb{Q}(\zeta_p)) \vert} (L(E^*) \oplus L(E^*)[1])^{\otimes_{L(E^*)} \frac{p-3}{2}},\end{align*}
where $\Cl(\mathbb{Q}(\zeta_p))$ denotes the ideal class group of $\mathbb{Q}(\zeta_p)$. 
\end{prop}

\begin{proof}

Given $A \in GL_{p-1}(\mZ)$ of order $p$, then $\zeta_p=e^{2\pi i/p}$ is one of its eigenvalues. Let $x=(x_1, \dots,x_{p-1}) \in \mathbb{Q}(\zeta_p)^{p-1}$ be one of the eigenvectors and $I(A)$ denote the ideal class of the fractional ideal $\mZ x_1+\dots +\mZ x_{p-1}$ of $\mathbb{Q}(\zeta_p)$.
This class is independent of the choice of $x$ since the eigenspaces of $A$ over $\mathbb{Q}(\zeta_p)$ are $1$-dimensional. The latter follows since the Galois group $\Delta=\Gal(\mathbb{Q}(\zeta_p)\colon \mathbb{Q})$ is isomorphic to $\mZ/p^{\times}\cong \mZ/(p-1)$ and all Galois conjugates of $\zeta_p=e^{2\pi i/p}$ are eigenvalues of $A$. In fact, the ideal class $I(A)$ only depends on the conjugacy class of $A$ and by \cite{LatMac}, the map
\[I : GL_{p-1}(\mZ) \bs GL_{p-1}(\mZ)_{1,p} \to \Cl(\mathbb{Q}(\zeta_p))  \]
from conjugacy classes of the order $p$-elements to the ideal class group is a bijection;
see also \cite{SjYang} and \cite{Ash}.
The Galois group $\Delta=\Gal(\mathbb{Q}(\zeta_p)\colon \mathbb{Q})$ acts on the class group  $\Cl(\mathbb{Q}(\zeta_p))$ by the Galois action.
And $\Delta$ acts on $GL_{p-1}(\mZ) \bs GL_{p-1}(\mZ)_{1,p}$ by taking powers of the matrices.
By \cite[Section 1]{Ash}, the isomorphism $I$ is in fact $\Delta$-equivariant with respect to these two actions. This implies that the conjugacy classes of order $p$ subgroups are in bijection with the set of orbits $\Delta \bs \Cl(\mathbb{Q}(\zeta_p))$.
Additionally, it follows as in \cite[Section 1]{Ash} that for any order $p$ element $A$,
there is a short exact sequence 
\[1 \to C\td{A} \to N\td{A} \to S_{I(A)} \to 1,\]
where $S_{I(A)}  \leq \Delta$ is the stabilizer of $I(A)$ (which agrees with the stabilizer of the conjugacy class of $A$), the group $N\td{A} $ is the normalizer of the subgroup generated by $A$,
and $C\td{A}$ is the centralizer of $A$. Any tuple $(A_1, \dots,A_n) \in GL_{p-1}(\mZ)_{n,p}$ with at least one non-trivial element is contained in the unique subgroup of order $p$. This implies that 
\begin{align} \label{eq:conj class} \vert  GL_{p-1}(\mZ) \bs GL_{p-1}(\mZ)_{n,p}\vert-1=\sum_{(H)} \frac{\vert H \vert^n -1}{\vert N(H)/C(H) \vert},  \end{align}
where $H$ runs over conjugacy classes of subgroups of order $p$, and $N(H)$ is the normalizer of $H$ and $C(H)$ is the centralizer of $H$. This and the above discussion imply that
\[\vert  GL_{p-1}(\mZ) \bs GL_{p-1}(\mZ)_{n,p}\vert-1=\sum_{[I] \in \Delta \bs \Cl(\mathbb{Q}(\zeta_p)) } \frac{p^n-1}{\vert S_I \vert},\]
where $S_I \leq  \Delta$ is the stabilizer of $I \in \Cl(\mathbb{Q}(\zeta_p))$. Let $O_I$, denote the orbit of $I$. Then 
\begin{align*}
  \vert  GL_{p-1}(\mZ) \bs GL_{p-1}(\mZ)_{n,p}\vert-1
  &=\sum_{[I] \in \Delta \bs \Cl(\mathbb{Q}(\zeta_p))} \tfrac{p^n-1}{\vert S_I \vert \cdot  \vert O_I \vert } \cdot\vert O_I \vert \\&= \sum_{[I] \in \Delta \bs \Cl(\mathbb{Q}(\zeta_p))} \tfrac{p^n-1}{p-1} \cdot\vert O_I \vert\\
  &=\tfrac{p^n-1}{p-1}\sum_{[I] \in \Delta \bs \Cl(\mathbb{Q}(\zeta_p))} \vert O_I \vert = \tfrac{p^n-1}{p-1} \vert \Cl(\mathbb{Q}(\zeta_p)) \vert. \end{align*}
The last identity is the class equation. 

Now as observed in the proof of Proposition \ref{prop:Euler_KN_equals_rational_GL}, for any tuple $(A_1, \dots,A_n) \in GL_{p-1}(\mZ)_{n,p}-\{(1, \dots, 1)\}$, we have an isomorphism 
\[C\td{A_1, \dots,A_n} \cong \mZ/p \times \mZ/2 \times \mZ^{\frac{p-3}{2}}.\]
Hence, one obtains
  \begin{align*}
    H^*(BC\td{A_1, \dots,A_n}; L(E^*))\ 
    &\cong \ H^*((S^1)^{\times \frac{p-3}{2}}; L(E^*)) \\
    &\cong\ (L(E^*) \oplus L(E^*)[1])^{\otimes_{L(E^*)} \frac{p-3}{2}}.
  \end{align*}
This together with the previous paragraph completes the proof. 
\end{proof}

\begin{remark}\label{rk:GL_2(Z) for p=3}
  To compute  $\chi_{K(n)} (BGL_2(\mZ))$ and $E^*(BGL_2(\mZ))$
  at the prime $p=3$,
  one can proceed in much the same way as for $p\geq 5$, with the caveat that
  in this case the centralizers are finite.
  Alternatively, one can also proceed directly
  using the calculations for $SL_2(\mZ)$ at the beginning
  of Subsection \ref{subsec:slOK} and the extension
  \[1 \to SL_2(\mZ) \to GL_2(\mZ) \to \mZ/2 \to 1.\]
  It follows that $\chi_{K(n)} (BGL_2(\mZ))=(3^n+1)/2$
  and $E^*(BGL_2(\mZ))$ is even and free of rank $(3^n+1)/2$ as a graded $E^*$-module. 
\end{remark}

\begin{remark}
  The rational Euler characteristic of $BGL_{p-1}(\mZ)$ has been computed.
  For $p\geq 13$ we have $\chi_{\mathbb{Q}}(BGL_{p-1}(\mZ))=0$ by \cite[Theorem 0.1 (a)]{Hor}.
  Moreover, $\chi_{\mathbb{Q}}(BGL_{4}(\mZ))=\chi_{\mathbb{Q}}(BGL_{6}(\mZ))=\chi_{\mathbb{Q}}(BGL_{10}(\mZ))=1$ by the following explicit results:
  From \cite[pages 16-17]{DEKM} and \cite{LeeSzcz} we know that $GL_4(\mZ)$ is rationally acyclic. Rational cohomology of $GL_6(\mZ)$ is computed in \cite[Theorem 7.3]{EVGS}:
\[H^n(GL_6(\mZ); \mathbb{Q})=\begin{cases} \mathbb{Q}, \;n=0,5,8\\
0, \; \text{otherwise},\\
\end{cases}\]
implying that $\chi_{\mathbb{Q}}(BGL_{6}(\mZ))=1$. Horozov in \cite[Theorem 3.3]{Hor} shows that $\chi_{\mathbb{Q}}(BGL_{10}(\mZ))=1$. To our knowledge $H^n(GL_{p-1}(\mZ); \mathbb{Q})$ is not computed for $p \geq 11$. Hence Proposition \ref{prop:Morava E GL} only gives a full computation when $p=5,7$. 

\end{remark}

\subsection{The special linear group \texorpdfstring{$SL_{p-1}(\mZ)$}{SL(p-1,Z)} for a prime \texorpdfstring{$p\geq 5$}{p>=5}}

The group $GL_{p-1}(\mZ)$ non-canonically splits as a semidirect product
of $SL_{p-1}(\mZ)$  and $\mZ/2$;
one can use this to obtain formulas for the Morava $K$-theory Euler characteristic
and $E^*$-cohomology of $BSL_{p-1}(\mZ)$ 
from those for $BGL_{p-1}(\mZ)$ from Subsection \ref{subsec:GLZ}.
Instead of exploiting the splitting,
we prefer to follow the same arguments as in the previous subsection and show that computations for $SL_{p-1}(\mZ)$ can be independently deduced. We recall that $SL_{p-1}(\mZ)$ has finite model of $\underline{E}G$ as an arithmetic group \cite[Theorem 3.2]{Ji}.

\begin{prop} \label{prop:Euler char SL and Morava E SL} Let $p\geq 5$ be a prime. 
  \begin{enumerate}[\em (i)]
  \item  For any integer $n \geq 0$,
  \[\chi_{K(n)}(BSL_{p-1}(\mZ))=\chi_{\mathbb{Q}}(BSL_{p-1}(\mZ)).\]
  \item The character map gives an isomorphism
\begin{align*} L(E^*)\tensor_{E^*} E^*(BSL_{p-1}(\mZ)) & \cong  \\ H^*(BSL_{p-1}(\mZ); L(E^*)) & \oplus \bigoplus_{ 2 \cdot \frac{p^n-1}{p-1} \vert \Cl(\mathbb{Q}(\zeta_p)) \vert} (L(E^*) \oplus L(E^*)[1])^{\otimes_{L(E^*)} \frac{p-3}{2}},\end{align*}
  \end{enumerate}
\end{prop}
\begin{proof}
  Given an integral ideal $I$ of $\mathbb{Q}(\zeta_p)$
and two bases $(x_1,\dots, x_{p-1})$ and $(y_1, \dots, y_{p-1})$ for $I$ (i.e., $I=\mZ x_1+\dots+\mZ x_{p-1}=\mZ y_1+\dots+\mZ y_{p-1}$), we say that $(x_1,\dots, x_{p-1})$ and $(y_1, \dots, y_{p-1})$ have the same orientation if there exists $C \in SL_{p-1}(\mZ)$ such that 
\[C\begin{pmatrix} x_1 \\ \vdots \\x_{p-1}  \end{pmatrix}=\begin{pmatrix} y_1 \\ \vdots \\y_{p-1}  \end{pmatrix}.\]
Let $\Cl(\mathbb{Q}(\zeta_p))^{+}$ denote the set of oriented ideal classes. An element of $\Cl(\mathbb{Q}(\zeta_p))^{+}$ is an equivalence class of pairs $(I,[x_1,\dots, x_{p-1}])$, where $I$ is an integral ideal and $[x_1,\dots, x_{p-1}]$ an orientation class of a basis. A pair $(I,[x_1,\dots, x_{p-1}])$ is equivalent to $(J,[y_1,\dots, y_{p-1}])$ if and only if there exist $r,s \in \mZ[\zeta_p]\setminus \{0\} $ such that $rI=sJ$ and $(rx_1,\dots, rx_n)$ and $(sy_1, \dots, sy_n)$ have the same orientation. 
The bijection $I : GL_{p-1}(\mZ) \bs GL_{p-1}(\mZ)_{1,p} \to \Cl(\mathbb{Q}(\zeta_p))$ from the proof of Proposition \ref{prop:Morava E GL} induces a well-defined bijection
\[I^{+} : SL_{p-1}(\mZ) \bs SL_{p-1}(\mZ)_{1,p} \to \Cl(\mathbb{Q}(\zeta_p))^{+}\]
sending a class of $A$ to the class of $(\mZ x_1+\dots+\mZ x_{p-1}, [x_1,\dots, x_{p-1}])$, where $(x_1, \dots, x_{p-1})$ is an eigenvector of $A$ with respect to the eigenvalue $\zeta_p$ (recall $I(A)$ is the class of the fractional ideal $\mZ x_1+\dots+\mZ x_{p-1}$). 
Since every fractional ideal has exactly two equivalence classes of bases
it follows that $\vert \Cl(\mathbb{Q}(\zeta_p))^{+} \vert=2\vert \Cl(\mathbb{Q}(\zeta_p)) \vert$.
By \cite[Lemma 4]{Ash} the centralizer of $A$ in $GL_{p-1}(\mZ)$ is generated by
\[ \{-\Id,\;  A,\;  \tfrac{A^k-1}{A-1} \; \vert \; 1 \leq i \leq p-1, \;\; 2 \leq k < p/2 \}.\]
All these matrices are in fact elements of $SL_{p-1}(\mZ)$,
so that the centralizers of $A$ in $GL_{p-1}(\mZ)$ and $SL_{p-1}(\mZ)$ agree.
The rest follows exactly as in Subsection \ref{subsec:GLZ} for $GL_{p-1}(\mZ)$,
using the action of the Galois group $\Delta=\Gal(\mathbb{Q}(\zeta_p)\colon \mathbb{Q})$
on $\Cl(\mathbb{Q}(\zeta_p))^{+}$ and a formula analogous to \eqref{eq:conj class}.
\end{proof}

\begin{remark} By \cite{Hor}, we have $\chi_{\mathbb{Q}}(BSL_{p-1}(\mZ))=0$, for $p\geq 13$. The paper \cite{LeeSzcz} computes the rational cohomology 
\[H^n(SL_4(\mZ); \mathbb{Q})=\begin{cases} \mathbb{Q}, \;n=0,3,\\
0, \; \text{otherwise},\\
\end{cases}\]
implying that $\chi_{\mathbb{Q}}(BSL_{4}(\mZ))=0$. Next, the rational cohomology for $SL_6(\mZ)$ is computed in \cite{EVGS}:
\[H^n(SL_6(\mZ); \mathbb{Q})=\begin{cases} \mathbb{Q} \oplus \mathbb{Q}, \;n=5,\\
\mathbb{Q} , \; n=0,8,9,10,\\
0,\; \text{otherwise}.\\
\end{cases}\]
This implies  $\chi_{\mathbb{Q}}(BSL_{6}(\mZ))=0$.
We are not aware of a calculation of $\chi_{\mathbb{Q}}(BSL_{10}(\mZ))$,
and to our knowledge $H^n(SL_{p-1}(\mZ); \mathbb{Q})$ has not been computed for $p \geq 11$.
Hence Proposition \ref{prop:Euler char SL and Morava E SL}
only gives a full computation of $L(E^*)\tensor_{E^*} E^*(BSL_{p-1}(\mZ))$ when $p=5,7$. 
\end{remark}

\subsection{The symplectic group \texorpdfstring{$Sp_{p-1}(\mZ)$} {Sp(p-1,Z)} for a prime \texorpdfstring{$p\geq 5$}{p>=5}}\label{subsec:SpZ}

In this subsection we apply the main results of this paper to the symplectic group $Sp_{p-1}(\mZ)$. The main references for the conjugacy classification of $p$-subgroups of $Sp_{p-1}(\mZ)$ are \cite{Busch} and \cite{SjYang}. The arguments are very analogous to those in the previous two subsections. Again, $Sp_{p-1}(\mZ)$ has finite model of $\underline{E}G$
as an arithmetic group \cite[Theorem 3.2]{Ji}. 

\begin{prop} \label{prop:Euler char Sp and Morava E Sp} Let $p\geq 5$ be a prime. 
  \begin{enumerate}[\em (i)]
  \item  For any integer $n \geq 0$,
\[\chi_{K(n)}(BSp_{p-1}(\mZ))=\chi_{\mathbb{Q}}(BSp_{p-1}(\mZ))+2^{\frac{p-1}{2}}\cdot h_p^{-} \cdot \tfrac{p^n-1}{p-1} .\]
\item The character map gives an isomorphism
  \[ L(E^*)\tensor_{E^*} E^*(BSp_{p-1}(\mZ)) \cong H^*(BSp_{p-1}(\mZ); L(E^*)) \oplus F, \]
  where $F$ is a free $L(E^*)$-module in even degrees
  of rank $ 2^{\frac{p-1}{2}}\cdot h_p^{-} \cdot\frac{p^n-1}{p-1}$.
\end{enumerate}
Here 
\[h_p^{-}=\frac{\vert \Cl(\mathbb{Q}(\zeta_p) \vert }{\vert \Cl(\mathbb{Q}(\zeta_p+\zeta_p^{-1})\vert}\] 
is the relative class number. 
\end{prop}
\begin{proof}
  Since $Sp_{p-1}(\mZ) \leq SL_{p-1}(\mZ)$, we know that any non-trivial $p$-subgroup
  of $Sp_{p-1}(\mZ)$ is isomorphic to $\mZ/p$. Using a similar construction
  as in the proof of Proposition \ref{prop:Morava E GL},
  the paper \cite{SjYang} shows that the conjugacy classes of the elements of order $p$ in $Sp_{p-1}(\mZ)$ biject with the set of equivalence classes of pairs $(I, a)$, where $I$ is an integral ideal and $I \cdot \overline{I}=(a)$, where $a=\overline{a}$. It follows form \cite[Theorem 3]{SjYang} that the number of such equivalence classes is equal to $2^{\frac{p-1}{2}}h_p^{-}$. By \cite[Section 3.2]{Busch} and \cite{KBro1}, we know that for any element $A$ of order $p$, the centralizer of $A$ is finite and
\[C\td{A} \cong \mZ/p \times \mZ/2.\] 
The Galois group $\Delta=\Gal(\mathbb{Q}(\zeta_p)\colon\mathbb{Q})$
acts on the set of equivalence classes of pairs $(I, a)$ and very similar arguments as in Subsection \ref{subsec:GLZ},
including a formula analogous to \eqref{eq:conj class},
imply that the cardinality of the set of equivalence classes $Sp_{p-1}(\mZ) \bs (Sp_{p-1}(\mZ)_{n,p}-\{(1,\dots,1)\})$ is equal to the number $\frac{p^n-1}{p-1}\cdot 2^{\frac{p-1}{2}}\cdot h_p^{-}$.
\end{proof}

\begin{remark} The paper \cite{LeeBrownst} computes the rational cohomology 
\[H^n(Sp_4(\mZ); \mathbb{Q})=\begin{cases} \mathbb{Q}, \;n=0,2,\\
0, \; \text{otherwise},\\
\end{cases}\]
implying that $\chi_{\mathbb{Q}}(BSp_{4}(\mZ))=2$. Next, the rational cohomology for $Sp_6(\mZ)$ is computed in \cite{Hain}
\[H^n(Sp_6(\mZ); \mathbb{Q})=\begin{cases} \mathbb{Q}, \;n=0,2,4,\\
 \mathbb{Q} \oplus \mathbb{Q}, \; n=6,\\
0,\; \text{otherwise}.\\
\end{cases}\]
This implies  $\chi_{\mathbb{Q}}(BSp_{6}(\mZ))=5$. We are not aware of a calculation
of $H^n(Sp_{p-1}(\mZ); \mathbb{Q})$ for $p \geq 11$.
Hence Proposition \ref{prop:Euler char Sp and Morava E Sp} only gives a full computation of $L(E^*)\tensor_{E^*} E^*(BSL_{p-1}(\mZ))$ when $p=5,7$. 

To our knowledge $\chi_{\mathbb{Q}}(BSp_{2l}(\mZ))$ is only known for $l \leq 9$,
see \cite[Appendix]{Hulek-Tommasi}.
For example when $p=19$, we have $\chi_{\mathbb{Q}}(BSp_{18}(\mZ))=528$. The class number of the cyclotomic field $\mathbb{Q}(\zeta_n)$ is equal to 1 for $n \leq 22$. Hence, by Proposition \ref{prop:Euler char Sp and Morava E Sp}, we get
\[\chi_{K(n)}(BSp_{18}(\mZ))=528+2^9\cdot \tfrac{19^n-1}{18}=\tfrac{256 \cdot 19^n+4496}{9}\]
at the prime $p=19$.
Computations for lower primes are done similarly using a list
for $\chi_{\mathbb{Q}}(BSp_{2l}(\mZ))$ in \cite[Appendix]{Hulek-Tommasi}. 
\end{remark}

\subsection{The mapping class group \texorpdfstring{$\Gamma_{\frac{p-1}{2}}$}{Gamma} for a prime \texorpdfstring{$p\geq 5$}{p>=5}}\label{subsec:gamma}

Let $\Gamma_{\frac{p-1}{2}}$ be the mapping class group
of the closed oriented surface of genus $(p-1)/2$.
The papers \cite{Broughton, JiWol, Mislin} show that mapping class groups
admit a finite model for $\underline{E}G$.
The calculations in this subsection are based on results from Chapter III
of Xia's thesis \cite{Xia1}.

It follows from \cite[Chapter III]{Xia1} that any non-trivial finite $p$-subgroup of $\Gamma_{\frac{p-1}{2}}$ is isomorphic to $\mZ/p$. One considers two cases:

\emph{Case 1. $p=6k-1$:} By \cite[Corollary 3.3.2]{Xia1} the number of conjugacy classes of non-trivial $p$-subgroups is equal to $(p+1)/6$,
for any such subgroup $H$, the normalizer $N(H)$ is finite and we have $N(H)/C(H)=1$.
Hence, using a formula analogous to \eqref{eq:conj class},
we obtain 
\[ \vert \Gamma_{\frac{p-1}{2}} \bs (\Gamma_{\frac{p-1}{2}})_{n,p}\vert = \tfrac{p+1}{6}(p^n-1)+1.\]

\emph{Case 2. $p=6k+1$:} By \cite[Corollary 3.3.2]{Xia1} the number of conjugacy classes of non-trivial $p$-subgroups is equal to $(p+5)/6$.
For any such subgroup $H$, the normalizer $N(H)$ is finite. For one conjugacy class we have $N(H)/C(H)=3$ and $N(H)/C(H)=1$ for all the others. Hence, again using a similar argument as in Subsection \ref{subsec:GLZ}, we obtain 
\[ \vert \Gamma_{\frac{p-1}{2}} \bs (\Gamma_{\frac{p-1}{2}})_{n,p} \vert = (\tfrac{p+5}{6}-1)(p^n-1)+\tfrac{p^n-1}{3}+1=\tfrac{p+1}{6}(p^n-1)+1.\]

Combining these results with Corollary \ref{cor:euler morava quotients} and Corollary \ref{cor:Characters for BG}, we obtain

\begin{prop} \label{prop:Euler char Gamma and Morava E Gamma}
  Let $p\geq 5$ be a prime. 
  \begin{enumerate}[\em (i)]
  \item For any integer $n \geq 0$,
    \[\chi_{K(n)}(B\Gamma_{\frac{p-1}{2}})=\chi_{\mathbb{Q}}(B\Gamma_{\frac{p-1}{2}})+\tfrac{(p^n-1)(p+1)}{6} .\]
  \item  The character map gives an isomorphism
    \[ L(E^*)\tensor_{E^*} E^*(B\Gamma_{\frac{p-1}{2}}) \cong H^*(B\Gamma_{\frac{p-1}{2}}; L(E^*)) \oplus F, \]
  where $F$ is a free $L(E^*)$-module in even degrees of rank $(p^n-1)(p+1)/6$.
\end{enumerate}
\end{prop}

\begin{remark}
  In the case $p=5$, the rational cohomology of $\Gamma_2$ is well known,
  see for example \cite{Kawa}:
\[H^n(\Gamma_2; \mathbb{Q})=\begin{cases} \mathbb{Q}, \;n=0,\\
0, \; \text{otherwise}.\\
\end{cases}\]
Next, the rational cohomology for $\Gamma_3$ is computed in \cite{Looijenga}
\[H^n(\Gamma_3; \mathbb{Q})=\begin{cases} \mathbb{Q}, \;n=0,1,6,\\
0,\; \text{otherwise}.\\
\end{cases}\]
Hence Proposition \ref{prop:Euler char Gamma and Morava E Gamma} gives a full computation of $L(E^*)\tensor_{E^*} E^*(B\Gamma_{\frac{p-1}{2}})$ when $p=5,7$. 

Harer and Zagier in \cite{HZ} give a general formula for $\chi_{\mathbb{Q}}(\Gamma_{g})$ for any genus $g$. For example, when $p=31$, we have  $\chi_{\mathbb{Q}}(\Gamma_{15})=717766$ and hence
\[\chi_{K(n)}(B\Gamma_{15})=717766+\tfrac{16(31^n-1)}{3}=\tfrac{16 \cdot 31^n+2153282}{3}. \]
For lower primes we get similar formulas using the table in \cite[\S 6]{HZ}. 
\end{remark}

\providecommand{\bysame}{\leavevmode\hbox to3em{\hrulefill}\thinspace}

\end{document}